\documentclass[a4paper, 12pt]{amsart}
\usepackage{amsmath}
\usepackage{amssymb, latexsym, amsthm}

\newtheorem{Thm}{Theorem}[section]
\newtheorem{Prop}[Thm]{Proposition}
\newtheorem{Cor}[Thm]{Corollary}
\newtheorem{Lem}[Thm]{Lemma}

\theoremstyle{definition}
\newtheorem*{definition}{Definition}
\newtheorem*{remark}{Remark}

\numberwithin{equation}{section}

\begin{document}

\newcommand{\R}{\mathbb{R}}
\newcommand{\Z}[0]{\mathbb{Z}}
\newcommand{\Q}[0]{\mathbb{Q}}
\newcommand{\F}[0]{\mathbb{F}}
\newcommand{\N}[0]{\mathbb{N}}
\renewcommand{\O}[0]{\mathcal{O}}
\newcommand{\p}[0]{\mathfrak{p}}
\newcommand{\m}[0]{\mathrm{m}}
\newcommand{\Tr}{\mathrm{Tr}}
\newcommand{\Hom}[0]{\mathrm{Hom}}
\newcommand{\Gal}[0]{\mathrm{Gal}}
\newcommand{\Res}[0]{\mathrm{Res}}
\newcommand{\id}{\mathrm{id}}
\newcommand{\mult}{\mathrm{mult}}
\newcommand{\adm}{\mathrm{adm}}
\newcommand{\tr}{\mathrm{tr}}
\newcommand{\pr}{\mathrm{pr}}
\newcommand{\Ker}{\mathrm{Ker}}
\newcommand{\ab}{\mathrm{ab}}
\newcommand{\triv}{\mathrm{triv}}
\newcommand{\alg}{\mathrm{alg}}
\newcommand{\ur}{\mathrm{ur}}
\newcommand{\Coker}{\mathrm{Coker}}
\newcommand{\Aut}{\mathrm{Aut}}
\newcommand{\Ext}{\mathrm{Ext}}
\newcommand{\Iso}{\mathrm{Iso}}
\newcommand{\GL}{\mathrm{GL}}
\newcommand{\Fil}{\mathrm{Fil}}
\newcommand{\an}{\mathrm{an}}
\renewcommand{\c}{\mathcal }
\newcommand{\crys}{\mathrm{crys}}
\newcommand{\st}{\mathrm{st}}
\newcommand{\Map}{\mathrm{Map}}
\newcommand{\Sym}{\mathrm{Sym}}
\newcommand{\Spec}{\mathrm{Spec}}
\newcommand{\Frac}{\mathrm{Frac}}
\newcommand{\LT}{\mathrm{LT}}
\newcommand{\Alg}{\mathrm{Alg}}
\newcommand{\To}{\longrightarrow}
\newcommand{\Aug}{\mathrm{Aug}}
\newcommand{\wt}{\widetilde}
\newcommand{\op}{\mathrm}
\newcommand{\Ad}{\op{Ad}}
\newcommand{\ad}{\op{ad}}
\newcommand{\fr}{\mathfrak}

\title[Ramification via deformations]
{Ramification filtration via deformations, II}
\author{Victor Abrashkin}
\address{Department of Mathematical Sciences, 
Durham University, Science Laboratories, 
South Rd, Durham DH1 3LE, United Kingdom \ \&\ Steklov 
Institute, Gubkina str. 8, 119991, Moscow, Russia
}
%\ead{victor.abrashkin@durham.ac.uk}
%\date{}
%\begin{keyword}{local field, Galois group, ramification filtration}
%\end{keyword}
\email{victor.abrashkin@durham.ac.uk}
\date{21.05.2025}

\begin{abstract} Let $\mathcal K$ be a field of 
formal Laurent series 
with coefficients in a finite field of 
characteristic $p$. For $M\in\mathbb{N}$, let 
$\mathcal G_{<p,M}$ be the maximal quotient of 
the Galois group of $\c K$ 
of period $p^M$ and 
nilpotent class $<p$ and let 
$\{\mathcal G_{<p,M}^{(v)}\}_{v\geqslant 0}$ be 
the filtration by ramification subgroups in upper numbering.  
We use the identification $\mathcal G_{<p,M}=G(\mathcal L)$ of 
nilpotent Artin-Schreier theory: here 
$G(\mathcal L)$ is the group obtained from a 
suitable profinite Lie 
$\mathbb{Z}/p^M$-algebra $\mathcal L$ via the 
Campbell-Hausdorff composition law. 
We develop new techniques to obtain a 
``geometrical'' construction of the ideals 
$\mathcal L^{(v)}$ such that 
$G(\mathcal L^{(v)})=\mathcal G_{<p,M}^{(v)}$. 
Given $v_0\geqslant 1$, we construct a 
decreasing central filtration 
$\mathcal L(w)$, $1\leqslant w\leqslant p$, on $\c L$, an 
epimorphism of 
Lie $\mathbb{Z}/p^M$-algebras 
$\bar{\mathcal V}:
\bar{\mathcal L}^{\dag }\longrightarrow 
\bar{\mathcal L}:=\mathcal L/\mathcal L(p)$, and 
a unipotent action 
$\Omega $ of $\mathbb{Z} /p^M$  
on $\bar{\mathcal L}^{\dag }$, which induces the identity 
action on $\bar{\mathcal L}$. 
Suppose  
$d\Omega =B^{\dag }$, where 
$B^{\dag }\in\operatorname{Diff}\bar{\mathcal L}^{\dag }$, 
and 
 $\bar{\mathcal L}^{\dag [v_0]}$ is the ideal of 
$\bar{\mathcal L}^{\dag }$ generated 
by the elements of $B^{\dag }(\bar{\mathcal L}^{\dag })$. 
Our main result states that the ramification ideal 
$\mathcal L^{(v_0)}$ appears as the preimage of 
the ideal in $\bar{\mathcal L}$ 
generated by $\bar{\mathcal V}B^{\dag }(\bar{\mathcal L}^{\dag [v_0]})$. 
In the last section we apply this to   
the explicit construction of generators of 
$\bar{\mathcal L}^{(v_0)}$. The paper justifies a geometrical origin 
of ramification subgroups of $\Gamma _{\c K}$ 
and can be used for further developing of 
non-abelian local class field theory.
\end{abstract}
\maketitle

\section*{Introduction} \label{S0} \ \

Everywhere in the paper $p$ is a prime number. If $K$ is a field 
we denote by $K_{sep}$ its separable closure and set  
$\Gamma _K=\op{Gal}(K_{sep}/K)$.

\subsection{Motivation} \label{S0.1} 
We say that $K$ is cdvf if it is a complete discrete 
valuation field with finite residue field of characteristic $p$. 
Then the structure of the group 
$\Gamma _K$, cf.\,\cite{JW, Se3}, can be explicitly described. 
As a matter of fact, it depends only on $\op{char}K$ and 
the torsion subgroup of 
$K^*$. 
In particular, it is very weak invariant of $K$. Note that 
in the case of algebraic number fields $E$ the situation is quite 
opposite. Here by the Neukirch-Uchida theorem, the functor 
$E\mapsto \op{Gal}(E_{sep}/E)$ is 
fully faithful \cite{Ne}. In other words, $\Gamma _E$ is an absolute 
invariant of 
the field $E$. 

For a cdvf $K$, consider the decreasing filtration by normal 
ramification subgroups $\{\Gamma _K^{(v)}\}_{v\geqslant 0}$ in $\Gamma _K$, \cite{Se1}. 
This (ramification) filtration allows us to recover arithmetic invariants  
(say, discriminants, differentes, etc) 
of subextensions $L$ of $K$ in $K_{sep}$. The importance of 
ramification filtration 
has been never doubted because this is 
highly non-trivial number theoretic interpretation of the 
fundamental geometric concept of ramification. 
However, during very long time the structure of ramification 
subgroups $\Gamma _K^{(v)}$ was known just modulo the subgroup of 
second commutators of $\Gamma _K$. (In other words, we knew only the images of 
$\Gamma _K^{(v)}$ in the maximal abelian quotient 
$\Gamma _K^{ab}$ of $\Gamma _K$ 
given by local class field theory.) 
Only in \cite{Mo} Mochizuki    
proved that in the mixed characteristic case the group 
$\Gamma _K$ together with additional structure 
given by its ramification filtration behaves 
like a Galois group of algebraic number field. In other words, if 
$K$ is a cdvf of characteristic 0 then 
the functor 
$K\mapsto (\Gamma _K, \{\Gamma _K^{(v)}\}_{v\geqslant 0})$ is fully faithful. 
This gave the 
first formal confirmation of the importance of ramification theory in the 
number theoretic setting. 
As a result, the ramification filtration should be not anymore 
treated just as 
a technical concept,  
it is  a strict invariant of $K$ and, as such, 
requires very careful and 
detailed study. Notice that Mochizuki's proof is a spectacular application of the theory of 
$p$-adic representations; at the same time his proof is highly non-efficient. 

It is well-known  that the group structure of 
$\Gamma _K$ can be easily recovered 
from its maximal 
pro-$p$-quotient $\Gamma _K(p)=\op{Gal} (K(p)/K)$, where 
$K(p)$ is the maximal $p$-extension of $K$ in $K_{sep}$, cf.\,\cite{JW}. 
This group is provided with induced ramification filtration 
$\{\Gamma _K(p)^{(v)}\}_{v\geqslant 0}$. 
In 
\cite{Ab1, Ab2, Ab3} the author 
obtained first explicit description (at substantially non-abelian level) 
of the subgroups $\Gamma _K(p)^{(v)}$ modulo the subgroup 
$C_p(\Gamma _K(p))$ 
of commutators of order $\geqslant p$ in $\Gamma _K(p)$. 
This approach systematically used the equivalence of 
the categories of $p$-groups and 
Lie $\Z _p$-algebras with nilpotent class $<p$. It is based on 
new techniques which allowed us 
to work with ramification properties in terms of the appropriate Lie algebras. These results already found the following applications. 

\subsubsection{Mixed characteristic case}  \label{S011}
In \cite{Ab12, Ab13, Ab14} we used the Fontaine-Wintenberger 
field-of-norms functor \cite{Wtb1} 
to extend the approach from \cite{Ab1, Ab2, Ab3} 
to the case of cdvf's $K$ of mixed characteristic 
under the assumption that $K$ contains non-trivial 
$p$-th roots of unity.  In particular, suppose that  
$K$ contains a primitive $p^M$-th root of unity, $M\geqslant 1$, and let 
$K(s,M)$ be the maximal $p$-extension of $K$ in $K(p)$ with the 
Galois group of period $p^M$ and nilpotent class $s<p$. Then we obtained explicit value for the number $v(s,M)$ (the last jump) such that $\op{Gal}(K(s,M)/K)^{(v)}=e$ iff $v>v(s,M)$.

\subsubsection{Local analogue of the Grothendieck Conjecture} \label{S012} 

In \cite{Ab6, Ab11} we obtained by essentially different method 
quite efficient generalisation of Mochizuki's result \cite{Mo} 
in the context of $p$-extensions for cdvf $K$ of any characteristic. 
Namely, if $K$ is a cdvf then the functor $K\mapsto 
(\Gamma _K(p), \{\Gamma _K(p)^{(v)}\}_{v\geqslant 0})$ is fully faithful. 
If $p\geqslant 3$ the proof of the equicharacteristic case 
in \cite{Ab6} uses explicit description of the ramification subgroups 
$\Gamma _K(p)^{(v)}$ modulo third commutators 
$C_3(\Gamma _K(p))\supset C_p(\Gamma _K(p))$. Then 
the mixed chracteristic case is deduced via 
application of the field-of-norms functor. 
In the case $p=2$ our approach  
does not work because modulo 
$C_p(\Gamma _K(p))=C_2(\Gamma _K(2))$ we are just 
in the abelian situation where the ramification filtration 
carries no essential information. 
However, in \cite{Ab11} we studied the images of ramification filtration 
in the Magnus algebra associated with $\Gamma _K(2)$ modulo 
subgroup $C_3(\Gamma _K(2))$ of third commutators. This new 
techniques allowed us to prove the case $p=2$ in the both 
(equicharacteristic and mixed characteristic) cases. Our method even worked in 
the case of higher local fields, cf.\,\cite{Ab8, Ab10}.
\medskip 

\subsubsection{``Small'' $p$-adic representations} \label{S013}  

Suppose $V$ is a finite $\Z _p$-module and $\pi :\Gamma _K\To GL_{\Z _p}(V)$ 
is a group homomorphism. Then we have a finite field extension 
$K_{\pi }:=K_{sep}^{\Ker \pi }$ of $K$ and arithmetic properties 
of $K_{\pi }/K$ depend on the images 
$\pi (\Gamma _K^{(v)})=\op{Gal}(K_{\pi }/K)^{(v)}\subset GL_{\Z _p}(V)$ 
of subgroups $\Gamma _K^{(v)}$ of $\Gamma _K$.
\medskip 

{\sc Question.} How to describe the subgroups 
$\pi (\Gamma _K^{(v)})\subset GL_{\Z _p}(V)$? 
\medskip

In \cite{Ab16} we introduced the concept of ``small'' $p$-adic 
representation. By definition, $\pi $ is {\sc small} if for the wild inertia subgroup 
$\tilde{I}_K$ of $\Gamma _K$ it holds 
$\pi (\tilde{I}_K)=\exp (L)$, where $L$ is a Lie subalgebra in 
$\op{End}_{\Z _p} V$ such that $L^p=0$. We proved that if  
$\op{char}K=p$  and 
$\pi $ is {\sc small} then all images $\pi (\Gamma _K^{(v)})\subset GL_{\Z _p}(V)$ can be 
uniquely recovered from explicitly constructed 
differential form on Fontaine's 
$\varphi $-module $\c M$ 
associated with $\pi $. The construction of such differential form 
is purely geometric. It is related to the canonical connection 
$\nabla $ on $\c M$ and its preimages $\varphi ^{*N}(\nabla)$ on 
$\c M\otimes _K K_{rad}$, 
where $K_{rad}$ is the radical closure of $K$ in $K_{sep}$. In particular, this construction 
gives quite unexpected relation of the 
number-theoretic ramification theory 
to a very popular theory of $\c D$-modules, lifts of Frobenius, etc.  

\subsubsection {Non-abelian explicit reciprocity laws (ERL)} \label{S014}

Suppose that $V$ is a finite cyclic $\Z _p$-module. 
If $V=(\Z /p^{M+1})v^*$ then 
$\op{Gal}(K_{\pi}/K)\subset (\Z /p^{M+1})^{\times }$ and 
$\pi $ factors as 
$$\Gamma _K\To \Gamma _K^{ab}\To 
\op{Gal}(K_{\pi }/K)\subset (\Z /p)^{\times }\times 
\Z /p^{M}\, .$$
Since $K_{\pi }/K$ is abelian all $\op{Gal}(K_{\pi }/K)^{(v)}$ come from 
$\Gamma _K^{ab(v)}$ which are well-known by local 
class field theory. Then the corresponding 
images of $\Gamma _K^{(v)}$ can be found via one-or-another 
ERL for the Hilbert symbol. 
For example, suppose   
$K=k((t_0))$ contains a 
primitive $p^{M}$-th root of unity and $[K_{\pi }:K]=p^M$ 
then $K_{\pi }=K(\root {p^{M}}\of {v_{\pi }})$, 
$v_{\pi }\in K^*$. 
So the images $\pi (\Gamma _K^{(v)})$ can be calculated via calculations  
with the $M$-th Hilbert symbols $(u,v_{\pi})$,  
where $u\equiv 1\,\op{mod}\,t_0^{v}$.  
The study of different ERL's for the Hilbert symbol (and its analogs) 
is one of most popular  
areas of Algebraic Number Theory. It has very rich history \cite{Le} 
started with Kummer's work in 1859. All known ERL's have 
more-or-less similar structure: the answer appears 
as a residue of a suitable explicitly constructed differential form. 
As example, we can refer to  
fundamental contributions of  
Shafarevich \cite{Sh}, Artin-Hasse \cite{AH}, 
Iwasawa \cite{Iw}, Br\" uckner-Vostokov \cite{Br, Vo}, 
Witt \cite{Fo}, Coates-Wiles \cite{CW}, et al.  
If $V$ from above Subsection 
\ref{S013} is not cyclic then one-or-another 
non-abelian version of the Hilbert 
symbol and the corresponding ERL's could be used to 
study a similar question in full generality. 
Such symbol should use as essential ingredients nilpotent versions of 
local class field theory and appropriate versions of  
Kummer or Artin-Schreier theory.  
The idea of creating nilpotent class field theory was very popular 
for a very long time, e.g.\,\cite{Sh, KdS, KKL}. Recently, 
F.\,Laubie \cite{LF1}  
obtained such theory by 
generalizing the approach from \cite{KdS}. It is based on 
the idea of iteration of Lubin-Tate theories via the 
field-of-norms functor. This theory posseses 
many nice properties but is not adjusted for obtaining 
(even very well-known) explicit information about  
involved Galois groups. A nilpotent analog of the Artin-Schreier theory was constructed 
by the author in \cite{Ab1, Ab2}; it allowed us to obtain 
new results about ramification filtration in \cite{Ab1, Ab3}. Introduction of 
nilpotent analogue of Hilbert symbol could be stated now as the question about linking  
the Laubie 
and our theories. This will allow us to benefit from 
functorial properties of the first  
and explicit constructions of the second theories. 
Notice that the case of ``small'' representations \cite{Ab16} 
gives a positive feeling about existence of ERL's in the 
non-abelian context. 
%However, we can still ask question from C about  
%subgroups $\pi (\Gamma _K^{(v)})\subset GL (V)$. 
%This could be considered as a weaker version of the problem of 
%existence of non-abelian ERL's. 
%The answer to this question for ``small'' $\pi $ from \cite{Ab16} 
%matches with abelian ERL's -- it is also stated in terms of differential forms. 

\subsubsection{Geometrical approach to ramification filtration} \label{S015}
The above 
applications not only illustrate a fundamental role of explicit results  
from \cite{Ab1, Ab2, Ab3} but also rise the question about their   
geometric origin and more structural approach.  
The proofs in \cite{Ab1} are reasonably technical but their 
modulo $p^M$-version in \cite{Ab2, Ab3} relies on heavy computations with 
enveloping algebras. 
When working with Lie algebras of nilpotent class $<p$ the most efficient and clear approach should use one-or-another 
deformation theory. Such theory was recently constructed in the case of 
Galois groups of period $p$ in \cite{Ab15} and this gave us a geometrical 
approach to \cite{Ab1}. 
In this paper we construct more general deformation theory which adjusted to 
the modulo $p^M$ situation. This theory  
not just gives the results of \cite{Ab2, Ab3} but also 
considerably simplifies the approach from \cite{Ab15}. 
We expect that this  
new techniques could be applied to develop the modulo $p^M$-case 
of results obtained in the mixed characteristic 
case in \cite{Ab12, Ab13}. 
Note that in \cite{Ab15} we gave also an explicit description of arbitrarily long  
segments of ramification filtration (not just of 
a single ramification subgroup) in the case of 
Galois groups of period $p$. 
This description admits a straight generalisation to the modulo 
$p^M$-situation. As a result, we have a substantial control 
on explicit information about arithmetical structure of 
Galois groups of cdvf's. We also expect the existence of other 
fundamental relations between Arithmetic and Geometry of cdvf's.

\subsection{The content of the paper} \label{S0.2} 

Everywhere in the paper $M\in\N $ and $v_0\in\R _{>0}$ are fixed. 
For $s\in\N $ and a topological 
group $G$, we denote by $C_s(G)$ the closure of the subgroup of 
$G$ generated by commutators of order $\geqslant s$. 
If $L$ is a toplogical 
Lie algebra then $C_s(L)$ is the closure of the ideal generated by 
commutators of order $\geqslant s$. 
For any topological $A$-modules $M$ and $B$, we use the notation  
$M_B:=M\hat\otimes _AB$. 
\medskip 

Let $\c K$ be a cdvf and  
let $\c K_{<p,M}$ be a maximal $p$-extension of $\c K$ 
with the Galois group 
$\Gal (\c K_{<p,M}/\c K):=\c G_{<p,M}$ 
of nilpotence class $<p$ and exponent $p^M$. 
If $L$ is a Lie $\Z /p^M$-algebra then $G(L)$ is the 
group obtained from the elements of 
$L$ via the Campbell-Hausdorff composition law $\circ $, cf.\,Sect.\,\ref{S1.2}. 

Consider a decreasing filtration by ramification 
subgroups in upper numbering
$\{\c G^{(v)}_{<p,M}\}_{v\geqslant 0}$ of $\c G_{<p,M}$. 
In \cite{Ab15} we constructed a deformation theory which 
allowed us to retrieve (in essentially geometrical way) 
the results from \cite{Ab1}. 
In this paper, we obtain a 
generalisation (and also a considerable 
simplification) 
of our approach from \cite{Ab15} to the case of 
Galois groups $\c G_{<p,M}$ with arbitrary exponent $p^M$. 
We also explain how this theory can be applied to recover the structure of 
the ramification subgroups $\c G_{<p,M}^{(v)}$. 
Our approach could be overviewed as follows. 
\medskip 

a) Let $\c K=k((t_0))$, where $k\simeq\F _{p^{N_0}}$. Then we can use the uniformiser $t_0$ 
to define the lifts modulo 
$p^M$, $O_M(\c K)=W_M(k)((t))\subset W_M(\c K)$ and 
$O_M(\c K_{sep})=W_M(\sigma ^{M-1}\c K_{sep})[t]\subset W_M(\c K_{sep})$. Here $W_M$ is 
the functor of Witt vectors of length $M$ and  
$t=[t_0]\in W_M(\c K)$ is the Teichmuller representative  
of $t_0$. This allows us to apply the nilpotent version of 
Artin-Shreier theory, cf.\,\cite{Ab2},  to fix an 
identification $\eta ^0:\c G_{<p,M}\simeq G(\c L)$, where $\c L$ is 
explicitly constructed  pro-finite 
Lie $\Z /p^M$-algebra. The identification $\eta ^0$ 
appears via choosing 
a suitable $e\in\c L_{O_M(\c K)}$, taking  
$f\in\c L_{O_M(\c K_{sep})}$ such that $\sigma (f)=e\circ f$ and setting 
for any 
$\tau\in\c G_{<p,M}$, $\eta ^0(\tau )=(-f)\circ \tau (f)\in G(\c L)$. 
This provides us also with 
``ramification'' ideals $\c L^{(v)}$ in $\c L$ such that 
$\eta ^0(\c G_{<p,M}^{(v)})=\c L^{(v)}$. 
\medskip 

b) The extension of scalars $\c L_{W(k)}$ 
appears with the set of  explicit 
generators $\{D_{an}\ |\ \op{gcd}(a,p)=1, n\in\Z /N_0\}\cup \{D_0\}$. 
We use them to define a decreasing filtration of ideals 
$\c L=\c L(1)\supset \c L(2)\supset\ldots \supset \c L(p)$ and provide 
$\bar{\c L}=\c L/\c L(p)$ with induced filtration 
$\bar{\c L}(w)$, $1\leqslant w\leqslant p$. 
The algebra $\bar{\c L}$ is considerably smaller than $\c L$;  
it is generated 
over $W(k)$ by all $D_{an}$ with $0\leqslant a<(p-1)v_0$. 
However,  
we prove that $\c L(p)\subset\c L^{(v_0)}$ and, as a result, 
it will be sufficient to 
describe the image $\bar{\c L}^{(v_0)}$ of $\c L^{(v_0)}$ in $\bar{\c L}$. 
\medskip 

c) Our next step is a construction of a new 
Lie $\Z /p^M$-algebra $\bar{\c L}^{\dag }$ 
together with projection $\bar{\c V}:\bar{\c L}^{\dag }\to \bar{\c L}$. As a 
matter of fact, there is a splitting of Lie algebras 
$\bar{\c L}^{\dag }=\Ker\bar{\c V}\oplus \bar{\c L}$. 
The construction of $\bar{\c L}^{\dag }$ depends on a construction of a 
special system of coefficients inside ${O_M(\c K)}$, cf. 
Sect.\,\ref{S2}. 
We just mention that this system specifies explicitly  the 
powers of $t$ which depend on 
a choice of a sufficiently large $N^*\in\N $ and a sufficiently close to 
$v_0$ rational number $r^*<v_0$. More precisely, if $q:=p^{N^*}$ then 
$r^*(q-1):=b^*$ is a natural number prime to $p$, and 
there are no ramification breaks of the ramification 
filtration $\c G^{(v)}_{<p,M}$ 
inside the interval $(r^*,v_0)$. 
(The ramification breaks of the filtration $\c G^{(v)}_{<p,M}$ 
are discrete because 
they are discrete for abelian extensions and 
$\c G_{<p,M}$ has finite nilpotence class.) Then the 
admissible exponents $t^{-\iota }$, $\iota\in\mathfrak{A}^0$, appear 
in the form 
$\iota =qp^{M-1}A-b^*U\in\mathfrak{A}^0$, where $A$ is a finite sum of various 
$p^na$ with $n\geqslant 0$ 
and $a\in [0,(p-1)v_0)$, and $U$ is a finite sum of $u_jp^{m_j}$ with 
$u_j,m_j\geqslant 0$ and $1\leqslant \sum _ju_j<p$. 
One of properties of such exponents $\iota\in\mathfrak{A}^0$ is that 
$\op{ch}(\iota ):=\sum _ju_j$ is an invariant of $\iota $. In these terms, 
working with the Lie algebra 
$\bar{\c L}$ inside $\bar{\c L}^{\dag }$ is related  
to $t$-exponents with $\op{ch}(\iota )=0$,  and working 
with $\Ker\bar{\c V}\subset \bar{\c L}^{\dag }$ corresponds to 
t-exponents with $\op{ch}(\iota )\geqslant 1$.
\medskip 

d) Let us come back to the identification 
$\eta ^0:\c G_{<p,M}\simeq G(\c L)$ from a) 
determined (up to conjugation) 
by a choice of a suitable $e\in\c L_{O_M(\c K)}$. Then 
the image $\bar e$ of $e$ under the projection 
$\c L_{O_M(\c K)}\to\bar{\c L}_{O_M(\c K)}$ determines the epimorphic map 
$\bar{\eta }^0:{\c G}_{<p,M}\to G(\bar{\c L})$. 
On the other hand, choose an analog $\bar e^{\dag }$ of $\bar e$ for $\bar{\c L}^{\dag }$ 
which is compatible with $\bar e$, i.e. 
such that $(\bar{\c V}\otimes\id )\bar e^{\dag }=\bar e$. 
This will allow us 
to obtain a unique lift 
$\bar\eta ^{\dag }:\c G_{<p,M}\to G(\bar{\c L}^{\dag })$ of 
$\bar{\eta }^0$ such that 
$\bar{\c V}\bar{\eta }^{\dag }=\bar{\eta }^0$. 
\medskip 

e) Note that $\bar{\c L}^{\dag }$ also could be 
provided with a decreasing filtration of length 
$p$ such that the projection 
$\bar{\c V}:\bar{\c L}^{\dag }\to \bar{\c L}$ is  
epimorphism of filtered Lie algebras. As a matter 
of fact, the filtration on $\bar{\c L}^{\dag }$ 
is defined in terms of invariant $\op{ch}(\iota )\in [0,p-1]$. 
This allows us to define the 
action $\Omega $ of the elements $\gamma\in\Z /p^M$ 
on the system of coefficients $\mathfrak{A}^0$ 
by the correspondences 
$t^{-qp^{M-1}a}\mapsto t^{-qp^{M-1}a}E(t^{b^*})^{a\gamma p^{M-1}}$, 
where $E$ is the Artin-Hasse exponential.
\medskip 

f) We use the action $\Omega $ to obtain a deformation of  
the element $\bar e^{\dag }$ to $\Omega ({\bar e}^{\dag })
\in\bar{\c L}^{\dag }_{O_M(\c K)}$. This element determines the deformation 
$\Omega (\bar{\eta }^{\dag })$ of $\bar{\eta }^{\dag }$ which 
is given by a unipotent automorphism $A^{\dag }$ of 
the Lie algebra $\bar{\c L}^{\dag }$. Equivalently, 
it appears from a 
differentiation $B^{\dag }$ of $\bar{\c L}^{\dag }$ such that 
$A^{\dag }=\wt{\exp}B^{\dag }$ (this is the truncated exponential). 
With this notation, our main result states that 
the ramification ideal 
$\bar{\c L}^{(v_0)}$ appears as 
$\bar{\c V}B^{\dag }(\Ker{\bar{\c V}})$. 
\medskip 

In the remaining part of the paper we are doing explicit 
calculations with specially chosen element $\bar e$ to 
find explicit generators of $\bar{\c L}^{(v_0)}$ 
in the style of paper 
\cite{Ab3}.
\medskip 

{\sc Acknowledgements.} The author is very grateful to the referee 
for encouraging comments, pointing out a few inexactitudes and suggesting several simplifications.

\section{Lie algebra $\c L$ and its filtration} \label{S1} 

In this section 
$k\simeq\F _{p^{N_0}}$, $W(k)$ is the corresponding ring of 
Witt vectors, $\sigma $ is the morphism of taking $p$-th powers  
in $k$. The corresponding morphism $W(\sigma )$ we will denote for 
simplicity just by $\sigma $. 

\subsection{Lie algebra $\c L$} \label{S1.1} 
Set $\Z ^+(p)=\{a\in\N\ |\ \op{gcd}\,(a,p)=1\}$ and  
$\Z ^0(p)=\Z ^+(p)\cup\{0\}$.  
Let  
$\widehat{\c L}_{W(k)}$ be a 
free Lie $W(k)$-algebra with the 
module of (free) 
generators    
\begin{equation}  \label{E1.1} 
\{D_0\}\cup\{D_{an}\ |\ a\in\Z ^+(p),  
n\in\Z/N_0\}\,.
\end{equation}

Define $\sigma $-action on $\widehat{\c L}_{W(k)}$ by setting  
$\sigma :D_{an}\mapsto D_{a,n+1}$ and $\sigma :D_0\mapsto D_0$. 
Then $\widehat{\c L}_{W(k)}|_{\sigma =\id}:=\widehat{\c L}$ 
is a free Lie $\mathbb{Z} _p$-algebra such that 
$\widehat{\c L}\otimes W(k)=\widehat{\c L}_{W(k)}$. 
We set $\c L:=\widehat{\c L}/C_p(\widehat{\c L})\otimes\Z /p^M$. 
This is the maximal quotient of $\widehat{\c L}$ of 
nilpotent class $<p$ and period $p^M$.

\subsection{Ideals $\c L(w)$, $w\geqslant 1$} \label{S1.2} Recall that we  
fixed $v_0\in\mathbb{R}_{>0}$. 

\begin{definition} For $m\in\Z _{\geqslant 0}$ and $s\in\N$ , let $U(s,m)$ be  
the minimal value of 
$1/p^{i_1}+\ldots +1/p^{i_{s}}$, where all 
$i_j$ run over $\Z _{\geqslant 0}$ and satisfy the condition 
$i_1+\dots +i_{s}=m$. By definition, set for any $m\geqslant 0$, $U(0,m)=0$. 
\end{definition}

Note that directly from definitions it follows that: 

-- $U(s,0)=s$ and for $s\geqslant 1$, $U(s,m)>0$;

-- if $s'>s$ then for any $m\geqslant 1$, $U(s',m)>U(s,m)$;

-- if $m'>m$ and $s\geqslant 1$ then 
$U(s,m')<U(s,m)$.
\medskip

\begin{Prop} \label{P1.1} 
Let $s,s',m,m'\in\Z _{\geqslant 0}$. Then:

{\rm a)} $U(s,m)+U(s',m')\geqslant U(s+s',m+m')\,.$
\medskip 

{\rm b)} For $s\geqslant 1$, let 
$m=su+s_1$, where $0\leqslant s_1<s$. 
Then  
$$U(s,m)=s_0/p^{u}+s_1/p^{u+1}\, \text{with } 
s_0=s-s_1\,.$$
In particular, $u, s_0,s_1$ are uniquely determined by $m$ and $s$. 
\medskip 

{\rm c)} For any $m^0, s, s'\geqslant 0$, 
$$\min\{U(s,m)+U(s',m')\ |\ m+m'=m^0\}=U(s+s',m^0)\,.$$
\end{Prop}

\begin{proof}  
a) is obvious. 

b) Suppose $m_1,\dots ,m_s\in\Z _{\geqslant 0}$, 
$m_1+\ldots +m_s=m$ and 
$$U(s,m)=1/p^{m_1}+\cdots +1/p^{m_s}\, .$$
If for some indices $1\leqslant i,j\leqslant s$ 
it holds that $m_i\geqslant m_j+2$ 
then  we obtain the contradiction by the use of the obvious inequality 
$$1/p^{m_i-1}+1/p^{m_j+1}<1/p^{m_i}+1/p^{m_j}\, .$$ 
Therefore, the minimal value $U(s,m)$ can appear 
only if there is 
$u\in \mathbb{N}$ such that among all numbers $m_i$, 
$1\leqslant j\leqslant s$, there are 
$s_0>0$ of those which are equal to $u$ and all 
remaining $s_1=s-s_0$ 
numbers are equal to $u+1$.  
It remains to note that 
$m=m_1+\dots +m_s=us_0+(u+1)s_1=su+s_1\, .$

c) Suppose $U(s,m)+U(s',m')$ is the minimal value of 
the LHS. We can assume that 
$s,s'\geqslant 1$. Let $m=si+s_1$ and $m'=s'i'+s_1'$, 
where $0\leqslant s_1<s$ and 
$0\leqslant s_1'<s'$. Then for $s_0=s-s_1$ and $s_0'=s'-s_1'$, 
$U(s,m)=s_0/p^i+s_1/p^{i+1}$ and 
$U(s',m')=s_0'/p^{i'}+s_1'/p^{i'+1}$. 

$c_1$) Suppose $i=i'$. Then $U(s,m)+U(s',m')=U(s+s',m+m')$ 
and our statement is proved. 

$c_2$) Suppose $i>i'$.

$c_{21}$) Let $s_1\ne 0$ and $s_1'\ne s'-1$. In this case 
$m-1=si+(s_1-1)$ 
and $m'=s'i'+(s_1'+1)$ imply that 
\begin{equation} \label{E1.2}
U(s,m-1)+U(s',m'+1)<U(s,m)+U(s',m')
\end{equation}
because $U(s,m-1)=U(s,m)+(1/p^i-1/p^{i+1})$ and 
$U(s',m'+1)=U(s',m')-(1/p^{i'}-1/p^{i'+1})$.
\medskip 

$c_{22}$) Let $s_1=0$ but $s_1'\ne s'-1$. In this case 
$m-1=s(i-1)+(s-1)$ and 
$m+1=s'i'+(s'_1+1)$ imply that either we have again 
the contradiction \eqref{E1.2} 
or $i=i'+1$. But in the later case, 
$m+m'=(s+s')i'+s+s_1'$ and 
$$U(s+s', m+m')=\frac{s'}{p^{i'}}+\frac{s+s_1'}{p^{i'+1}}=
U(s,m)+U(s',m')\, .$$

$c_{23}$) Let $s_1\ne 0$ but $s_1'-1=s'$. Here $m-1=si+s_1-1$ 
and $m+1=(s'+1)i'$. Then 
$U(s,m-1)=U(s,m)+(1/p^i-1/p^{i+1})$ and 
$U(s',m'+1)=U(s',m')-(1/p^{i'}-1/p^{i'+1})$.
Then again we obtain contradiction \eqref{E1.2} because $i>i'$. 

$c_{24}$) Let $s_1=0$ and $s_1'=s'-1$. Here 
$m-1=s(i-1)+s-1$ and $m'+1=s'(i'+1)$. 
We obtain either contradiction \eqref{E1.2} or $i-1=i'$. 
In the later case 
$m+m'=s'i'+s+s'-1$ and $U(s+s',m+m')=U(s,m)+U(s',m')$.
\end{proof}

\begin{definition} Let  
$\{\c L(w)\}_{w\in\N }$ be the minimal (decreasing) 
filtration of ideals in $\c L$ such that: 

1) $\c L(1)=\c L$ and 
$\forall w_1,w_2$, 
$[\c L(w_1), \c L(w_2)]\subset \c L(w_1+w_2)$;

2) if  $a/v_0\geqslant U(w-1,m)$ then 
$p^mD_{an}\in\c L(w)_{W(k)}$. 
\end{definition} 

\begin{remark} Let $\hat{\c A}_{W(k)}$ be the enveloping algebra of 
$\hat{\c L}_{W(k)}$. This is free associative $W(k)$-algebra. 
It has (as a $W(k)$-module) a $W(k)$-basis consisted of 
all finite products of 
elements from the set \eqref{E1.1}. We can introduce a minimal 
decreasing 
filtration of ideals $\{\hat{J}(w)\}_{w\in\N }$ in 
$\hat{\c A}_{W(k)}$ such that

1) $\hat J(1)$ is the augmentation ideal of 
$\hat{\c A}_{W(k)}$ and for all $w_1,w_2$, 
$$\hat J(w_1)\hat J(w_2)\subset \hat J(w_1+w_2);$$
2) for all $D_{an}$ and $m\geqslant 0$, 
if $p^ma\geqslant v_0U(w-1,m)$ then 
$p^mD_{an}
\in\hat J(w)\,.$ 

It can be easily verified that we have the induced filtration 
on the Lie algebra $\hat{\c L}$ which induces our filtration 
$\{\c L(w)\}_{w\in\N }$ on $\c L$. 
 
\end{remark}

{\sc Example.} Suppose $m=0$ and $s\geqslant 1$. 
Then $u_0=s_1=0$ and for all $w\geqslant 1$, 
$U(w-1,0)=w-1$. 
If $(w-1)v_0\leqslant a< wv_0$ then 
$D_{an}\in\c{L}(w)_{W(k)}$. 
In particular, if $a\geqslant (p-1)v_0$ then 
$D_{an}\in\c L(p)_{W(k)}$. 
In the modulo $p$ situation 
from \cite{Ab15}   
the ideals $\c L(w)$ coincide with the ideals 
introduced in that paper (and 
denoted by $\c L(s)$). 
\medskip

Part b) of above proposition \ref{P1.1} implies the following property.

\begin{Prop} \label{P1.2}
In the above definition of the filtration $\{\c L(w)\}_{w\in\N }$ 
property 2) can be replaced by the following one:

$2')$  
for all $a\in\Z ^0(p)$ and  
integers $\alpha _0, \alpha _1,j\geqslant 0$, it holds 
\medskip 

$\bullet $\ $R_a(\alpha _0,\alpha _1,j)$:
  if 
$a/v_0\geqslant \alpha _0/p^{j}+\alpha _1/p^{j+1}$ then 
$$p^{\alpha _0j+\alpha _1 (j+1)}D_{an}
\in\c L(1+\alpha _0+\alpha _1)_{W(k)}\, .$$
\end{Prop}  

\begin{remark} 
 For $j\geqslant 1$, the relations $R_a(\alpha _0,0,j)$ 
 and $R_a(0,\alpha _0,j-1)$ coincide. 
\end{remark}

\subsection{Relations 
$p^{M(a,s)}D_{an}\in\c L(1+s)_{W(k)}$, $s<p$}
\label{S1.5}
\begin{definition} 
Given $a\in\Z ^0(p)$ and $0\leqslant s<p$, let $M(a,s)$ 
be the minimal non-negative integer such that 
$p^{M(a,s)}D_{an}\in\c L(1+s)_{W(k)}$. 
\end{definition}

\begin{remark} The restriction 
$s<p$ is introduced because we are going to work 
below with the Lie 
 algebra $\bar{\c L}=\c L/\c L(p)$ provided with 
 the filtration by 
 the ideals $\bar{\c{L}}(w)=\c L(w)/\c L(p)$. 
 \end{remark}
 \medskip

Clearly, $M(a,0)=0$ and for all $m$, 
$p^mD_{0}\notin\c L(2)_{W(k)}$. 
In addition, if $a/v_0\geqslant U(p-1,0)=p-1$ 
then $D_{an}\in\c L(p)_{W(k)}$.
Therefore,  planning to work later modulo $\c L(p)_{W(k)}$ 
we can assume below that $s\geqslant 1$ and 
$1\leqslant a< (p-1)v_0$. 
\medskip 

Introduce the integers $u(a)\geqslant 0$, $1\leqslant 
\kappa _0(a)<p$ and $0\leqslant \kappa _1(a)<p$ such that 
$$\kappa _0/p^{u}+
\kappa _1/p^{u+1}\leqslant 
a/v_0<\kappa _0/p^{u}+(\kappa _1+1)/p^{u+1}\, ,$$
where we set $u=u(a)$, $\kappa _0=\kappa _0(a)$ and 
$\kappa _1=\kappa _1(a)$ to simplify the notation. These 
integers are uniquely determined by $a<(p-1)v_0$.

Now relation $R_a(\alpha _0,\alpha _1 ,j)$ 
from Prop.\ref{P1.2} means that   
\begin{equation}\label{E1.4} 
\alpha _0/p^{j}+\alpha _1 /p^{j+1}
\leqslant\kappa _0/p^{u}+
\kappa _1/p^{u+1}\,,
\end{equation} 
implies  
$p^{\alpha _0j+\alpha _1(j+1)}D_{an}
\in\c L(1+\alpha _0+\alpha _1)_{W(k)}\,.$ 
\medskip 

Note that $\alpha _0+\alpha _1\geqslant p-1$ implies 
$p^{\alpha _0j+
\alpha _1(j+1)}D_{an}\in \c L(p)_{W(k)}$. 
If we want to obtain basic relations for the ideals 
$\c L(s)$ modulo $\c L(p)$ then 
it makes sense to restrict ourselves with the cases where 
$\alpha _0+\alpha _1 <p$. In particular, 
$\alpha _0/p^j+\alpha _1 /p^{j+1}$ will appear as a 
standard decomposition 
in base $p$ with $p$-digits $\alpha _0,\alpha _1 $. 
Now inequality \eqref{E1.4} 
is equivalent to one of the following options:
 \medskip 

-- $j\geqslant u+1$; 

-- $j=u$ and $\alpha _0<\kappa_0$; 

-- $j=u$, $\alpha _0=\kappa _0$ and 
$\alpha _1\leqslant \kappa _1$. 
\medskip 

Proceeding along these lines introduce for 
a given $a\in\Z ^+(p)$, 
the following sets 
of ``basic''  relations 
$R_a(\beta _0, \beta _1, u)$ (where $u=u(a)$) 
which provide the condition 
$p^{M(a,s)}D_{an}\in\c L(1+s)_{W(k)}$: 
\medskip 

1)\  $\c R_a^I=\{R_a(\beta _0, 0, u)\ |\ 1\leqslant\beta _0
\leqslant\kappa _0\}$; here $s=\beta _0$, $M(a,s)=su$;
\medskip 

2)\ $\c R_a^{II}=\{R_a(\kappa _0, \beta _1,u)\ |\ 
1\leqslant\beta _1\leqslant\kappa _1, 
\kappa _0+\beta _1\leqslant p-1\}$; 
here 
\newline 
$s=\kappa _0+\beta _1$,  
$M(a,s)=M(a,\kappa _0)+\beta _1(u+1)$;  
\medskip 

3)\ 
$\c R^{III}_a=\{R_a(\kappa _0-1, \kappa _1+2,u)
\ |\ \kappa _0+\kappa _1\leqslant p-2\}$; here 
$s=\kappa _0+\kappa _1+1$ and  
$M(a,s)=M(a,\kappa _0+\kappa _1)+u+2$;
\medskip 

4)\ $\c R^{IV}_a=\{R_a(\kappa _0-1, \kappa _1+2+t,u)\ |\ 
t\geqslant 1, \kappa _0+\kappa _1+1+t\leqslant p-1\}$; here 
$s=\kappa _0+\kappa _1+1+t$ and 
$M(a,s)=M(a,\kappa _0+\kappa _1+1)+t(u+1)$. 
\medskip

\begin{Prop} \label{P1.3} 
Given $a\in\Z ^+(p)$ the above listed 
relations imply all 
relations $R_a(\alpha _0,\alpha _1,j)$ modulo $\c L(p)_{W(k)}$. 
\end{Prop} 

 \begin{proof} 
 Consider the relation $R_a(\alpha _0, \alpha ,j)$, i.e. 
 it holds 
 $a/v_0\geqslant \alpha _0/p^{j}+\alpha _1/p^{j+1}$ with 
 integers $j\geqslant 0$, 
 $\alpha _0, \alpha _1\geqslant 0$. 
 Note that $\alpha _0$ and $\alpha _1$ 
 are not necessarily $p$-digits.
 Prove that the corresponding condition 
 $p^{m}D_{an}\in\c L(1+s)_{W(k)}$, where 
 $m=\alpha _0j+\alpha _1(j+1)$ 
 and $s=\min\{\alpha _0+\alpha _1,p-1\}$, 
 is implied by one of the relations $R_a(\beta _0, \beta _1,u)$ 
 from the sets in above 1)-4).  

 Otherwise, there is at least one relation 
 $R_a(\beta _0,\beta _1,u)$ listed in 1)-4),  
 such that the 
 following two conditions hold 
 \begin{equation} \label{E1.5} 
\beta _0+\beta _1\leqslant \min\{\alpha _0+\alpha _1, p-1\}
\end{equation}  
\begin{equation} \label{E1.6} 
\beta _0u+\beta _1(u+1)> \alpha _0j+\alpha _1(j+1)\,.
\end{equation}

Then 
$\alpha_0\geqslant \alpha _0-\beta _0\geqslant 
(\beta _0+\beta _1-\alpha_0-\alpha _1)(u+1)+
(\alpha _0-\beta _0)=
\beta _0u+\beta _1(u+1)-(\alpha _0u+\alpha _1(u+1))
>(\alpha _0+\alpha _1)(j-u)$
implies that $j\leqslant u$. 

$\bullet $\ {\it Suppose $\alpha _0\ne 0$}. 

If 
$j<u$ then $p/p^{u}>a/v_0\geqslant 
\alpha _0/p^{j}$ implies $p>\alpha _0p^{u-j}\geqslant p$. 
The contradiction. 
Therefore, $j=u$. Then \eqref{E1.6} implies    
$\beta _1>\alpha _1$ because 
$$0\geqslant (\beta _0+\beta _1-\alpha _0-
\alpha _1)u>\alpha _1-\beta  _1\, .$$

In particular, $\beta _1\ne 0$ and 
$R_a(\beta _0,\beta _1,u)\notin \c R_a^{I}$.  
Now condition \eqref{E1.5} implies 
that $\beta _0<\alpha _0$.
This excludes the set $\c R_a^{II}$ 
because $\alpha _0\leqslant\kappa _0$ (use that $j=u$). 

It remains to consider the sets $\c R_a^{III}$ 
and $\c R_a^{IV}$,
where $\beta _0=\kappa _0-1$ and 
$\alpha _0=\kappa _0$. 
Here $\beta _1\geqslant\kappa _1+2$ and 
$\alpha _1\leqslant\kappa _1$.
This gives the contradiction 
$\beta _0+\beta _1\geqslant\kappa _0+
\kappa _1+1>\alpha _0+\alpha _1$. 
\medskip 

$\bullet $\ {\it Suppose $\alpha _0=0$ and $\alpha _1\ne 0$.} 

Proceed similarly to the above case where $\alpha _0\ne 0$. 
Here $0>\alpha _1(j-u)$ implies $j+1\leqslant u$. 
If $j+1<u$ we obtain the following contradiction 
$p/p^u>a/v_0\geqslant \alpha _1/p^{j+1}$, i.e. 
$p>\alpha _1p^{u-j-1}\geqslant p$. So, $u=j+1$.

Now $a/v_0\geqslant \alpha _1/p^{u}$ implies that 
$\beta _0\geqslant \alpha _1$ and condition 
\ref{E1.5} implies $\beta _1=0$. 
Therefore, $R_a(\alpha _0,\alpha _1,j)=
R_a(\alpha _1, 0, u)\in\c R_a^{I}$.
\end{proof} 
\begin{remark} 
 Given $a\in\Z ^+(p)$ the relations from 1)-4)   
 are independent: 
 
 a) if $u(a)=0$ then 
 relations $R_a(\beta _0, 0,0)\in\c R_a^{I}$ 
 could be replaced by 
 the relation $R_a(\kappa _0, 0,0)$. 
 This relation together with 
 remaining ones $R_a(\beta _0, \beta _1, 0)\in\c R_a^{I}$ are independent 
 because the appropriate values $M(a,s)$ 
 strictly increase in $s$.
 
 b) if $u(a)\ne 0$ then (for a fixed $a$) all $M(a,s)$ strictly increase 
 together with $s$ and, therefore, are independent. 
\end{remark}

\subsection{Explicit description of ideals 
$\c L(w)_{W(k)}$} \label{S1.4}

Recall that the ideal $\c L(w)_{W(k)}$ is generated 
as $W(k)$-module  
by all commutators 
$$p^m[\dots [D_{a_1n_1},D_{a_2n_2}],\dots ,
D_{a_rn_r}]\, :=p^{m}[D_{\bar a\bar n}]$$   
where $\bar a=(a_1,\ldots ,a_r)$, 
for $1\leqslant i\leqslant r:=r(\bar a)$,  
there are $m_i, s_i\geqslant 0$   
such that $m=\sum _im_i$, \ 
$r+\sum _is_i \geqslant w$ and 
all $a_i\geqslant U(s_i,m_i)v_0$. 
\medskip

\begin{definition}
 Set 
 $M(\bar a, s)=
 \min\{m\in\Z _{\geqslant 0}\ |\ p^m[D_{\bar a\bar n}]
 \in\c L(r(\bar a)+s)_{W(k)}\}$.
\end{definition}

We always assume that $r=r(\bar a)<p$ because 
$D_{\bar a\bar n}\in\c L(r)_{W(k)}$.

If $\bar a=\bar 0$ then, clearly, $M(\bar a,0)=0$ but $M(\bar 0, s)$ 
are not defined for $s\geqslant 1$. 

Assume that $\bar a\ne\bar 0$. 

Let $\bar u=\bar u(\bar a)$ be the minimal integer such that 
for all $1\leqslant i\leqslant r$, there are $p$-digits 
$\theta _i, \theta _i'\in [0,p-1]$ such that 
$$\theta _i/p^{\bar u}+\theta _i'/p^{\bar u+1}\leqslant a_i/v_0
<\theta _i/p^{\bar u}+(\theta _i'+1)/p^{\bar u+1}\,,$$ 
Since $\bar u$ is minimal, at least one $\theta _i\ne 0$. 
\medskip  

Set $\bar{\kappa }_0=\kappa _0(\bar a)=
\sum _{1\leqslant i\leqslant r}\theta _i$ and 
$\bar{\kappa}_1=\kappa _1(\bar a)=
\sum _{1\leqslant i\leqslant r}\theta '_i$. 
\medskip 

\begin{Prop} \label{P1.4}

{\rm a)} $M(\bar a,j)=j\bar u$ for 
$0\leqslant j\leqslant \min \{p-r, \bar\kappa _0\}$;
\medskip 

{\rm b)}  
 $M(\bar a, \bar\kappa _0+j)=\bar u\bar\kappa _0+(\bar u+1)j$, 
 for $1\leqslant j\leqslant 
 \min\{p-(r+\bar\kappa _0), \bar\kappa _1\}$;
 \medskip 
 
{\rm c)} $M(\bar a, \bar\kappa _0+\bar\kappa _1+j)=
 \bar\kappa _0\bar u+\bar\kappa _1(\bar u+1)+1+j(\bar u+1)$, 
 where $j$ runs over 
 $1\leqslant j\leqslant p-(r+\bar\kappa _0+\bar\kappa _1)$.
 \end{Prop}
 
 \begin{proof}
  Note that the initial $[D_{\bar a\bar n}]\in \c L(r+s)_{W(k)}$, 
  where $s=0$. The most ``economical'' way to increase $s$ is 
  to multiply $D_{a_in_i}$ with $\theta _i\ne 0$ by $p^{\bar u}$. 
  We can do this $\bar\kappa _0$ times to use 
  the inequalities 
  $a_i/v_0\geqslant \theta _i/p^{\bar u}$. At 
  this stage we use the 
  basic relations from $\c R_{a_i}^{I}$. 
  If we have not yet reached $\c L(p)_{W(k)}$ 
  we can use $\bar\kappa _1$ 
  multiplications by $p^{\bar u+1}$ to increase 
  $s$ by $\bar\kappa _1$. 
  
  If we still are not in $\c L(p)_{W(k)}$ then 
  choose $a_{i_0}$ with 
  $\theta _{i_0}\ne 0$ and proceed with relations 
  $\c R_{a_{i_0}}
  (\theta _{i_0}-1,\theta '_{i_0}+2+t,\bar u)$. 
  This will require 
  the factor $p^{\bar u+2}$ at the first step and 
  the factors $p^{\bar u+1}$ at all next steps. 
  
  It is easy to verify that the results do 
  not depend on a chosen strategy. 
 \end{proof}

 \begin{remark}
If $\bar u=0$ then $[D_{\bar a\bar n}]
\in\c L(r+\bar\kappa _0)_{W(k)}$. 
 \end{remark}

\subsection{$\Z _p$-submodules  
$\c L_{\c N}(w)\subset \c L(w)$}\label{S1.5} 

\begin{definition}
 For $w\geqslant 1$, let $\c L_{\c N}(w)$ be the  
 closed submodule in $\c L$ such that 
 $\c L_{\c N}(w)_{W(k)}$ is generated by all 
 $p^m[D_{\bar a\bar n}]$ 
 which satisfy the inequality 
 $|\bar a|\geqslant v_0U(w-1,m)$. 
\end{definition}

\begin{Prop} \label{P1.5} We have the following 
decomposition of $\Z _p$-modules 
$$\c L_{\c N}(w)\cap C_2(\c L)\subset 
\sum _{w_1+w_2\geqslant w}
[\c L_{\c N}(w_1), \c L_{\c N}(w_2)]$$ 
\end{Prop}

It will be sufficient to prove the following lemma. 

\begin{Lem} \label{L1.6}
 If $p^m[[D_{\bar a_1\bar n_1}],[D_{\bar a_2\bar n_2}]]
 \in \c L_{\c N}(w)_{W(k)}$ 
 then for 
 $i=1,2$, there are 
 $m_i, w_i$ such that:
 \medskip 
 
 {\rm 1)}\  $m_1+m_2=m$, $w_1+w_2=w$;
 \medskip 
 
 {\rm 2)}\  $p^{m_i}[D_{\bar a_i\bar n_i}]
 \in\c L_{\c N}(w_i)_{W(k)}$. 
\end{Lem}

\begin{proof}[Proof of Lemma]\ \ 
\medskip 

$\bullet $\ {\it Choose $w_1,w_2\geqslant 1$ 
with maximal $w_1+w_2$ s.t. 
for some $m_1,m_2$, } 
$m=m_1+m_2$ {\it and for }$i=1,2$,   
$|\bar a_i|\geqslant v_0U(w_i-1,m_i)\,.$ 
\medskip 

We must prove that $w_1+w_2\geqslant w$.
\medskip 

Suppose $w_1+w_2<w$. 
\medskip 

We can assume that (otherwise, $w_1+w_2$ is not maximal), 
$$v_0U(w_i,m_i)>|\bar a_i|\geqslant v_0U(w_i-1,m_i)\,, i=1,2\, .$$ 

Let $v_0U(w_1+w_2,m)>|\bar a_1|+|\bar a_2|$. 
Then $U(w_1+w_2,m)>U(w-1,m)$ (use that $|\bar a|\geqslant v_0U(w-1,m)$) 
implies 
$w_1+w_2\geqslant w$, q.e.d.
\medskip 

Otherwise, $|\bar a_1|+|\bar a_2|\geqslant v_0U(w_1+w_2,m)$. 
By Prop.\,\ref{P1.1}c),   
 $\exists m_1', m_2'$ such that   \
$m_1'+m_2'=m$ $\&$  
$U(w_1,m_1')+U(w_2,m_2')=U(w_1+w_2,m)$.
\medskip

Let $m_1'\leqslant m_1$ and $m_2'\geqslant m_2$. 
Then $v_0U(w_1,m_1')
\geqslant v_0U(w_1,m_1)>|\bar a_1|$  
and  $v_0U(w_2,m_2')\leqslant v_0U(w_2,m_2)$. 
\medskip 

If $|\bar a_2|<v_0U(w_2,m'_2)$ then 
$v_0U(w_1+w_2,m)>|\bar a_1|+|\bar a_2|$. The contradiction.

If $|\bar a_2|\geqslant v_0U(w_2,m_2')$ then $w_2$ 
can be replaced with a bigger one and $w_1+w_2$ 
is not maximal. The lemma is proved.
\end{proof}

\begin{Cor} \label{C1.7}
 $\c L_{\c N}(w)\subset \c L(w)$. 
\end{Cor}

\begin{proof}
 By induction on $w$ it will be enough to 
 verify the statement for 
 $p^m[D_{\bar a\bar n}]$ with $r(\bar a)=1$. 
 But this follows from 
 the definition of $\c L(w)$. 
\end{proof}

\begin{definition}
 Let $\bar{\c L}=\c L/\c L(p)$ with the induced filtration by the ideals 
 $\bar{\c L}(w)=\c L(w)/\c L(p)$, $w\geqslant 1$. 
\end{definition}

Clearly, $\bar{\c L}(p)=0$ and  
$D_0$ with 
$D_{an}$,  
$a<(p-1)v_0$, generate $\bar{\c L}_{W(k)}$.

%%%%%%%%%%%%%%%%%%%%%%%%%%%%%%%%%%%%%%%%%%%%%%%%%%%%%%Let

\section{Lie algebra $\bar{\c L}^{\,\dag }$ and ideal   
$\bar{\c L}^{[v_0]}\subset\bar{\c L}$} \label{S2}

 \subsection{Parameters $r^*$ and $N^*$} \label{S2.1} \ 
Let $u^*=(p-1)(p-2)+1$ 
and $w^*=(p-1)v_0$.  
\medskip

Introduce 
the following subset in $\Q $\,:
\medskip   

--- $A[u^*]$ is the set of all 
$$a_1p^{n_1}+a_2p^{n_2}+\dots +a_{u^*}p^{n_{u^*}}\, ,$$ 
where 
$n_1,\dots ,n_{u^*}\leqslant 0$, $n_1=0$ and 
 all   
$a_i\in [0,w^*)\cap\Z $.  

The set $A[u^*]$ is ``left-discrete'', i.e. for any 
$\gamma >0$, there is $\delta (\gamma )>0$ such that 
if $\alpha\in A[u^*]$ and $\alpha <\gamma $ then 
$\alpha\leqslant \gamma -\delta (\gamma )$. This property 
was proved in \cite{Ab12}, Sect.\,2.1. Therefore, 
for $1\leqslant s<p$, $0\leqslant m<M$, there is 
a positive infinum $\delta _0(s,m)$ of positive values 
$v_0-\alpha /U(s,m)$, where $\alpha $ runs over 
the set $A[u^*]$. 
\medskip 

{\bf Definition.} $\delta _0:=\min\{\delta _0(s,m)\ |
\ 1\leqslant s<p,\ 0\leqslant m<M\}$. 
\medskip 

For $m\in\Z _{\geqslant 0}$,  
denote by $A[u^*,m]$ the subset 
of $A[u^*]$ consisted of the elements such that all  
$n_u\geqslant -m$. 
\medskip

Let $r^*\in\Q $ be such that 
$r^*=b_0^*/(q^*_0-1)$, where $q^*_0=p^{N^*_0}$ with 
$N^*_0\geqslant 2$, $b^*_0\in\N $  
and $\op{gcd}(b_0^*,\, p(q^*_0-1))=1$. The set of 
such $r^*$ is dense 
in $\mathbb R _{>0}$ 
and we can assume that $r^*\in (v_0-\delta _0, v_0)$. 
\medskip 

Introduce 
the following subset in $\Q $\,:
\medskip   

--- $B[u^*]$ is the set of all numbers 
$$r^*(b_1p^{m_1}+b_2p^{m_2}+\dots +b_{u^*}p^{m_{u^*}})\, ,$$ 
where all $m_i\leqslant M-1$, $m_1\geqslant 0$, 
$b_i\geqslant 0$, $b_1\ne 0$, $b_1+\dots +b_{u^*}<p$. 
(In particular, $0\notin B[u^*]$.) 
For $m\in\Z _{\geqslant 0}$, $B[u^*,m]$ is the subset of $B[u^*]$ 
consisted of elements such that all 
$m_i\geqslant -m$. 
\medskip 

The above sets $A[u^*]$ and an analog of the set 
$B_1[u^*]:=p^{-(M-1)}B[u^*]$ (where we required $m_1=0$ 
instead of $0\leqslant m_1\leqslant M-1$) were introduced in 
\cite{Ab15}. Proceeding similarly 
to \cite{Ab15} 
we obtain that 
$A[u^*]\cap B_1[u^*]=\emptyset $ 
and there is $\rho (u^*)>0$ such that  
$$\rho (A[u^*], B_1[u^*]):=\inf\{|\alpha -\beta |\ |\ \alpha\in A[u^*], 
\beta\in B_1[u^*]\}\geqslant \rho (u^*)\, .$$ 

Choose $N^*\in\N $ satisfying the 
following conditions:
\medskip 

{\bf C1)}\  $N^*\equiv 0\,\op{mod}\,N^*_0$;
\medskip

{\bf C2)}\  $p^{N^*}\rho (A[u^*], B_1[u^*])\geqslant 2r^*(p-1)$;  
\medskip 

{\bf C3)}\  $r^*(1-p^{-N^*})\in (v_0-\delta _0, v_0)$. 
\medskip

Introduce $q=p^{N^*}$ and $b^*=b^*_0(q-1)/(q^*_0-1)\in\N $.
\medskip  

Note that $r^*=b^*/(q-1)$ and $b^*\in\Z ^+(p)$. 
\medskip

\begin{Prop} \label{P2.1} 
 a) If $\alpha\in A[u^*]$ and $\beta \in B[u^*]$ 
then  
$$q\,|qp^{M-1}\alpha -(q-1)\beta |>p^{M-1}b^*(p-1)\, ;$$

b) If $\alpha\in A[u^*]$ and $\alpha /U(s,m)< v_0$ then 
$q\alpha -b^*U(s,m)<0$.
 \end{Prop}

\begin{proof} 
a) Use condition {\bf C2)}. 
Let $\beta _1=\beta /p^{M-1}\in B_1[u^*]$. 
Then \linebreak 
$q\,|q\alpha -(q-1)\beta _1|=q^2
|\alpha -\beta _1+\beta _1/q|
\geqslant q^2|\alpha -\beta _1|-
\beta _1q
\geqslant q^2\rho (A[u^*],\, B_1[u^*])$
\linebreak 
$-r^*(p-1)q
\geqslant 2r^*(p-1)q-r^*(p-1)q=r^*(p-1)q>b^*(p-1)$. 

It remains to multiply this inequality by $p^{M-1}$. 
\medskip 

b) Use condition {\bf C3)}:  
$\alpha /U(s,m)<v_0$ implies that 
$\alpha /U(s,m)\leqslant v_0-\delta _0$ and 
$q\alpha /U(s,m)\leqslant q(v_0-
\delta _0)<qr^*(1-q^{-1})=b^*$.
\end{proof}

 \subsection{The set $\mathfrak{A}^0$} \label{S2.2} 
 Use the above parameters $r^*$, $N^*$, $q=p^{N^*}$.   

\begin{definition}  ${\mathfrak{A}}^0$ is the set of all  
$\iota =p^m(qp^{M-1}\alpha -(q-1)\beta )$,  where   $m\geqslant 0$,   
$\alpha \in A[u^*,m]$, $\beta \in B[u^*,m]\cup\{0\}$ 
and $|\iota |\leqslant  p^{M-1}b^*(p-1)$. 
(Note that $p^m\alpha \in\Z _{\geqslant 0}$ and $p^m\beta /r^*\in\N $.)
\end{definition} 

\begin{remark} 
a) Prop.\ref{P2.1} implies that for all 
$\iota\in\mathfrak{A}^0$, we have $m<N^*$. 

b) 
Any $p^m\alpha $ appears in the form 
$a_1p^{n_1}+\ldots +a_{u^*}p^{n_{u^*}}$, where 
all $a_i\in [0,(p-1)v_0)$ and $n_i\leqslant m$. 

c) 
All $p^m\beta \ne 0$ appear in the form 
$r^*(b_1p^{m_1}+\ldots +b_{u^*}p^{m_{u^*}})$, where 
all $m_i\geqslant 0$ and $1\leqslant b_1+\ldots +b_{u^*}\leqslant p-1$.  
\end{remark}

 Let $\mathfrak{A}^0_0:=\{\iota\in\mathfrak {A}^0\ |\ \beta =0\}$.

\begin{Lem} \label{L2.2} 
Suppose $\iota =
p^m(qp^{M-1}\alpha -(q-1)\beta )\in\mathfrak{A}^0$. Then:
\medskip 

{\rm a)} $\mathfrak{A}^0_0=
\{qp^{M-1}a\ |\ a\in [0,(p-1)v_0)\cap\Z \,\}$; 
\medskip 

{\rm b)}\ if $\beta\ne 0$ then $m<N^*$ 
(in particular, $\mathfrak{A}^0$ is finite);
\medskip 

{\rm c)} the integers $p^m\alpha $ and $p^m\beta /r^*$ 
do not depend on the presentation of 
$\iota $ in the form $p^m(qp^{M-1}\alpha -(q-1)\beta )$ 
from the above definition of $\mathfrak{A}^0$. 
\end{Lem} 

\begin{proof} a) If $\iota\in\frak{A}^0_0$ then 
$\iota =qp^{M-1}(p^m\alpha )\in\frak{A}^0$ 
implies that 
$p^{m}\alpha /(p-1)\leqslant b^*/q=r^*(1-q^{-1})\in (v_0-\delta _0,v_0)\, $ 
and 
%Note that $U(p-1,0)=p-1$ and $\delta _0\leqslant \delta (p-1,0)$, 
%cf.\,Sect.\,\ref{S2.1}. Therefore, the inequalities 
%$p^{m}\alpha /(p-1)<v_0$ and $p^{m}\alpha /(p-1)
%\leqslant v_0-\delta _0$ are equivalent. 
%As a result, 
$\frak{A}^0_0\subset 
\{qp^{M-1}a\ |\ a\in [0,(p-1)v_0)\cap\Z\}\,.$  
The opposite embedding is obvious. 
\medskip

b) If $m\geqslant N^*$ then by Prop.\,\ref{P2.1}, 
 $|\iota |> p^{M-1}b^*(p-1)$ i.e. 
$\iota \notin \mathfrak{A}^0$. 
\medskip

c) 
If $\iota =p^{m'}(qp^{M-1}\alpha '-(q-1)\beta ')$ 
is another presentation 
of $\iota $ then 
$p^m\beta /r^*$ and $p^{m'}\beta '/r^*$ are 
non-negative congruent modulo $qp^{M-1}$ integers  
and the both are smaller than $qp^{M-1}$. Indeed, 
we have  
$\beta /r^*=b_1p^{m_1}+b_2p^{m_2}+\dots +b_{u^*}p^{m_{u^*}}\,,$  
where all $m_i\leqslant M-1$. As a result, 
$$p^m\beta /r^*\leqslant p^{m+M-1}(b_1+\dots +b_{u^*})
\leqslant p^{m+M-1}(p-1)<p^{m+M}\leqslant qp^{M-1}$$
because $m<N^*$. Similarly, $p^{m'}\beta '/r^*<qp^{M-1}$.  
Therefore, 
they coincide and this implies 
also that 
$p^m\alpha =p^{m'}\alpha '$. 
 \end{proof}
 
 \begin{Cor} \label{C2.3}
  Suppose that  $\iota =p^m(qp^{M-1}\alpha -(q-1)\beta )\in\mathfrak{A}^0$.  
  Then the sum of \lq\lq\,$p$-digits\rq\rq\  
  $b_1+\dots +b_{u^*}$ of the appropriate $\beta /r^*
  =b_1p^{m_1}+\dots +b_{u^*}p^{m_{u^*}}$ depends only on $\iota $. 
 \end{Cor}
 
 \begin{definition} 
  $\op{ch}\,\iota :=b_1+\dots +b_{u^*}$\,. 
 \end{definition} 
 \begin{remark}
  If $\op{ch}\,\iota \geqslant 1$ then $\iota\ne 0$. 
  (Use Prop.\,\ref{P2.1}a).)
 \end{remark}

Let  
$\iota =p^m(qp^{M-1}\alpha -(q-1)\beta )\in{\mathfrak{A}}^0$.   
By Lemma\,\ref{L2.2}, $p^m\alpha $ depends only on $\iota $ and 
can be presented (non-uniquely) in the form 
$a_1p^{n_1}+\dots +a_up^{n_u}$ with non-zero 
coefficients 
$a_i\in [1,(p-1)v_0))\cap \Z $, 
$n_i\leqslant m$, $n_1=m$ and $u\leqslant u^*$. 
(We do not count the zero terms.)

\begin{definition} 
 $\kappa (\iota )$ is the maximal natural number 
such that for any above presentation  
 of $p^m\alpha $, $\kappa (\iota )\leqslant u$. 
\end{definition}

\begin{remark}  
a) If $\iota\in\mathfrak{A}^0$ then 
$\kappa (\iota )\leqslant u^*$ and $\op{ch}(\iota )\leqslant p-1$;

b) if $\iota\in\mathfrak{A}^0_0$ then $\op{ch}(\iota )=0$;

c) if $\iota\in\mathfrak{A}^0_0$ and $\iota\ne 0$ 
then $\kappa (\iota )=1$.
\end{remark}
\medskip 

\subsection{Lie algebras ${\c L}^{\,\dag }$ and   
$\bar{\c L}^{\,\dag }$} \label{S2.3}

Let 
$\iota =p^m(qp^{M-1}\alpha -(q-1)\beta )\in\mathfrak{A}^0$ 
be given in notation from Sect.\,\ref{S2.2}. 

%Let $w^0(\iota )$ be the minimal natural number such that 
%$\iota < w^0(\iota )p^{M-1}b^*$ (or $p^{M-1}qv_0$).

\begin{definition} The subset 
$\mathfrak{A}^+(p)$ consists of  $\iota\in\mathfrak{A}^0$ such that 
 
--- $\iota >0$;

--- $\op{gcd}(p^m\alpha \,,p^m\beta /r^*,\,p)=1$;

--- $\kappa (\iota )\leqslant (p-2)\op{ch}(\iota )+1$. 
\end{definition} 

\begin{remark} 
For $\iota\in\mathfrak{A}^+(p)$, 
$(p-2)\op{ch}(\iota )+1\leqslant
(p-2)(p-1)+1=u^*\, .$
\end{remark}

The elements of $\{t^{-\iota }\ |\ \iota\in \mathfrak{A}^+(p)\}$ behave 
\lq\lq well\rq\rq\ modulo 
$(\sigma -\id )\c K$, i.e. the natural map 
$\sum\limits _{\iota\in\mathfrak{A}^+(p)}kt_0^{-\iota }\To \c K/(\sigma -\id )\c K$ 
is injective. This is implied by the following proposition. 

\begin{Prop} \label{P2.4} 
Let $v_p$ be the $p$-adic valuation such that $v_p(p)=1$. 
\medskip 

{\rm a)} 
Then all $\iota p^{-v_p(\iota )}$, where $\iota\in\mathfrak{A}^+(p)$, 
 are pairwise different. 
\medskip 

{\rm b)} If $\iota\in\mathfrak{A}^+(p)$ and $\op{ch}(\iota )=1$ then 
$\iota p^{-v_p(\iota )}\geqslant qv_0-b^*$.
\end{Prop}

\begin{proof} a) Suppose 
$\iota =p^m(qp^{M-1}\alpha -(q-1)\beta )\in \mathfrak{A}^+(p)$. 
\medskip 

If $\op{ch}(\iota )=0$ then  
$\iota\mapsto \iota p^{-v_p(\iota )}$ 
identifies $\{\iota\in\mathfrak{A}^+(p)\ |\ 
\op{ch}(\iota )=0\}$ 
with  
$\Z ^+(p)\cap [0,(p-1)v_0)$, cf.\,Lemma\,\ref{L2.2}a).

If $\op{ch}(\iota )\geqslant 1$ then 
$\iota p^{-m}\notin p^M\N $, i.e. $m+M>v_p(\iota )$. 

Indeed, 
$\iota p^{-m}=qp^{M-1}\alpha -(q-1)\beta \in p^M\N $ implies (use that 
$qp^{M-1}\alpha\in p^M\N $ because $m<N^*$) that 
$\beta /r^*=p^{-m}(b_1p^{m_1}+b_2p^{m_2}+\dots +b_up^{m_u})\in p^M\N $ with 
all $m_i\leqslant m+M-1$. But this number is 
$\leqslant p^{M-1}\op{ch}\,\iota <p^M$. 
The contradiction. 
Then Prop.\,\ref{P2.1}a) implies that 
$$\iota p^{-v_p(\iota )}\geqslant \iota p^{-m-M+1}=
p^{-M+1}|qp^{M-1}\alpha -(q-1)\beta |>$$ 
$$(b^*/q)(p-1)=r^*(1-q^{-1})(p-1)> 
(v_0-\delta _0)(p-1)$$  
(use property {\bf C3} from Sect.\,\ref{S2.1}). 

Finally, 
if $\iota \in\mathfrak{A}^0_0$ then $\iota p^{-v_p(\iota )}=a<
(p-1)v_0$ implies the inequality  $a<(v_0-\delta _0)(p-1)$ 
by the choice of $\delta _0$, cf.\,Sect.\,\ref{S2.1}. 
On the other hand, 
for all $\iota\in \mathfrak{A}^+(p)$ with 
$\op{ch}(\iota )\geqslant 1$, the values  
$\iota p^{-v_p(\iota )}$ are different (use that 
$\op{gcd}(p^m\alpha ,p^m\beta /r^*)\not\equiv 0\,\op{mod}\,p$) 
and are bigger than $(v_0-\delta _0)(p-1)$. 
\medskip 

b) There is $0\leqslant A\leqslant M-1$ such that 
$\iota =p^m(qp^{M-1}\alpha -p^{M-1-A}b^*)$, i.e. 
$\iota p^{-v_p(\iota )}=q(p^A\alpha )-b^*$ and b) holds when 
$p^A\alpha\geqslant v_0$. But if $p^A\alpha =\alpha /U(1,A)<v_0$ then 
Prop.\ref{P2.1}b) implies that 
$q\alpha -b^*U(1,A)=p^{-A}(q(p^A\alpha )-b^*)<0$.  The contradiction. 

The proposition is completely proved.  
\end{proof}

\begin{definition} 
  $\mathfrak{A}^0(p)=\mathfrak{A}^+(p)\cup\{0\}$.
\end{definition} 

Let $\widehat{\c L}_{W(k)}^{\,\dag }$ be the 
Lie $W(k)$-algebra with the set of 
free generators 
$$\{\ D^{\,\dag }_{\iota n}\ |\ \iota\in\mathfrak{A}^+(p), 
n\in\Z /N_0\}\cup \{D^{\,\dag }_0\}\, .$$ 
Use the notation $\sigma $ for 
the $\sigma $-linear automorphism of 
$\widehat{\c L}^{\,\dag }_{W(k)}$ such that  
$\sigma :D^{\,\dag }_{\iota n}\mapsto D^{\,\dag }_{\iota ,n+1}$ and 
$\sigma :D^{\dag }_0\mapsto D^{\dag }_0$. Introduce 
the Lie $\Z _p$-algebras    
$\widehat{\c L}^{\,\dag }:=\widehat{\c L}^{\,\dag }_{W(k)}|_{\sigma =\id }$ and  
$\c L^{\,\dag }=\widehat{\c L}^{\,\dag }/
C_p(\widehat{\c L}^{\,\dag })\otimes \Z /p^M$.   
\medskip

\begin{definition}  $\{\c L^{\,\dag}(w)\}_{w\geqslant 1}$ is the minimal 
filtration of $\c L^{\,\dag }$ such that:

-- for all $w_1,w_2$, $[\c L^{\dag }(w_1),\c L^{\dag }(w_2)]\subset \c L^{\dag }(w_1+w_2)$;

-- $p^mD^{\,\dag }_{\iota n}\in \c L^{\,\dag}(1+s+\op{ch}(\iota ))_{W(k)}$ if 
$\iota \geqslant p^{M-1}b^*U(s,m)$. 
\end{definition} 

\begin{remark} The above condition 
$\iota \geqslant b^*p^{M-1}U(s,m)$ is  
analogous to the 
condition $a\geqslant v_0U(s,m)$ from Section 
\ref{S1.4}, where $v_0$ is replaced by 
$b^*/q$ (which is very close to $v_0$, 
cf.\,condition {\bf C3} from Section \ref{S2.1}. 
\end{remark}

\begin{definition}
 For $\bar\iota =(\iota _1, \dots ,\iota _r)$,  
 set 
 $r(\bar\iota )=r$, 
 $\op{ch}(\bar\iota )=\op{ch}(\iota _1)+\dots +\op{ch}(\iota _r)$ 
 and 
 $|\bar \iota |=\iota _1+\dots +\iota _r$. 
\end{definition}

Proceeding similarly to Sections 
\ref{S1.4}-\ref{S1.5} introduce the submodules 
$\c L^{\dag }(w)_{\c N}$ in $\c L^{\dag}$ generated by 
$p^m[D^{\dag }_{\bar{\iota }\bar n}]$ 
(this is an analogue of the notation $[D_{\bar a\bar n}]$) such that 
either $\op{ch}(\bar\iota )+1\geqslant w$ or 
$|\bar{\iota }|\geqslant p^{M-1}b^*U(w-\op{ch}(\bar\iota )-1,m)$, otherwise. 
Then we can prove an analogue of Corollary 
\ref{C1.7} that for any $1\leqslant w\leqslant p$, 
\begin{equation} \label{E2.1} 
\c L^{\dag }_{\c N}(w)\subset \c L^{\dag }(w)\,. 
\end{equation} 

There is also an analogue of Proposition \ref{P1.4}, 
where the quotients $a_i/v_0$ are replaced by 
$\iota _i/p^{M-1}b^*$ or, equivalently, by $\iota _i/qp^{M-1}v_0$. 
\medskip 

More precisely,
set 
$$M(\bar\iota , s)=
 \min\{m\in\Z _{\geqslant 0}\ |\ p^m[D_{\bar\iota \bar n}]
 \in\c L(r(\bar\iota )+\op{ch}(\bar\iota )+s)_{W(k)}\}\,.$$

We always assume that $r^*=r(\bar\iota )+\op{ch}(\bar\iota )<p$ because 
$D_{\bar \iota \bar n}\in\c L^{\dag }(r^*)_{W(k)}$.

If $\bar a=\bar 0$ then, clearly, $M(\bar a,0)=0$. 

Assume that $\bar a\ne\bar 0$ and $s\geqslant 1$.  

Let $\bar u=\bar u(\bar\iota )$ be the minimal integer such that 
for all $1\leqslant i\leqslant r$, there are $p$-digits 
$\theta _i, \theta _i'\in [0,p-1]$ such that 
$$\theta _i/p^{\bar u}+\theta _i'/p^{\bar u+1}\leqslant \iota _i/qb^*
<\theta _i/p^{\bar u}+(\theta _i'+1)/p^{\bar u+1}\,,$$ 
Since $\bar u$ is minimal, at least one $\theta _i\ne 0$. 
\medskip  

Set $\bar{\kappa }_0=\kappa _0(\bar\iota )=
\sum _{1\leqslant i\leqslant r}\theta _i$ and 
$\bar{\kappa}_1=\kappa _1(\bar\iota )=
\sum _{1\leqslant i\leqslant r}\theta '_i$. 
\medskip 

\begin{Prop} \label{P2.5}

{\rm a)} $M(\bar\iota ,j)=j\bar u$ for 
$0\leqslant j\leqslant \min \{p-r^*, \bar\kappa _0\}$;
\medskip 

{\rm b)}  
 $M(\bar\iota , \bar\kappa _0+j)=\bar u\bar\kappa _0+(\bar u+1)j$, 
 for $1\leqslant j\leqslant 
 \min\{p-(r^*+\bar\kappa _0), \bar\kappa _1\}$;
 \medskip 
 
{\rm c)} $M(\bar\iota , \bar\kappa _0+\bar\kappa _1+j)=
 \bar\kappa _0\bar u+\bar\kappa _1(\bar u+1)+1+j(\bar u+1)$, 
 where $j$ runs over 
 $1\leqslant j\leqslant p-(r^*+\bar\kappa _0+\bar\kappa _1)$.
 \end{Prop}

For $m\in\Z _{\geqslant 0}$, let  
$\c L^{\,\dag }[m]_{W(k)}$ be the $W(k)$-submodule 
of $\c L^{\dag }_{W(k)}$  
generated  
by all monomials 
$[D^{\,\dag }_{\bar\iota\bar n}]$ such that 
$\op{ch}(\iota _1)+\dots +\op{ch}(\iota _r)=m$. 
Denote by $\c L^{\dag }[m]$ the $\Z _p$-submodule of 
$\c L^{\dag }$  such that 
$\c L^{\dag }[m]\otimes W(k)=\c L^{\dag }[m]_{W(k)}$. 
\medskip

We obtained the grading $\c L^{\,\dag }=
\oplus _{m\geqslant 0}\c L^{\,\dag }[m]$ in the category of 
Lie algebras. 
If 
$\c L^{\,\dag }(w)[m]=\c L^{\,\dag }(w)\cap\c L^{\,\dag }[m]$, 
then 
$\c L^{\,\dag }(w)=\oplus_{0\leqslant m\leqslant w-1}\c L^{\,\dag }(w)[m]$. 

\begin{remark}
 The algebra $\c L^{\dag}[0]$ can be naturally identified with 
 a Lie subalgebra of $\c L$ from Section \ref{S1} via 
 $j^0: D^{\dag }_{qp^{M-1}a,n}\mapsto D_{an}$, 
 where $a$ runs over the set $\Z ^0(p)\cap [0,(p-1)v_0)$. 
 Then the special choice of $\delta _0$ and condition {\bf C3} 
 from Section\,\ref{S2.1} 
 imply that this identification is 
 compatible with filtrations, i.e. for any 
 $w\geqslant 1$, $j^0(\c L^{\dag }(w)[0])\subset \c L(w)$. 
\end{remark}

Let $\bar{\c L}^{\,\dag }=\c L^{\,\dag }/\c L^{\dag }(p)$. 
This algebra is provided with the induced 
central filtration 
$\{\bar{\c L}^{\,\dag }(w)\}_{w\geqslant 1}$  
such that $\bar{\c L}^{\,\dag }(p)=0$. 
We also have the induced gradings 
$\bar{\c L}^{\,\dag }=\oplus _{0\leqslant m\leqslant p-2}\bar{\c L}^{\,\dag }[m]$ 
and $\bar{\c L}^{\,\dag }(w)=
\oplus _{0\leqslant m\leqslant w-1}\bar{\c L}^{\,\dag }(w)[m]$, 
where 
$\bar{\c L}^{\,\dag }(w)[m]:=\bar{\c L}^{\,\dag }
(w)\cap \bar{\c L}^{\,\dag }[m]$. Note that $\bar{\c L}^{\dag }[p-1]=0$. 

%\begin{definition}
% If  
%$l\in\bar{\c L}^{\,\dag }[m]_{W(k)}$, $l\ne 0$,  
%we set $\op{ch}(l)=m$.
%\end{definition}

Clearly, for any $m_1,m_2$, 
$[\bar{\c L}^{\,\dag }[m_1],\bar{\c L}^{\,\dag }[m_2]]
\subset \bar{\c L}^{\,\dag }[m_1+m_2]$ and we have 
the induced identification $\bar{\c L}^{\dag }[0]=
\bar{\c L}$ which is compatible with filtrations. 
(Use that all $D_{an}$ with $a\geqslant (p-1)v_0$ 
belong to ${\c L}(p)_{W(k)}$.)

%%%%%%%%%%%%%%%%%%%%%%%%%%%%%%%%%%%%

\subsection{Lie algebras $\c N^{\dag }$, $\bar{\c N}^{\dag }$ 
and $\wt{\c N}^{\dag }$} \label{S2.4} 
Consider $N^*, r^*, q$ from Section\,\ref{S2.1}.

Our approach to the study of the 
ramification filtration in $\c G_{<p,M}$ substantially 
uses the construction of a totally 
ramified separable extension $\c K'$ of $\c K$ such that 
$[\c K':\c K]=q$ and 
the Herbrand function $\varphi _{\c K'/\c K}$ 
has only one edge point $(r^*,r^*)$.  
For a detailed construction of  
$\c K'$ cf.\,e.g.\,\cite{Ab3}, Sect.\,1.5. We just need that 
$\c K'=k((t_{01}))$, where our fixed local parameter $t_0$ of $\c K$ equals 
$t_{01}^qE(t_{01}^{b^*})^{-1}$ and $E=E(X)$ is the 
Artin-Hasse exponential. 

We do not need here a general concept of a lift modulo 
$p^M$ which was used e.g. in \cite{Ab2}. In the case of the field 
$\c K$ its lift $O_M(\c K)$ with respect  
to the parameter $t_0$ appears as 
$W_M(k)((t))\subset W_M(\c K)$, where $t=[t_0]$ is the Teichmuller 
representative of $t_0$. 
We need also a similar lift for $\c K'$ 
related to its local parameter $t_{01}$. So, we set 
$O'_M(\c K')=W_M(k)((t_1))\subset W_M(\c K')$, where $t_1=[t_{01}]$. 
%%%%%%%%%%%%%%%%%%%%%%%%%%%%%%%%%%%%%%%%%%%%%%5
%
%We use the uniformiser $t_{01}$  
%to construct (similarly to Sec.\,\ref{S1.1}) 
%the corresponding lifts $O'_M(\c E')$ of separable extensions $\c E'/\c K'$.  
%Notice that for any separable extension $\c E'$ of $\c K'$, 
%$O'_M(\c E')=W_M(\sigma ^{M-1}\c E')[t_1]\subset W_M(\c E')$, 
%where $t_1=[t_{01}]$ is the Teichmuller representative of 
%$t_{01}$. In particular, $O'_M(\c K')=W_M(k)((t_1))$. 
%
%%%%%%%%%%%%%%%%%%%%%%%%%%%%%%%%%%%%%%%%%%
There is no a natural embedding of 
$O_M(\c K)$ into $O'_M(\c K')$. However, 
we have the embedding  
$$\iota : \sigma ^{M-1}O_M(\c K)=W_M(k)((t^{p^{M-1}}))
\subset  O'_M(\c K')$$  
determined by the ``modulo $p^M$'' lift 
$t^{p^{M-1}}=t_{1}^{qp^{M-1}}E(t_{1}^{b^*})^{-p^{M-1}}\, $  
to $O'_M(\c K')$ 
of the ``modulo $p$'' relation 
$t_0=t_{01}^{q}E(t^{b^*}_{01})^{-1}$.

\begin{definition} 
Let ${\c N}^{\,\dag }$  
be the $W(k)$-submodule   
in ${\c L}^{\,\dag }_{O'_M(\c K')}$ 
generated by the elements  
$t_1^{-b}[D^{\,\dag }_{\bar\iota \bar n}]$, where 
%  $b\in\mathfrak{A}^0$, $\op{ch}(b)\geqslant\op{ch}(\bar\iota )$ and 
$-b\geqslant -|\bar\iota |$.
\end{definition}
\medskip

Clearly, $\c N^{\dag }$ is a   
Lie $W(k)$-subalgebra in $\c L^{\dag }_{O_M'(\c K')}$. 
We can define similarly $\bar{\c N}^{\dag }$ by replacing $\c L^{\dag}$ 
to $\bar{\c L}^{\dag }$. Then a natural projection 
$\c L^{\dag }\to \bar{\c L}^{\dag }$ induces the 
projection $\c N^{\dag }\to \bar{\c N}^{\dag }$. 

\begin{definition} For $s\geqslant 0$,  
 
a) let ${\c N}^{\,\dag }\langle s\rangle $ be the 
$W(k)$-submodule in ${\c N}^{\,\dag }$ 
generated by all  
$p^{A}t_1^{-b}[D^{\,\dag }_{\bar\iota \bar n}]$, where 
$-b\geqslant -|\bar\iota |+p^{M-1}b^*U(s,A)\,;$

b) let ${\c N}^{\,\dag }(s)$ be the $W(k)$-submodule 
in ${\c N}^{\,\dag }$ generated by all ${\c N}^{\,\dag }\langle s_1\rangle [m]$ with 
$s_1+m=s$. 
\end{definition} 

We obtained a decreasing filtration $\{\c N^{\,\dag }(s)\}_{s\geqslant 0}$ 
of $\c N^{\dag }$ 
such that 
$$[{\c N}^{\,\dag }(s_1),{\c N}^{\,\dag }(s_2)]
\subset {\c N}^{\,\dag }(s_1+s_2)\, .$$

\begin{Prop} \label{P2.6} 
For any $s, m$ and $\m'=t_1W_M(k)[[t_1]]$, it holds  
$$\c N^{\dag }\langle s\rangle [m]\subset \c L^{\dag}_{\m '}+
\c L^{\dag }(1+m+s)_{O_M'(\c K')}\,.$$
\end{Prop} 

\begin{proof} If $p^At_1^{-b}[D^{\dag }_{\bar{\iota}\bar n}]
\in\c N^{\dag }\langle s\rangle [m]$ 
then 
$-b\geqslant -|\bar\iota |+p^{M-1}b^*U(s,A)$.  
If $|\bar\iota |\geqslant p^{M-1}b^*U(s,A)$ then 
$p^A[D^{\dag }_{\bar\iota \bar n}]
\in\c L^{\dag }(1+m+s)_{W(k)}\,.$  
Otherwise, $-b\geqslant 1$.  
\end{proof}

 Similar definitions can be done in the context of the Lie algebra 
 $\bar{\c L}^{\dag }$ and the  analog 
 $\bar{\c N}^{\dag }$ of 
 $\c N^{\dag }$ with the corresponding filtration 
 $\bar{\c N}^{\dag }(s)$. The projection 
 $\c N^{\dag }\to\bar{\c N}^{\dag }$ allows us to 
 transform above properties to $\bar{\c N}^{\dag }$. In particular, 
 we have the following 
 
 \begin{Cor} \label{P2.7} 
 a) $\bar{\c N}^{\dag }(p-1)
 \subset \bar{\c L}^{\dag}_{\m '}$;
 
 b) if $p^At^{-b}[D^{\,\dag }_{\bar\iota\bar n}]\in\bar{\c N}^{\,\dag }$ is non-zero then 
 $|\bar\iota |\in\frak{A}^0$;
 
 c) if $p^At^{-b}[D^{\,\dag }_{\bar\iota\bar n}]\notin\bar{\c N}^{\,\dag }(p-1)$ 
 then $-b\leqslant p^{M-1}(p-1)b^*$. 
 \end{Cor}

\begin{definition} Let  $\wt {\c N}^{\,\dag }$ be the submodule of 
$\bar{\c N}^{\,\dag }/\bar{\c N}^{\,\dag }\langle p-1\rangle $ such that 
its $W(k)$-extension of scalars is generated by 
the elements of the form 
$t^{-\jmath }[D^{\dag }_{\bar\iota \bar n}]$, where $\jmath \in\frak{A}^0$. 
 \end{definition} 
 
 Clearly, $\wt{\c N}^{\,\dag }$ is a Lie subalgebra of 
 $\bar{\c N}^{\,\dag }/\bar{\c N}^{\,\dag }(p-1)$ and it has a double grading 
 via submodules 
 $\wt{\c N}^{\,\dag }[m,t]$ generated by the elements of the form 
 $t^{-j}[D^{\,\dag }_{\bar\iota\bar n}]$, where $\op{ch}(j)=t$ and 
 $\op{ch}(\bar\iota )=m$. Here $0\leqslant m\leqslant t\leqslant p-2$. 
 This Lie algebra plays a key role in our further 
 interpretation of the ramification criterion.

\subsection{The action $\Omega $} \label{S2.5}
 Suppose $\gamma\in\Z /p^M$. 

For   
$b =p^m(qp^{M-1}\alpha -(q-1)\beta )\in \mathfrak{A}^0$ 
and $l=p^At_1^{-b}[D^{\,\dag }_{\bar
\iota\bar n}]\in\wt{\c N}^{\,\dag }\,$, 
set  
$$\Omega (l):=E(t_1^{b^*})^{(p^m\alpha )p^{M-1}}l\,,$$ 
where $E(X)$ is the Artin-Hasse exponential. 
We have 
$$E(t_1^{b^*})^{(p^m\alpha )p^{M-1}}\equiv 
\wt{\exp }\left ((p^m\alpha )S_{b^*}\right )\op{mod}\,S_{b^*}^p\,,$$
where $S_{b^*}=t_1^{b^*p^{M-1}}+pt_1^{b^*p^{M-2}}+\ldots +p^{M-1}t_1^{b^*}$ and 
$\wt{\exp}$ is the truncated exponential. 
Therefore, 
$\Omega (l)=\wt{\exp}\left ((p^m\alpha )S_{b^*}\right )l$. 

\begin{remark}
 Note that $S_{b^*}^p\bar{\c N}^{\dag }\subset \bar{\c N}^{\dag }(p)$. So, 
 it is essential that we could define 
 the action $\Omega $ in the context of $\wt{\c N}^{\dag }$ modulo 
 $\bar{\c N}^{\dag }(p-1)$ (which contains $\bar{\c N}^{\,\dag }(p)$). 
 At the same time, $\wt{\c N}^{\dag }$ is 
 still sufficient to recover the Lie algebra $\bar{\c L}^{\dag }$. Indeed, 
 the natural embedding 
 $\bar{\c L}^{\dag }\to \bar{\c N}^{\dag }
 \subset \bar{\c L}^{\dag }_{O'_M(\c K')}$ still induces embedding 
 $\bar{\c L}^{\dag }\to \wt{\c N}^{\dag }$ after taking the projection from 
 $\bar{\c N}^{\dag }$ to $\wt{\c N}^{\dag }$. 
 (Use that $\bar{\c N}^{\dag }(p-1)\subset\bar{\c L}^{\dag }_{\m '}$.)
\end{remark}

By linearity, $\Omega $ (when considered in the context of 
$\wt{\c N}^{\dag }$) can be extended to the 
whole $\wt{\c N}^{\,\dag }$. 
 
The map 
$n\mapsto \Omega _{\gamma }(n):=
\Omega ^{\gamma }(n)$, $\gamma\in\Z /p^M$, 
is a 
well-defined action of the additive group 
$\Z /p^M$ on 
the Lie algebra $\wt{\c N}^{\,\dag }$. This action is 
unipotent: for any $n\in \wt{\c N}^{\,\dag }[m,t]$, 
$\Omega _{\gamma }(n)-n\in \oplus _{t'>t}\wt{\c N}^{\,\dag }[m,t']\,.$ 
\medskip  

Let $\bar{e}^{\,\dag }\in\bar{\c N}^{\dag }$ be such that  
\begin{equation} \label{E2.1} \ \ \bar{e}^{\,\dag }\equiv 
\sum\limits _{\iota\in\mathfrak{A}^0(p)}t_1^{-\iota }D^{\,\dag}_{\iota 0}
\,\op{mod}\,C_2(\bar{\c L}^{\dag }_{O_M'(\c K')}) \,.
\end{equation} 

\begin{remark}
 Such an element is used usually in the nilpotent Artin-Shreier theory, 
 cf. Section\,\ref{S3}. In the above expression the element $D^{\dag }_{00}$ 
 is not defined. In this and similar cases for the element $e$ from Section\,\ref{S3} 
 we set $D^{\dag }_{00}:=\alpha D^{\dag }_0$, where $\alpha =\hat\alpha 
 \,\op{mod}\, p^M\in W_M(k)$ and the trace of $\hat\alpha $ over $\Q _p$ is equal to 1. 
\end{remark}

Denote by $\wt{e}^{\,\dag }$ the image of 
$\bar e^{\,\dag}$ in $\wt{\c N}^{\,\dag }$.

A special choice of $\bar{e}^{\,\dag }$ allows us to associate 
to the 
above defined action $\Omega _{\gamma }$ on $\wt{\c N}^{\dag }$ the  
\lq\lq conjugated\rq\rq\ action  
$A^{\,\dag }_{\gamma }$ on $\bar{\c L}^{\,\dag }$ as follows. 

\begin{Prop} \label{P2.8}
For $\gamma\in\Z /p^M$, there are unique 
$\wt{c}_{\gamma }\in 
\wt{\c N}^{\,\dag }\langle 1\rangle $ and 
 $A^{\,\dag }_{\gamma }\in \Aut _{\op{Lie}}\bar{\c L}^{\,\dag }$ such that 
 \medskip 
 
{\rm a)}\ $\sigma \tilde c_{\gamma }\in
\wt{\c N}^{\,\dag }\langle 1\rangle $ 
and $ \Omega _{\gamma } (\wt{e}^{\,\dag })=
(\sigma \wt{c}_{\gamma })\circ (A^{\,\dag } _{\gamma }
\otimes\id _{\c K}) \wt{e}^{\,\dag }\circ (- \wt{c}_{\gamma })$;
\medskip 

{\rm b)}\ for any $\iota\in\mathfrak{A}^{0}(p)$, 
$A^{\,\dag }_{\gamma }(D^{\,\dag }_{\iota 0})-D^{\,\dag }_{\iota 0}
\in \underset{m<\op{ch}(\iota )}{\oplus }\bar{\c L}^{\,\dag }[m]_{W(k)}$.
\end{Prop}

\begin{proof} We need the following lemma. 

\begin{Lem} \label{L2.9}
Suppose $n\in\wt{\c N}^{\,\dag}[m,t]\cap C_s
(\bar{\c L}^{\,\dag })_{O'(\c K')}$ where $t\geqslant 1$.
Then there are 
unique $\c S(n),\c R(n)
\in\wt{\c N}^{\,\dag }[m,t]\cap 
C_s(\bar{\c L}^{\,\dag })_{O'(\c K')}$ such that: 

{\rm a)} $\c R(n)=\sum \limits_{\iota}
t^{-\iota }l_{\iota }$ where $\iota $ runs over $\frak{A}^+(p)$ with  
$\op{ch}(\iota )=t$ and all 
 $l_{\iota }\in \left (\bar{\c L}^{\,\dag }[m]
 \cap C_s(\bar{\c L}^{\,\dag })\right )_{W(k)}$;
\medskip 

{\rm b)} $n=\c R(n)+(\sigma -\id )\c S(n)$. 
\end{Lem}

\begin{proof} [Proof of lemma] 
It will be sufficient to specify the elements  
$\c R(n)$ and $\c S(n)$ for 
$n=p^At_1^{-b}[D^{\,\dag }_{\bar\iota \bar n}]$, 
where $b\in\frak{A}^0$, $\op{ch}(b)=t$ and $\op{ch}(\bar\iota )=m$. 

It is essential that $\op{ch}(b)\geqslant 1$ and, therefore, $b\ne 0$. 

If $-b>0$ we set $\c R(n)=0$ and $\c S(n)=-\sum _{i\geqslant 0}\sigma ^i(n)$. 

If $-b=-\iota _0p^{i_0}<0$ with $\iota _0\in\frak{A}^+(p)$ we set 
$\c R(n)=\sigma ^{-i_0}n$ and 
$\c S(n)=\sum _{0\leqslant j<i_0}\sigma ^{-j}n$. 

The uniqueness property holds because $\op{ch}(b)\geqslant 1$ implies that 
$b\ne 0$. The lemma is proved.
\end{proof} 

Continue the proof of Prop.\,\ref{P2.8}. 
\medskip 

Use induction on $i\geqslant 1$ to proceed modulo the ideals 
$\wt{\c N}^{\,\dag }\langle 1\rangle 
\cap C_i(\bar{\c L}^{\,\dag })_{O_M'(\c K')}$. 
\medskip 

--- If $i=1$ take  $ \wt{c}_{\gamma }=0$, 
$A^{\,\dag }_{\gamma }=\id $ and use that  
$\Omega _{\gamma }(\wt{e}^{\,\dag })- \wt{e}^{\,\dag }\in 
\wt{\c N}^{\,\dag}\langle 1\rangle \,.$ 
\medskip 

--- Suppose $i\geqslant 2$,    
$ \wt{c}_{\gamma }\in\wt{\c N}^{ \dag }\langle 1\rangle $ 
and $A_{\gamma }^{\,\dag }\in
\op{Aut}_{\op{Lie}}(\bar{\c L}^{\,\dag })$ are such that   
$$H=\Omega _{\gamma } \wt{e}^{\,\dag }- 
(\sigma  \wt{c}_{\gamma })\circ (A^{\,\dag } _{\gamma }
\otimes\id _{\c K'})\wt{\,e}^{\,\dag }\circ 
(- \wt{c}_{\gamma })\in \wt{\c N}^{\,\dag }\langle 1\rangle \cap\, 
C_i(\wt{\c N}^{\,\dag })_{O_M'(\c K')} \, .$$

Then $\c R(H),\c S(H)\in \wt{\c N}^{\,\dag }\langle 1\rangle \cap\, 
C_i(\wt{\c L}^{\,\dag })_{O'_M(\c K')})$. 

Set   
$\c R(H)=\sum\limits  _{\op{ch}(\iota )
\geqslant 1+m}t^{-\iota }H_{\iota m}$, where 
all $H_{\iota m}\in \wt{\c L}^{\,\dag }[m]_k
\cap C_i(\bar{\c L}^{\,\dag })_{W(k)}$. 

Introduce $A^{\,\dag \prime}_
{\gamma }\in\Aut _{\op{Lie}}(\bar{\c L}^{\,\dag })$ 
by setting for all involved $\iota $ and $m$,  
$A_{\gamma }^{\,\dag \prime }(D^{\,\dag }_{\iota 0})=
A^{\,\dag }_{\gamma }(D^{\,\dag }_{\iota 0})-
\sum _{m}H_{\iota m}$. Set also 
$\wt{c}^{\,\prime }_{\gamma }= \wt{c}_{\,\gamma }-{\c S}(H)$. Then 
$$\Omega _{\gamma } \wt{e}^{\,\dag }\equiv  
(\sigma  \wt{c}^{\,\prime }_{\gamma 1})
\circ (A^{\,\dag \prime } _{\gamma }\
\otimes\id _{\c K})\wt{e}^{\,\dag }\circ 
(- \wt{c}^{\,\prime }_{\gamma })\,\op{mod}\,
\wt{\c N}^{\,\dag }\langle 1\rangle\cap C_{i+1}(\wt{\c N})_{O'_M(\c K')} \, .$$

The uniqueness follows similarly by induction on 
$i$ and the uniqueness part of 
Lemma \ref{L2.9}. 

The proposition is proved. 
\end{proof} 

We have also the following properties. 

\begin{Cor} \label{C2.10} 
 For any $\gamma ,\gamma _1\in\Z /p^M$, 
 
 {\rm a)}\ $A^{\,\dag } _{\gamma +\gamma _1}=
A^{\,\dag } _{\gamma }A^{\,\dag } _{\gamma _1}$;
 
 {\rm b)}\ $\Omega _{\gamma }(\wt{c}_{\gamma _1})
\circ (A^{\,\dag } _{\gamma _1}\otimes\id _{\c K}) \wt{c}_{\gamma }=
  \wt{c}_{\gamma +\gamma _1}$;
 
 {\rm c)}\ if $l\in\bar{\c L}^{\,\dag }[m]$ then 
$A^{\,\dag }_{\gamma }(l)-l\in \oplus _{m'<m}
\bar{\c L}^{\,\dag }[m']$, e.g.  
$A^{\,\dag } _{\gamma }|_{\bar{\c L}^{\,\dag }[0]}=\id \, .$
\end{Cor}

\subsection{The action $\Omega _U$} \label{S2.6}

Let $A^{\,\dag }:=A^{\,\dag }_{\gamma }|_{\gamma =1}$. 
Then for any $\gamma =n\,\op{mod}\,p^M$,   
$A^{\,\dag }_{\gamma }=A^{\,\dag n}$, in particular,  
$A^{\,\dag p^M}=\id _{\bar{\c L}^{\,\dag }}$. 
By part c) of the above corollary, 
for all $m\geqslant 0$,  
there is a differentiation $B^{\,\dag }\in
\op{End} _{\op{Lie}}\bar{\c L}^{\,\dag }$ such that 
for all $m\geqslant 0$, 
$B^{\,\dag }(\bar{\c L}^{\,\dag }[m])
\subset\oplus _{m'<m}\bar{\c L}^{\,\dag }[m']$ 
and for all $\gamma\in\Z /p^M$, 
$A^{\,\dag } _{\gamma }=\wt{\exp }(\gamma B^{\,\dag })$. 

The derivation $B^{\dag }$ can be recovered via methods from 
papers \cite{Ab12}-\cite{Ab13}. 

Let $U$ be an indeterminate. 
If  
$b =p^m(qp^{M-1}\alpha -(q-1)\beta )\in \mathfrak{A}^0$ 
and $l=p^At_1^{-b}[D_{\bar\iota\bar n}]\in\wt{\c N}^{\,\dag }\,$ set 
$$\Omega _{U}l:=
\wt{\exp}( U\otimes (p^m\alpha )S_{b^*})l\,\in \Z /p^M[U]
\otimes \wt{\c N}^{\,\dag }\, .$$ \

Proceeding similarly to the proof of Proposition\,\ref{P2.8} we obtain 
\begin{equation} \label{E2.3}
\Omega _U(\wt{e}^{\,\dag})=\sigma (\wt c_U)
\circ (A^{\,\dag }_U\otimes\id )\wt{e}^{\,\dag }\circ (-\wt c_U)\, ,
\end{equation} 
where   
$A^{\,\dag }_U=\wt{\exp}(UB^{\,\dag })$.  

It is easy to see that we have the action 
of the additive group $\mathbb{G}_a
=\op{Spec}\,\Z /p^M[U]$ on $\wt{\c N}^{\,\dag }$. 
Similarly to \cite{Ab15} and by taking into account 
the compatibility of this action with the grading on  
$\wt{\c N}^{\,\dag }$ we obtain the following:
\medskip 

\begin{Prop} \label{P2.11} \ 

{\rm a)}\ $\wt{c}_U= \wt{c}^{\,(1)}U+\dots +
 \wt{c}^{\,(p-1)}U^{p-1}$, where all $\wt{c}^{(j)},\sigma \wt{c}^{(j)}
 \in\wt{\c N}^{\,\dag }\langle j\rangle $;   
\medskip

{\rm b)}\ the cocycle $ \wt{c}_U$ is determined uniquely by its 
linear part $\wt{c}^{\,(1)}$; 
\medskip 

{\rm c)}\ the action $\Omega _U=\sum _{0\leqslant i<p}\Omega ^{\,(i)}U^i$ 
(here $\Omega ^{(0)}=\id $) is recovered uniquely from its 
differential $d\,\Omega _U:=\Omega ^{(1)}U$. 
\medskip 

{\rm d)} $A^{\dag }_U=\wt{\exp}B^{\dag }_U$ and 
$B^{\dag }_U(\bar{\c L}^{\dag }[m])\subset \bar{\c L}^{\dag }[m-1]$. 
\end{Prop} 

\begin{remark} Prop.\,\ref{P2.11}  is directly related to  
Corollary\,\ref{C2.10} as follows: 
\medskip 

 1) For any $\gamma\in\Z /p^M$, $\wt{c}_U|_{U=\gamma }=\wt{c}_{\gamma }$ 
 and $\Omega _U|_{U=\gamma }=\Omega _{\gamma }$.
 \medskip 
 
 2) Vice versa, the coefficients $\wt{c}^{(j)}$ can be uniquely recovered 
 from the values $\wt{c}_U|_{U=[\alpha ]}$, where $[\alpha ]$ are 
 the Teichmuller representatives of $\alpha\in\F _p$. (Use that 
 the matrix of Teichmuller representatives 
 $(([i\,\op{mod}\,p]^j))_{1\leqslant i,j<p}$, is non-degenerate.) 
\end{remark}

\begin{Cor}\label{C2.12} For $m\in\mathbb{N}$, 
 $B^{\,\dag }(\wt{\c L}^{\,\dag }[m])\subset \wt{\c L}^{\,\dag }[m-1]$. 
\end{Cor}

\subsection{Ideals 
$\bar{\c L}^{\,\dag [v_0]}$ and $\bar{\c L}^{[v_0]}$} 
\label{S2.7} \ \

Recall that  $\bar{\c L}^{\,\dag }[0]$ is the minimal Lie 
subalgebra of $\bar{\c L}^{\,\dag }$  
containing all $D^{\,\dag }_{\iota n}$ 
with $\op{ch}(\iota )=0$. 
By Remark from Sect.\,\ref{S2.3}, $\c L^{\dag }[0]_{W(k)}$ is generated by 
\begin{equation} \label{E2.3} 
\{D^{\dag }_{qp^{M-1}a,n}\ |\ a\in\Z ^0(p)\cap [0,(p-1)v_0), n\in\Z /N_0\, \}\,.
\end{equation}
Recall also that $\c L^{\dag }[0]$ is naturally identified with a  
Lie subalgebra of $\c L$ via 
$j^0: D^{\dag }_{qp^{M-1}a,n}\mapsto D_{an}$. 
This identification agrees with filtrations and plays 
a crucial role below, where we work with the criterion from 
Prop.\,\ref{P3.2} in the algebra $\c L^{\dag}[0]$. 
\medskip

Note also that ${\c L}^{\,\dag }[0]$ has the induced filtration 
$\{{\c L}^{\,\dag }(w)[0]\}_{w\geqslant 1}$ and \
there is epimorphism of filtered Lie algebras 
$\c V^{\,0}:{\c L}^{\,\dag }\To {\c L}^{\,\dag }[0]$ such that 
$D^{\,\dag }_{\iota n}\mapsto D^{\,\dag }_{\iota n}$ if 
$\op{ch}(\iota )=0$ and $D^{\,\dag }_{\iota n}\mapsto 0$, otherwise. 
Similarly, we have the epimorphism $\bar{\c V}^0:
\bar{\c L}^{\dag }\to\bar{\c L}^{\dag }[0]$ and the identification 
$\bar {\jmath }^0:\bar{\c L}^{\dag }[0]\simeq \bar{\c L}$ of filtered Lie algebras.

Let $\bar{\c V}:=\bar{j^0}\bar{\c V}^0:\bar{\c L}^{\,\dag }\To\bar{\c L}$.

Define the ideal $\bar{\c L}^{\,\dag [v_0]}$ as the minimal 
ideal in $\bar{\c L}^{\,\dag }$ containing all 
$B^{\,\dag }(\Ker\,\bar{\c V})$.  
Set $\bar{\c L}^{[v_0]}=\bar{\c V}(\bar{\c L}^{\,\dag [v_0]})$.  

\begin{definition}
 Let $\mathfrak{A}^0_1=\{\iota\in\mathfrak{A}^0\ |\ \op{ch}(\iota )=1\}$ 
 and 
 $\mathfrak{A}_1^+(p)=\mathfrak{A}^0_1\cap \mathfrak{A}^+(p)$.
\end{definition}

\begin{Prop} \label{P2.13} 
$\bar{\c L}^{[v_0]}$ is the minimal ideal in $\bar{\c L}$ such that 
for all $\iota\in \frak{A}_1^+(p)$,  
$\bar{\c V}(D^{\dag }_{\iota 0})\in\bar{\c L}^{[v_0]}_{W_M(k)}$. 
\end{Prop} 

\begin{proof} Clearly, $\Ker \bar{\c V}$ is generated by 
the elements of $\bar{\c L}^{\,\dag }[m]$ with $m\geqslant 1$. 
Therefore, for $m\geqslant 2$, 
$B^{\,\dag }(\bar{\c L}^{\,\dag }[m])
\subset \bar{\c L}^{\,\dag }[m-1]\subset \Ker \bar{\c V}$. As a result, 
the elements of $\bar{\c V}B^{\dag }(\bar{\c L}^{\,\dag }[1])$ 
generate the ideal 
$\bar{\c L}^{[v_0]}$. If $r(\bar\iota )>1$ and $\op{ch}(\bar\iota )=1$ 
then $l=[D^{\dag }_{\bar\iota \bar n}]$  can be written 
in the form $[D^{\dag }_{\iota n}, l']$ where $\iota\in\frak{A}_1^+(p)$ 
and 
$l'\in \wt{\c L}^{\,\dag }[0]_{W_M(k)}$. This implies that 
$\bar{\c V}B^{\,\dag }(l)\in [\bar{\c L}^{[v_0]}, \bar{\c L}]_{W_M(k)}$. 

The proposition is proved. 
\end{proof}

\begin{Prop} \label{P2.14} 
 There are unique elements $\bar{c}^{(j)}\in\bar{\c N}^{\dag }(j)$ such that 
 $\pi (\bar{c}^{(j)})=\wt{c}^{(j)}$  
and  
 $$\Omega (\bar{e}^{\,\dag})=\sigma (\bar c)
\circ (A^{\dag }\otimes\id _{O'_M(\c K')})\bar{e}^{\,\dag }\circ (-\bar c)\, $$
where $\bar c=\bar c^{(1)}+\dots +\bar c^{(p-1)}$.
\end{Prop}
\begin{proof}
 It follows easily from the fact that 
 $\bar{\c N}^{\dag }(p-1) \subset \bar{\c L}^{\dag}_{\m '}$ 
 and, therefore, the map 
 $\sigma -\id $ is a bijection on $\bar{\c N}^{\dag }(p-1)$.
\end{proof}

%%%%%%%%%%%%%%%%%%%%%%%%%%%%%%%%%%%%%%%%%%%%

\section{Relation to Galois theory} \label{S3}

In \cite{Ab2} we developed a nilpotent analogue of the classical 
Artin-Schreier theory of cyclic extensions of fields of characteristic $p$. Below 
we are going to use the covariant analogue of this theory, 
cf.\,the discussion in \cite{Ab9}, 
for an explicit description of the ramification subgroup 
$\c G^{(v_0)}_{<p,M}$ of $\c G_{<p,M}=
\c G/\c G^{p^{M}}C_p(\c G)$. 

Recall that $\c K=k((t_0))$ and we use 
the construction of the lift $O_M(\c K_{sep})$ of $\c K_{sep}$ 
modulo $p^M$ with respect to our given 
uniformiser $t_0$. This lift appears in the form 
$W_M(\sigma ^{M-1}\c K_{sep})[t]\subset W_M(\c K_{sep})$, where $t=[t_0]$ 
is the Teichmuller representative of $t_0$. Then 
$\c G=\op{Gal}(\c K_{sep}/\c K)$ acts on this lift and 
$O_M(\c K_{sep})^{\c G}=O_M(\c K)=W_M(k)((t))$. 
\medskip

\subsection{Groups and Lie algebras of nilpotent class $<p$, 
\cite{La}}\label{S3.1} 

Let $L$ be a finite Lie  $\Z /p^M$-algebra of 
nilpotent class $<p$, i.e. $C_p(L)=0$. 
Let $A$ be an enveloping algebra  of $L$. 
There is a natural embedding 
$L\subset A/J^p$ (here $J$ is the augmentation ideal of $A$) 
and in terms of this embedding 
the Campbell-Hausdorff formula appears as    
$$(l_1,l_2)\mapsto l_1\circ l_2=
l_1+l_2+\frac{1}{2}[l_1,l_2]+\dots ,\ \ l_1,l_2\in L\, ,$$
where $\wt{\exp}(l_1)\cdot 
\wt{\exp}(l_2)\equiv \wt{\exp}(l_1\circ l_2)\,\op{mod}\,J^p$.   
This composition law provides the set $L$ with 
a group structure and we denote this group by $G(L)$. 
Clearly, $G(L)$  has period $p^M$ and nilpotent class $<p$. 
Then the correspondence 
$L\mapsto G(L)$ induces  
equivalence of the category of finite $p$-groups of 
period $p^M$ and nilpotent class $s<p$ 
and the category of finite Lie 
$\Z /p^{M}$-algebras of the same nilpotent class $s$. 
This  
equivalence can be naturally extended to similar  categories of 
pro-finite objects. 
\medskip

\subsection{Epimorphism  $\eta ^0$} \label{S3.2} 

Let $L$ be a finite 
$\Z /p^{M}$-Lie algebra of nilpotent class $<p$. 
Set  $L_{sep}:=L_{O_M(\c K_{sep})}$. 
Then  the elements of 
$\c G=\Gal (\c K_{sep}/\c K)$ and  Frobenius 
$\sigma $ act on $L_{sep}$ through the second factor, 
$L_{sep}|_{\sigma =\id}=L$ and $(L_{sep})^{\c G}=L_{O_M(\c K)}$. 
The covariant nilpotent Artin-Schreier theory states that  
for any $e\in G(L_{O_M(\c K)})$,  the set 
$$\c F(e)=\{f\in G(L_{sep})\ |\ \sigma (f)=e\circ f\}$$  
is not empty and for any fixed $f\in \c F(e)$, the map  
$\tau\mapsto (-f)\circ \tau (f)$ is a continuous group homomorphism 
$\pi _f(e):\c G\To G(L)$.

Suppose $\c L$ is the Lie algebra from Section \ref{S1.1} and 
$e=e(t)\in \c L_{\c K}$ satisfies $e\equiv \sum _{a\in\Z ^0(p)}
t^{-a}D_{a0}\,\op{mod}\, C_2(\c L)_{O_M(\c K)}$. Recall that for $n\in\Z /N_0$, 
$D_{0n}$ is defined as 
$\sigma ^n(\alpha _0)D_0$, where $\alpha _0=\hat\alpha _0\,\op{mod}\,p^M$ 
and $\hat{\alpha} _0\in W(k)$ has absolute trace 1. 
(We followed this agreement in \cite{Ab1}-\cite{Ab3} 
to relate the identification $\eta ^0\,\op{mod}\,C_2(\c G)$, 
cf.\,below, to the  identification of class field theory.)
Then a profinite version of the above Artin-Schreier 
theory gives the existence of
$f\in \c L_{sep}$ such that 
$\sigma f=e\circ f$ and the 
map 
$\eta ^0_f:=\pi _f(e)\,\op{mod}\,\c G^{p^{M}}C_p(\c G)$ 
induces a group isomorphism 
$\eta ^0_f:\c G_{<p}\simeq G(\c L)$. We fix a choice of such $f$ 
and use the simpler notation $\eta ^0=\eta ^0_f$. 

Theorem\,\ref{T3.1} below requires the following  
assumption:
\medskip 

$\bf A)$ {\it Let $\c N$ be a $W(k)$-Lie subalgera of $\c L_{O_M(\c K)}$ 
generated by all $t^{-a}D_{an}$, $a\in\Z ^0(p)$. Then 
$e\in\c N$ and $e\equiv\sum _{a\in\Z ^0(p)}t^{-a}D_{a0}
\,\op{mod}\,[\c N,\c N]$.}

\begin{Thm} \label{T3.1} 
 $\eta ^0(\c G^{(v_0)})=\c L^{[v_0]}$.
\end{Thm}

\begin{remark} The above quite general result will 
be explicitly specified in 
Sec.\,\ref{S4}, where we present another proof 
of the explicit description of $\c L^{[v_0]}$     
in terms of generators 
$\{D_{an}\ |\ a\in\Z ^+(p), n\in\Z/N_0\}\cup\{D_0\}$ of $\c L_{W(k)}$ from \cite{Ab3}.  
\end{remark}

The proof of Theorem \ref{T3.1} is given below in 
sections \ref{S3.3}-\ref{S3.5}.

\subsection{The criterion} \label{S3.3} 

Consider the parameters $N^*, r^*, q$ from Section\,\ref{S2.1} and  
the field extension $\c K'/\c K$ from Section\,\ref{S2.4}. 
Recall that $\c K'=k((t_{01}))$ and we use the local parameter 
$t_{01}$ to construct 
the lifts $\ O'_M(\c K')=W_M(k)((t_1))$ and 
$O'_M(\c K'_{sep})=W_M(\sigma ^{M-1}\c K')[t_1]$, where $t_1=[t_{01}]$. 
There is a natural embedding of $W(k)$-algebras 
$\iota :\sigma ^{M-1}O_M(\c K)\to O'_M(\c K')$. 
This allows us to   
identify the Lie $W(k)$-algebra $\c N^{\dag }[0]$ with 
subalgebra of $\c L_{O'_M(\c K')}$. 
In particular, 
$\sigma ^{M-1}e(t)$ is identified with   
$\Omega \left (\sigma ^{M-1}e(t_1^q)\right )$, cf. Section \ref{S2.5}. 

Clearly, we have $\sigma ^M(f)=
\sigma ^{M-1}e(t)\circ \sigma ^{M-1}f$ and 
for any $\tau\in\c G$, 
$$\eta ^0(\tau )=(-\sigma ^{M-1}f)\circ 
\tau (\sigma ^{M-1}f)
=(-f)\circ \tau (f)\, .$$  

Choose $f_1\in {\c L}_{O'_M(\c K_{sep})}$ such 
that $\sigma f_1=e(t_1)\circ f_1$ and consider 
$X\in{\c L}_{O'_M(\c K_{sep})}$ such that 
$$\sigma ^{M-1}f=X\circ (\sigma ^{M-1+N^*}f_1)\,.$$

For $Y\in \c L_{O'_M(\c K_{sep})}$ and an ideal $\c I$ 
in $\c L$,
define the field of definition  
of $Y\,\op{mod}\,\c I_{O'_M(\c K_{sep})}$ over  $\c K'$ 
as 
$\c K'(Y\,\op{mod}\,\c I_{O'_M(\c K_{sep})}):=\c K_{sep}^{\c H}$,  
where 
$$\c H=\{g\in\c G'=
\op{Gal}(\c K_{sep}/\c K')\ |\ (\id _{\c L}\otimes g)Y
\equiv Y\,\op{mod}\,\c I_{sep}\}\,.$$ 

For a field extension $\c E'/\c K'$ in $\c K_{sep}$,  
define the biggest upper ramification number 
$v(\c E'/\c K')=\op{max}\{v\ |\ \Gal (\c K_{sep}/\c K')^{(v)} 
\text{acts non-trivially on } \c E'\}\, .$ 
 
Set $e_1^{(q)}:=e(t_1^{q})$. Then   
$\Omega \left (\sigma ^{M-1}e_1^{(q)}\right )=
e(t_1^{qp^{M-1}}E(t_1^{b^*})^{-p^{M-1}})$. 
The methods from \cite{Ab1, Ab2, Ab3} are based 
on the following criterion.

\begin{Prop} \label{P3.2} 
If $X\in \c L_{sep }$ is such that 
\begin{equation} \label{E3.1}
\Omega \left (\sigma ^{M-1}e_1^{(q)}\right )=
\sigma X\circ \sigma ^{M-1}e_1^{(q)}\circ (-X)\, 
\end{equation} 
then $\c L^{(v_0)}$ is the minimal ideal in the set of all ideals 
$\c I$ of $\c L$ such that 
$$v(\c K'(X\,\op{mod}\,\c I_{O'_M(\c K_{sep})})/\c K')< qv_0-b^*\,.$$ 

\end{Prop} 

\begin{remark}
The above inequality can be rewritten also as  
$$v(\c K'(X\,\op{mod}\,\c I_{O'_M(\c K_{sep})})/\c K)< v_0\,.$$
 \end{remark}
 
% Namely, let $e^{(q)}=e(t^q)$, $\Omega $ -- the action from Sec.\,\ref{S??} 
% and $X\in \c L_{sep}$ is such that 
% $\sigma ^{M-1}(\Omega *e^{(q)})=(\sigma X)\circ e^{(q)}\circ (-X)$. 
% 
% \begin{Prop} \label{P3.3}
%  $\c L^{(v_0)}$ is the minimal ideal in the set of ideals 
%$\c I$ of $\c L$ such that 
%$v(\c K(X\,\op{mod}\,\c I_{sep})/\c K)< qv_0-b^*\,.$ 
%\end{Prop} 

Suppose $\wt{\c J}\subset\c L$ is a closed ideal and 
$\pi :\c L\To L:=\c L/\wt{\c J}$ 
is a natural projection. Then we can use 
the appropriate images $e_L, e_L^{(q)}\in L_{O'_M(\c K')}$ and 
$f_L, X_L\in L_{O'_M(\c K_{sep})}$ 
to state the following  analog 
of Proposition\,\ref{P3.2} for the ideal 
$L^{(v_0)}:=\pi (\c L^{(v_0)})$ of $L$. 

\begin{Prop} \label{P3.3} 
$L^{(v_0)}$ is 
the minimal ideal in the set of all ideals 
$\c I$ of $L$ such that 
$v(\c K'(X_{L}\,\op{mod}\,\c I_{sep})/\c K')< qv_0-b^*$.   
\end{Prop} 

%%%%%%%%%%%%%%%%%%%%%%%%%%%%%%%%%%%%%%%%%%%%%%%%%%%%%%%

\subsection{Relation between $\c L^{(v_0)}$ and $\c L^{[v_0]}$}
\label{S3.4} 

\begin{definition} 
For the projection $\op{pr} :\c L\to\bar{\c L}$, set  
$\c L^{[v_0]}:=\op{pr}^{-1}(\bar{\c L}^{[v_0]})$. 
\end{definition}

Consider equality \eqref{E3.1} from the above  criterion. 
On the other hand, according to Proposition\,\ref{P2.14} 
we have in $\bar{\c N}^{\,\dag }$ that  
\begin{equation} \label{E3.2}
\Omega \left (\sigma ^{M-1}\bar{e}^{\dag }\right )=
\sigma  (\bar{c}_1)\circ 
(A^{\dag }\otimes\id )\sigma ^{M-1}
\bar{e}^{\dag }\circ (-\bar{c}_1)\,.
\end{equation}

Applying $\bar{\c V}^0:\bar{\c L}^{\dag }\to \bar{\c L}^{\dag }[0]$ 
and using the definition of the ideal $\bar{\c L}^{[v_0]}$ 
we obtain the following congruence in $\bar{\c N}^{\dag }[0]$ 
with $\bar x=\bar{\c V}^0(\bar c_1)$:
\begin{equation}\label{E3.3} \Omega 
\left (\sigma ^{M-1}\bar{e}_1^{(q)}\right )\equiv 
\sigma (\bar x)
\circ \sigma ^{M-1}\bar{e}_1^{(q)}\circ (-\bar x)
\,\op{mod}\,\bar{\c L}^{[v_0]}_{O'_M(\c K')}\, .
\end{equation} 

Since $\c L^{\dag }[0](p)=\op{Ker} (\c L^{\dag }[0]\to\bar{\c L}^{\dag }[0])
\subset \c L^{[v_0]}$ we can take any 
$x\in \c N^{\dag }[0]$ such that 
$\bar x=x\,\op{mod}\,\c L^{\dag }[0]_{O'_M(\c K')}$ and consider  
congruence \eqref{E3.3} in $\c N^{\dag }[0]$: 
$$\Omega \left (\sigma ^{M-1}{e}_1^{(q)}\right )\equiv 
\sigma (x)
\circ \sigma ^{M-1}{e}_1^{(q)}\circ (-x)
\,\op{mod}\,{\c L}^{[v_0]}_{O'_M(\c K')}\, .$$
In particular, the criterion from Prop.\,\ref{P3.2} 
implies that ${\c L}^{(v_0)}\subset {\c L}^{[v_0]}$. 
(Use that $x$ is defined just over a trivial 
field extension $\c K'/\c K'$.)

This means also that we can assume $X\equiv x\,
\op{mod}\,\c L^{[v_0]}_{sep}$. 
\medskip 

The following proposition implies  
Theorem \ref{T3.1} when $w=p$. 

\begin{Prop} \label{P3.4} For any $1\leqslant w\leqslant p $, it holds 
 \medskip 

$a_{w})$\  $\c L^{(v_0)}+C_{w}(\c L)=
\c L^{[v_0]}+C_{w}(\c L)\, .$
\medskip 

$b_{w})$\  $\c L(w)\subset \c L^{(v_0)}+C_{w}(\c L)$. 
\end{Prop} 

\subsection{The proof of Proposition\,\ref{P3.4}} \label{S3.5}  

Use induction on $w\geqslant 1$. 
If 
$w=1$ there is nothing to prove. Assume that proposition holds for some 
$1\leqslant w<p$. 

For (non-commuting) variables $U$ and $V$, let 
$$\delta ^0(U,V):=U\circ V-(U+V)\, .$$
Note that if the values of 
$U$ and $V$ are defined modulo $C_{w}(\c L)_{O'_M(\c K')}$ then 
the value of 
$\delta ^0(U,V)$ is well-defined modulo $C_{w+1}(\c L)_{O'_M(\c K')}$.

\begin{Lem} \label{L3.5}  There is a congruence 
$$\sigma X-X\equiv y\,
\op{mod}\,\left ([\c L^{(v_0)},\c L]+
C_{w+1}(\c L)\right )_{sep}\, ,$$
where the element $y\in\c N^{\dag }[0]$ is such that 
$$y=\Omega \left (\sigma ^{M-1}e_1^{(q)}\right )-\sigma ^{M-1}e_1^{(q)}+
\delta ^0\left (\Omega \left (\sigma ^{M-1}e_1^{(q)}\right ),x\right )-
\delta ^0\left (\sigma x,\sigma ^{M-1}e_1^{(q)}\right )\,.$$
\end{Lem}

\begin{proof}[Proof of lemma] 

Use that  
$\Omega \left (\sigma ^{M-1}e_1^{(q)}\right )\circ X=
\sigma X\circ \sigma ^{M-1}e_1^{(q)}$ and 
there are the following congruences modulo 
$\left ([\c L^{(v_0)},\c L]+
C_{w+1}(\c L)\right )_{sep}$: 
$$\delta ^0\left (\Omega \left (\sigma ^{M-1}e_1^{(q)}\right ),
X\right )\equiv 
\delta ^0\left (\Omega \left (\sigma ^{M-1}e_1^{(q)}\right ), 
x\right )\, ,$$ 
$$\delta ^0(\sigma X,\sigma ^{M-1}e_1^{(q)})\equiv 
\delta ^0\left (\sigma x,\sigma ^{M-1}e_1^{(q)}\right )\, .$$ 
\end{proof}

In order to prove part $b_{w+1})$ it will be suficient 
to verify that
\medskip 

$(\bullet _1) $ {\it if $p^mD_{a0}\in \c L(w+1)_{W(k)}$ then  
$p^mD_{a0}\in (\c L^{(v_0)}+C_{w+1}(\c L))_{W(k)}$.} 
\medskip 

This will be done by the use of the characterization  
of $\c L^{(v_0)}$ via 

\begin{equation} \label{E3.4} 
v(\c K'(X\,\op{mod}\,\c L^{(v_0)}_{sep})/\c K')
< qv_0-b^*\,,
\end{equation}

Suppose $y=\sum _{\iota}y_{\iota }+y_+$, where $y_+\in\c L_{\m '}$, 
$\iota $ runs over $\mathfrak{A}^+(p)$ and the corresponding  
$y_{\iota }$ appear in the form 
$y_{\iota }=\sum _{r\geqslant 0}t_1^{-\iota p^r}l_{\iota r}$ 
with all $l_{\iota r}\in\c L_{W(k)}$. 

Then $X=\sum _{\iota }X_{\iota }+X_+$ where 
$\sigma X_{\iota }-X_{\iota }=y_{\iota }$ 
and $\sigma X_+-X_+=y_+$. Here 
$X_+\in\c L_{\m '}$ and does not contribute to \eqref{E3.4}. 
The elements $X_{\iota }$ contribute to \eqref{E3.4} as follows:  
\medskip 

-- {\it if $\iota \in\mathfrak{A}^+(p)$ and 
$\iota p^{-v_p(\iota )}\geqslant qv_0-b^*$ then} 
$$\sum _{r}\sigma ^{-r}l_{\iota r}
\in (\c L^{(v_0)}+C_{w+1}(\c L))_{W(k)}\, .$$ 

We can follow the contribution to the upper ramification 
in \eqref{E3.4} 
for each $\iota \in\mathfrak{A}^+(p)$ separately because 
by Proposiion\,\ref{P2.4} 
these contributions are independent. 

Let $\iota _0=qp^{M-1}a-p^{M-1}b^*U(1,m)=
p^{M-1-m}(qp^ma-b^*)\in\frak{A}^0$. 

Then $\iota _0>0$. Indeed, 
$p^mD_{a0}\in \c L(w+1)_{W(k)}$ implies that 
$$qp^{M-1}a\geqslant qp^{M-1}v_0U(w,m)>
p^{M-1}b^*U(w,m)\geqslant p^{M-1}b^*U(1,m)\,.$$ 
Therefore, $\iota _0\in\frak{A}_1^+(p)$  
(use that $a$ is prime to $p$) and $p^ma\geqslant v_0$. 
This means that the $\iota _0$-part $X_{\iota _0}$ 
should contribute to 
condition \eqref{E3.4}.

In order to find the contribution of $X_{\iota _0}$ 
notice that 
the element $y$ has a linear part 
\begin{equation} \label{E3.5} L(y):=
\Omega \left (\sigma ^{M-1}e_1^{(q)}\right )-\sigma ^{M-1}e_1^{(q)}
=\sum _{a,s}t_1^{-qp^{M-1}a}aS_{b^*}^sD_{a0}
\end{equation}
and a non-linear part, which is a $W(k)$-linear combination 
of various 
\begin{equation} \label{E3.6} Q(\iota ,\bar a_1,\bar a_2)=
p^{A}t_1^{-\iota }\left [[D_{\bar a_1\bar n_1}],
[D_{\bar a_2\bar n_2}]\right ]\,,
\end{equation}
where 
$\iota \in \mathfrak{A}^0$, $\op{ch}(\iota )=s$ and 
$-\iota \geqslant -qp^{M-1}(|\bar a_1|+|\bar a_2|)+p^{M-1}b^*U(s,A)$ 
with some $s\geqslant 1$. 
Any such term $Q(\iota , \bar a_1, \bar a_2)$ 
contributes to the $\iota _0$-part 
iff 
$\iota =p^r\iota _0$ and $s=\op{ch}(\iota _0)=1$. 

\begin{Lem} \label{L3.6} a) 
Any $Q(\iota _0p^u, \bar a_1, \bar a_2)
\in\c L(w+1)_{O'_M(\c K')}$.
\medskip 

b) If $Q \in\c L_{\c N}(w+1)_{O'_M(\c K')}$ then 
it belongs to $([\c L^{(v_0)},\c L]+C_{w+1}(\c L))_{O'_M(\c K')}$.
\end{Lem}

\begin{proof} a) 
The linear $\iota _0$-term is 
$$p^mD_{a0}t_1^{-qp^{M-1}a+p^{M-1}b^*/p^m}=
p^mD_{a0}t_1^{-p^{M-m-1}(qp^ma-b^*)}\,.$$
It contributes with the upper ramification number $qp^ma-b^*$ 
(use the inequality  
$ap^m\geqslant v_0$). 
Consider 
$$ Q(\iota _0p^u, \bar a_1, \bar a_2)=
p^At_1^{-qp^{M-1}b+p^{M-1-A}b^*}
[[D_{\bar a_1\bar n_1}],[D_{\bar a_2\bar n_2}]]$$
which contributes to 
the $\iota _0$-part. 
Here $-b=-p^{m-A}a\geqslant -qp^{M-1}(|\bar a_1|+|\bar a_2|)$. Therefore, 
$|\bar a_1|+|\bar a_2|\geqslant p^{m-A}a\geqslant p^{m-A}U(w,m)v_0$. 

Recall that $U(w,m)=w_0/p^u+w_1/p^{u+1}$, where $u\geqslant 0$, $w=w_0+w_1$, 
$m=wu+w_1$, $0\leqslant w_1<w$. 

In the case  $m-A\leqslant u$ we have $p^{m-A}U(w,m)=U(w,A)$ and, 
therefore,  $Q(\iota _0, \bar a_1, \bar a_2)\in\c L(w+1)_{O_M'(\c K')}$. 

If $m-A=u+u_1$ with $u_1\geqslant 1$ then 
$$|\bar a_1|+|\bar a_2|\geqslant 
p^{u_1}(w_0+w_1/p)v_0>wv_0$$ 
and already 
$[[D_{\bar a_1\bar n_1}],[D_{\bar a_2\bar n_2}]]\in \c L(w+1)_{W(k)}$.
\medskip

b) We know that 
$|\bar a_1|+|\bar a_2|\geqslant (b^*/q)U(w,A)\, $, where 
$b^*/q$ could be replaced by $v_0$, cf.\,Section\,\ref{S2.1}, and, 
therefore,  
$p^A\left [[D_{\bar a_1\bar n_1}], [D_{\bar a_2\bar n_2}]\right ]$
belongs to the $W(k)$-extension of scalars of 
$$\c L_{\c N}(w+1)\cap C_2
(\c L)\subset \sum _{w_1+w_2= w+1}[\c L_{\c N}(w_1), 
\c L_{\c N}(w_2)]\, .$$ 

By inductive assumption $b_{w})$ the RHS belongs to  
$$\sum _{w_1+w_2=w+1}[\c L^{(v_0)}+C_{w_1}(\c L), 
\c L^{(v_0)}+C_{w_2}(\c L)]
\subset [\c L^{(v_0)}, \c L]+C_{w+1}(\c L)\, .$$
\end{proof} 

As a result of the above Lemma, the $\iota _0$-contribution  
of the non-linear part of $y$ 
belongs to  
$[\c L^{(v_0)},\c L]+C_{w+1}(\c L)$ 
(and could be ignored modulo $\c L^{(v_0)}_{sep}$). 

It remains to note that if  
$y_{\iota }=p^mt_1^{-\iota _0p^u}l_u$ is a term from 
the linear part $L(y)$ then $u=0$. Therefore, 
$p^mD_{a0}$ is the only $\iota _0$-contribution from $L(y)$. 
This proves the property $(\bullet _1)$ 
and part $b_{w+1}$) is completely proved. 
\medskip 

Prove part $a_{w+1}$).

Applying $\bar{\c V}^0$ to \eqref{E3.2} 
we obtain the following congruence in $\bar{\c N}^{\dag }[0]$ 
modulo the extension of scalars of the ideal 
$[\bar{\c L}^{\dag [v_0]},\bar{\c L}^{\dag }]+
C_{w+1}(\bar{\c L}^{\dag })$: 
$$\Omega \left (\sigma ^{M-1}\bar{e}_1^{(q)}\right )\equiv 
\sigma (\bar{x})
\circ (\sigma ^{M-1}\bar{e}_1^{(q)})
\circ (-\bar{x}) +
\sum _{\iota \in\mathfrak{A}^+(p)}t_1^{-\iota }\bar{\c V}(A^{\dag }-\id )
D_{\iota 0}\, .$$

Comparing it with our criterion 
$$\Omega \left (\sigma ^{M-1}\bar{e}_1^{(q)}\right )\equiv 
\sigma (\bar X)
\circ \sigma ^{M-1}\bar{e}_1^{(q)}\circ (-\bar X)
\,\op{mod}\,\left ([\bar{\c L}^{[v_0]}, 
\bar{\c L}]+C_{w+1}(\bar{\c L})\right )_{sep}\, $$
we can assume that $\bar X=\bar x+\bar Y$, where 
$\bar Y\in\bar{\c L}^{[v_0]}_{sep }$ and satisfies the following relation 
$$\sigma(\bar Y)-\bar Y=
\sum _{\iota \in\mathfrak{A}^+(p)}t_1^{-\iota }
\bar{\c V}(A^{\dag }-\id )(D^{\dag }_{\iota 0})\,.$$

Choose $\iota $-components $\bar Y_{\iota }$ such that 
$\sigma \bar Y_{\iota }-\bar Y_{\iota }=t_1^{-\iota }
\bar{\c V}(A^{\dag }-\id )(D^{\,\dag }_{\iota 0})$ 
and $\bar Y=\sum _{\iota }\bar Y_{\iota }$. 
As earlier, all $\iota $ give independent contributions to 
upper ramification.  
\medskip 

Take $\iota\in\mathfrak{A}_1^+(p)$, i.e. $\op{ch}(\iota )=1$,    
then $\sigma (\bar Y_{\iota })-\bar Y_{\iota }=
t_1^{-\iota }B^{\,\dag }(D^{\dag }_{\iota 0})$.

If $\iota =p^m(qp^{M-1}\alpha -p^{M-1-A}b^*)$, 
cf.\,Prop.\,\ref{P2.4}b),  
then it gives the upper ramification number 
$q(p^{A}\alpha )-b^*$. On the other hand, 
$\alpha >p^{-A}b^*/q$ is equivalent to 
$\alpha\geqslant p^{-A}v_0$ or $qp^A\alpha -b^*\geqslant qv_0-b^*$. 
So, $\bar{\c V}B^{\dag }(D_{\iota 0})\in\bar{\c L}^{(v_0)}$. 

Therefore, $\bar{\c L}^{(v_0)}+C_{w+1}(\bar{\c L})
\supset \bar{\c L}^{[v_0]}+C_{w+1}(\bar{\c L}^{[v_0]})$.

Theorem \ref{T3.1} is completely proved. 

\subsection{Interpretation in terms of deformations} \label{S3.6}
The above described construction of the ramification ideals 
$\c L^{(v_0)}$ and $\bar{\c L}^{(v_0)}$ 
admits the following interpretation if we replace 
everywhere $t_1$ by $t$. 

We start, as earlier, with the Lie algebra $\c L$  
and $e=e(t)\in \c L_{O_M(\c K)}$ 
such that (for simplicity) 
$e=\sum _{a\in\Z ^0(p)}t^{-a}D_{a0}$. Choose 
$f\in\c L_{sep}$ such that $\sigma f=e\circ f$ and fix the identification 
$\eta ^0: \c G_{<p,M}\simeq G(\c L)$ via $\tau\mapsto (-f)\circ \tau (f)$, 
where 
$\tau\in\c G_{<p,M}$. As earlier, define the filtration $\c L(w)$ and set 
$\bar{\c L}=\c L/\c L(p)$. Then $\bar{\c L}_{W(k)}$ 
is generated by $D_{an}$, where $a\in 
\Z ^0(p)\cap [0,(p-1)v_0)$. If $\bar e\in\bar{\c L}_{O_M(\c K)}$ and 
$\bar f\in\bar{\c L}_{sep}$ are the images of $e$ 
and, resp., of $f$ under the extensions of scalars of the projection  
$\c L\to \bar{\c L}$ then $\tau\mapsto (-\bar f)\circ \tau (\bar f)$ 
determines the projection $\bar\eta ^0:\c G _{<p,M}\to G(\bar{\c L})$. 
Clearly, $\bar{\eta }^0$ is induced by $\eta ^0$. 

As earlier, choose $r^*$, $N^*$, set $q=p^{N^*}$, $b^*=r^*(q-1)$, 
introduce the system of coefficients 
$\{t^{-\iota }\ |\ \iota\in\mathfrak{A}^0\}$ and consider the Lie algebra $\c L^{\dag }$. 
Then $\c L^{\dag }_{W(k)}$ has generators $D^{\dag }_{\iota n}$, where $\iota\in\frak{A}^0(p)$. 
By Proposition \ref{P2.4} all $\iota p^{-v_p(\iota )}$ are pairwise different and, therefore, 
we can define isomorphic embedding $\psi :\c L^{\dag }\to \c L$ by the correspondences 
$D^{\dag }_{\iota n}\mapsto D_{\iota ^*n^*}$, where $\iota ^*=\iota p^{-v_p(\iota )}$ 
and $n^*=n+v_p(\iota )$.  

Let $\psi (\c L^{\dag })=\c L^{o\dag }$. If 
$\op{ch}(\iota )=0$ then $\iota =aqp^{M-1}$, where 
$a\in[0,(p-1)v_0)$, $\iota ^*=a$ and $\psi (D^{\dag }_{\iota n})
=D_{a,n+M-1}$. If $\op{ch}(\iota )\geqslant 1$ then $\iota p^{-v_p(\iota )}>(p-1)v_0$, cf. Prop.\,\ref{P2.4}. In particular, for all such $\iota $, $D_{\iota n}\in\c L(p)_{W(k)}$. 

As earlier, define the filtration on $\c L^{o\dag }$ 
and set $\bar{\c L}^{o\dag }
=\c L^{o\dag }/\c L^{o\dag }(p)$. 
By Prop.\,\ref{P3.4}, if $\op{ch} (\iota )\geqslant 1$ then all 
$D_{\iota n}\in\bar{\c L}(p)_{W(k)}$, and  
we have the induced projection $\bar{\c V}^o:
\bar{\c L}^{o\dag }\to\bar{\c L}$. 

Take 
$$e^{o\dag }=\sum _{\iota\in\mathfrak{A}^+(p)}t^{-\iota }D_{\iota ^*n^*}\,.$$ 
Then the image of $\bar{e}^{o\dag }$ in 
$\bar{\c L}_{O_M(\c K)}$ coincides with $\sigma ^{M-1+N^*}(\bar e)$. 
Therefore, we can choose $f^{\dag }\in\bar{\c L}^{\dag }_{sep}$ 
such that its image in $\bar{\c L}_{sep}$ coincides with 
$\sigma ^{M-1+N^*}\bar f$. This means that 
we have epimorphism $\c G_{<p,M}\to G(\bar{\c L}^{o\dag })$
which gives $\bar{\eta }^0$ when composed with projection 
$G(\bar{\c V^o}): 
G(\bar{\c L}^{o\dag })\to G(\bar{\c L})$. 

Now Theorem \ref{T3.1} states that a unipotent 
$\Z /p^M$-action determined by 
$t^{qp^{M-1}}\mapsto 
t^{qp^{M-1}}\wt{\exp}(-t^{b^*})^{p^{M-1}}$ induces via 
formalism of nilpotent Artin-Schreier theory 
a deformation $B^{o\dag }:\Ker\,\bar{\c V}^o\to \Ker\,\bar{\c V}^o$. 
This deformation allows us to recover the ``substantial'' part $\bar{\c L}^{(v_0)}$ 
of the ramification ideal $\c L^{(v_0)}$. It appears as 
$\bar{\c V}^o(\Ker\,\bar{\c V}^o)$, i.e. as a deformation of 
the ``trivial'' part $\Ker\,\bar{\c V}^o\subset\c L(p)$ via the 
nilpotent operator $B^{o\dag }$. 

Such interpretation also indicates a chance of developing a theory which 
works for all $v_0$ (rather than just for only one single value $v_0$).

\section{Explicit computations} \label{S4}

Now we are going to specify our approach to obtain explicit description 
of generators of the ramification ideal $\bar{\c L}^{[v_0]}$ and to see explicitly 
how the idea of deformation works in our case. 

\subsection{Agreements} \label{S4.1} 
We need to treat the Campbell-Hausdorff 
group law via exponentials in the appropriate 
enveloping 
algebras $\bar{\c A}$ and $\bar{\c A}^{\dag }$ of $\bar{\c L}$ and, 
resp., $\bar{\c L}^{\dag }$, 
by involving extension of scalars 
to $O'_M(\c K')$. 
As a matter of fact, the definition of the action $\Omega $ requires working 
inside $\wt{\c N}^{\dag }$ but can be formally treated in 
$\bar{\c N}^{\dag }$, cf. Prop.\,\ref{P2.14}.  
So, it makes sense to work mainly with the Lie algebra  
$\bar{\c L}^{\dag }[0]\simeq \bar{\c L}$ and 
its enveloping algebra 
$\bar{\c A}^{\dag }[0]$.

We are going also to use the advantage of 
special coefficients from $\mathfrak{A}^0$. 
So, we need just a part of the quotient $(\bar{\c A}^{\dag }/
\bar{\c J}^{\dag p})_{O'_M(\c K')}$ 
generated by 
the elements of $\bar{\c N}^{\dag }[0]$. 
(Here $\bar{\c J}^{\dag }$ 
is the augmentation 
ideal of $\bar{\c A}^{\dag }$.) Even more, 
the ideal $\bar{\c L}^{[v_0]}$ is completely determined by the elements 
$\bar{\c V}B^{\dag }(D^{\,\dag }_{\iota 0})$ 
with $\op{ch}(\iota )=1$. Therefore, it 
would be reasonable to ignore all 
$\iota $ with $\op{ch}(\iota )\geqslant 2$, i.e. 
to work with the part 
generated by 
$\bar{\c N}^{\dag }_{\leqslant 1}[0]$ consisting of 
the preimages of 
$\wt{\c N}^{\,\dag }[0,0]\oplus\wt{\c N}^{\dag }[0,1]$. 

As a result we are going to perform calculations 
inside the $W_M(k)$-module $\bar{\c A}^{\,\dag }_{\leqslant 1}[0]$ 
generated by all 
$p^At_1^{-\iota }D_{\bar a\bar n}$, where for 
$\bar a=(a_1, \dots ,a_r)$, $\bar n=(n_1, \dots ,n_r)$,  
$D_{\bar a\bar n}=D_{a_1n_1}\ldots D_{a_rn_r}$, 
$\iota\in\mathfrak{A}^0$, 
$\op{ch}(\iota )\in\{0,1\}$ and it holds 
$-\iota \geqslant -qp^{M-1}|\bar a|+p^{M-1-A}b^*$ if $\op{ch}(\iota )=1$ 
(recall that $U(1,A)=p^{-A}$). 
Since we do not need to follow the effects related 
to $\op{ch}(\iota )\geqslant 2$, 
we can agree that 
$t_1^{-\iota }\cdot t_1^{-\iota '}=0$ if 
$\op{ch}(\iota )=\op{ch}(\iota ')=1$. 
We do need to go to $\bar{\c A}^{\dag }_{\leqslant 1}[0]$ 
after applying exponential map to 
$\bar{\c N}^{\dag }_{\leqslant 1}$ but almost 
immediately we come back to $\bar{\c N}^{\dag }[0]$ 
and even to $\bar{\c N}^{\dag }[0,1]$. 
\medskip 

%%%%%%%%%%%%%%%%%%%%%%%%%%%%%%%%%%%%%%%%%%%%%%%%%%%5

\subsection{Choice of $e\in\c L_{O_M(\c K)}$} \label{S4.2}  

We keep sufficiently general choice of the element 
$e=e(t)\in {\c L}_{O_M(\c K)}$ used for the  
construction of the identification $\eta ^0:\c G_{<p,M}\simeq G(\c L)$. 
We fix its choice by specifying 
$$E=\wt{\exp}(e) =
1+\sum _{r(\bar a)\geqslant 1}\eta (\bar a)t^{-(a_1+
\dots +a_r)}D_{a_10}\ldots D_{a_r0}\,.$$
The structural constants $\eta (\bar a)\in W_M(k)$ reflect the 
fact that $E$ is a 
``diagonal element'' in ${\c A}/ J^p$ (here $J$ is 
the augmentation ideal of $\c A$). 
Explicitly, this can be stated via the following condition: 
\medskip 

$(*_e)$\   {\it if $0\leqslant s_1\leqslant s<p$ then 
$$\eta (a_1,\dots ,a_{s_1})\eta (a_{s_1+1},\dots ,a_s)=
\sum _{\pi\in I_{s_1s}} \eta (a_{\pi (1)},\dots ,a_{\pi (s)}),$$
where $I_{s_1s}$
consists of all permutations $\pi $ of order $s$ such that the sequences 
$\pi ^{-1}(1),\dots ,\pi ^{-1}(s_1)$ and
$\pi ^{-1}(s_1+1),\dots ,\pi ^{-1}(s)$
are increasing 
(i.e. $I_{s_1s}$ is the set of all \lq\lq insertions\rq\rq\
of the ordered set $\{1,\dots ,s_1\}$ into
the ordered set $\{s_1+1,\dots ,s\}$).}  

Note that in papers \cite{Ab1, Ab2, Ab3} we treated on the first place  
the simplest choice of such constants $\eta (\bar a)$ related to the case  
$e=\sum _{a\in\Z ^0(p)}t^{-a}D_{a0}$. In this case 
the constants are quite simple $\eta (\bar a)=1/r!$, 
where $\bar a=(a_1, \dots ,a_r)$. However, our approach works 
equally well in the above general situation. 

\begin{remark} There is a very interesting possibility to take 
$$e(t)=D_{00}\circ (t^{-1}D_{10})
\circ \ldots \circ (t^{-a}D_{a0})\circ \ldots \, .$$
Then $E=\sum _{r(\bar a)\geqslant 0}\eta (\bar a)t^{-|\bar a|}
D_{\bar a\bar 0}$, where 
$\eta (\bar a)=0$ if 
the condition $0\leqslant a_1\leqslant a_2\leqslant \ldots \leqslant a_r$ 
for $D_{\bar a\bar 0}$ is not satisfied. Note that under this condition 
$\eta (\bar a)=1/(r_1!r_2!\ldots r_s!)$ where 
$\bar a=(\bar a^{(1)}, \ldots ,\bar a^{(s)})$ and 
each of the involved vectors 
$\bar a^{(j)}$, $1\leqslant j\leqslant s$, 
has length $r_j$ and equal components. 
This could be very helpful for studying the ramification breaks 
of the filtration $\{\c L^{(v)}\}_{v\geqslant 0}$. 
\end{remark}

\subsection{Statement of the main result}\label{S4.3} 

\begin{definition} 
 Let $\bar n=(n_1,\dots ,n_s)$ with $s\geqslant 1$. Suppose there is a partition  
 $0=i_0<i_1<\dots <i_r=s$ such that 
 if $i_j<u\leqslant i_{j+1}$ then 
 $n_u=m_{j+1}$ and $m_1>m_2>\dots >m_r$. Then  set  
 $$\eta (\bar a,\bar n)=
 \sigma ^{m_1}\eta (\bar a^{(1)})\dots \sigma ^{m_r}\eta (\bar a^{(r)})\, ,$$
 where $\bar a^{(j)}=(a_{i_{j-1}+1}, \dots ,a_{i_j})$. 
 If such a partition does not exist we set $\eta (\bar a,\bar n)=0$. 
\end{definition}

If $s=0$ we set $\eta (\bar a, \bar n)=1$. 
As earlier,  
set $D_{\bar a\bar n}=D_{a_1n_1}\dots D_{a_sn_s}$.

Note, if  $e_{(N^*,\,0]}:=\sigma ^{N^*-1}(e)\circ 
\sigma ^{N^*-2}(e)\circ \dots \sigma (e)\circ e$ then 
$$\wt{\exp}\,e_{(N^*,\,0]}\equiv \sum _{\bar a, \bar n}
\eta (\bar a,\bar n)t^{-|\bar a|}
D_{\bar a\bar n}\,.$$

For $\gamma\geqslant 0$ and $N\in\Z _{\geqslant 0} $, introduce 
$\c F^0_{\gamma ,-N}\in{\c L}_{W_M(k)}$ such that 
$$\c F^0_{\gamma ,-N}=\sum _{\substack{
\bar a, \bar n\\
\gamma (\bar a,\bar n)=\gamma }}a_1p^{n_1}\eta (\bar a,\bar n)[\dots [D_{a_1 n_1},
D_{a_2 n_2}],\dots ,D_{a_s n_s}]\, .$$

Here: 

--- \ $\bar a=(a_1,\dots ,a_s)$, $\bar n\in\Z ^s$, $n_1\geqslant 0$ and 
all $n_i\geqslant -N$;

--- \ $\gamma =\gamma (\bar a,\bar n)=
a_1p^{n_1}+a_2p^{n_2}+\dots +a_sp^{n_s}$\,.
\medskip 

\begin{remark}
By definition, everywhere 
$D_{a_in_i}:=D_{a_i,n_i\op{mod}N_0}$, but when computing 
$\eta (\bar a, \bar n)$ or 
$\gamma (\bar a,\bar n)$ it is essential that $n_i\in\Z $. 
\end{remark}

Note that the non-zero terms in the above expression for 
$\c F_{\gamma , -N}^0$ can appear only if 
$n_1\geqslant n_2\geqslant\ldots \geqslant n_s$ and $\gamma\in A[p-1,N]$, 
cf. Sect.\,\ref{S2.1}.
\medskip 

The result about explicit generators of $\bar{\c L}^{[v_0]}$ 
can be stated as follows.

Let $\bar{\c F}^0_{\gamma ,-N}$ be the image of $\c F^0_{\gamma ,-N}$ in 
$\bar{\c L}_{W(k)}$.

Let $\iota =p^m(qp^{M-1}\alpha -(q-1)\beta )\in\mathfrak{A}_1^+(p)$ 
be the standard presentation 
from Section\,\ref{S2.2}. Then $|\iota |\leqslant p^{M-1}(p-1)b^*$, 
$p^m\alpha $ and $p^m\beta /r^*$ are integers not both divisible by $p$. 
We have also that  $p^m(q-1)\beta /b^*=p^{m+M-1-A}$ where 
$0\leqslant A\leqslant M-1$, compare to Prop.\,\ref{P2.4}b). 
So, by setting 
$\gamma (\iota ):=p^A\alpha $ and $M(\iota ):=m-A$ we obtain 
the following presentation of $\iota $: 
\begin{equation} \label{E4.1} 
\iota =p^{M-1+M(\iota )}(q\gamma (\iota )-b^*)\,.
\end{equation}

\begin{remark} 
The condition $\iota\leqslant p^{M-1}(p-1)b^*$ 
holds automatically when we work with elements of 
$\bar{\c N}^{\dag }[0]$. Indeed, 
if $t^{-\iota }[D_{\bar a\bar n}]\in\bar{\c N}^{\dag }[0]$ then 
$-\iota \geqslant -qp^{M-1}|\bar a|$. So, 
if $\iota >p^{M-1}(p-1)b^*$ then  
$|\bar a|> (p-1)b^*/q$. By condition {\bf C3} from Section\,\ref{S2.1} 
it is equivalent to $|\bar a|\geqslant (p-1)v_0=v_0U(p-1.0)$ and then  
$[D_{\bar a\bar n}]\in\bar{\c L}(p)_{W(k)}=0$. 
\end{remark}

For any $\iota\in\mathfrak{A}^+_1(p)$, fix 
a choice of $m_{\iota }\geqslant r(\iota )$, where 
$r(\iota )$ is the maximal integer such that 
$\iota p^{r(\iota )}\leqslant p^{M-1}(p-1)b^*$. 

\begin{Thm} \label{T4.1} 
 $\bar{\c L}^{[v_0]}$ is the minimal ideal in $\bar{\c L}$ such that 
for all $\iota\in\mathfrak{A}^+_1(p)$ given with above notation \eqref{E4.1},   
$\bar{\c F}^0_{\gamma (\iota ), -(M(\iota )+m_{\iota })}
\in\bar{\c L}^{[v_0]}_{W(k)}$. 
\end{Thm}

\begin{remark}
 All $\bar{\c F}_{\gamma , -N}$ appear as linear combinations of the 
 Lie monomials $[D_{\bar a\bar n}]$ such that  
 $\gamma =\gamma (\bar a,\bar n)=
 a_1p^{n_1}+\ldots +a_sp^{n_s}$ with $a_i\in\Z ^0(p)$. There 
 is a lot of monomials containing the generators $D_{0n}$ which do not contribute essentially 
 to the condition $\gamma =\gamma (\bar a, \bar n)$. It would be 
 natural to switch from the Lie algebra $\c L$ to its quotient 
 $\c L^o=\c L/(D_0)$ which is related just to the generators  
 $\{D_{an}\ |\ a\in\Z ^+(p)\}$. Then we obtain the description of 
 the corresponding ramification ideal $\c L^{o(v_0)}$ in terms of generators 
 which all contribute (essentially) 
 to the condition $\gamma =\gamma (\bar a,\bar n)$. 
 As a matter of fact, this means that we fixed a lift $\hat\sigma $ 
 of the Frobenius automorphism to $\c G_{<p,M}$ and work with 
 the Galois group of the totally ramified field extension 
 $\c K_{<p,M}^{\langle\hat\sigma \rangle }$. This field appears very naturally in the 
 traditional local class field theory as well as in its nilpotent version \cite{LF1}.
\end{remark}

The proof is given in Sect.\,\ref{S4.4}-\ref{S4.7} below.

\subsection{A recurrent relation} \label{S4.4} 

Consider the part of relation from Section \ref{S3.5} which is 
adjusted to $\bar{\c N}^{\dag }_{\leqslant 1}$ 
\begin{equation} \label{E4.2} \Omega (\sigma ^{M-1}\bar e_1^{(q)})=\sigma (\bar x_1)
\circ \sigma ^{M-1}\bar e_1^{(q)}\circ 
(-\bar x_1)+\qquad\qquad\qquad 
\end{equation} 
$$\qquad\qquad\qquad \sum _{\iota\in\mathfrak{A}_1^+(p)}t_1^{-\iota }
V_{\iota }\,\op{mod}\,[\bar{\c L}^{[v_0]},\bar{\c L}]_{O'_M(\c K')}\,.$$
Here $\bar x_1\in\bar{\c N}^{\dag }[0,1]$ and 
$V_{\iota }=\bar{\c V}^0B^{\dag }(D_{\iota 0}^{\dag })$. 
We know that all $\sigma ^nV_{\iota }$ generate the 
algebra $\bar{\c L}^{[v_0]}_{W(k)}$. 
Clearly, this property is not affected if we replace 
them modulo the ideal $[\bar{\c L}^{[v_0]},\bar{\c L}]_{W(k)}$. 
\medskip 

Let $E_0=\exp (\sigma ^{M-1}\bar e_1^{(q)})$ and 
$\exp (\Omega (\sigma ^{M-1}\bar e_1^{(q)}))=E_0+\c E_0$. 
In particular, $E_0=\sum _{a\in\Z ^0(p)}\sigma ^{M-1}\eta (\bar a)
t_1^{-qp^{M-1}|\bar a|}D_{\bar a,M-1}$ and 
$$\c E_0= 
\sum _{r(\bar a)\geqslant 1}
\sigma ^{M-1}\eta (\bar a)|\bar a|\sum _{A\geqslant 0}p^A
t_1^{-qp^{M-1}|\bar a|+p^{M-1-A}b^*}D_{\bar a,{M-1}}\,.$$

Under the agreement from Section \ref{S4.1} 
we have the following property: 

\begin{Lem} \label {L4.2}
 If $l\in\bar{\c N}^{\dag }_{\leqslant 1}[0]$, 
 $l'\in\bar{\c N}^{\dag }[0,1]$ 
 then 
 $$\exp (l+l')\cdot \exp (-l)-1=l''\in\bar{\c N}^{\dag }[0,1]\, ,$$
 and $l'\equiv l''\,\op{mod}\,[l',\bar{\c L}]_{W(k)}$. 
\end{Lem} 
\begin{proof} 
This follows from a well-known formula for 
$\exp (-X)\exp (X+\delta )$ modulo the terms where $\delta $ 
appears at least twice, cf.\,e.g.\,\cite{BF}, Theorem 4.22. 
\end{proof} 

Applying exponential to \eqref{E4.2} and using  
$\sum _{\iota }t_1^{-\iota }V_{\iota }\in\bar{\c N}^{\dag }[0,1]$ we obtain 
$$E_0+\c E_0\equiv (1+\sigma\bar x_1)E_0(1-\bar x_1)
\left (1+\sum _{\iota }t_1^{-\iota }V_{\iota }\right )\,\op{mod}\,
[\bar{\c L}^{[v_0]},\bar{\c L}]_{O'_M(\c K')}\,.$$
Then using that $\bar x_1,\sigma\bar x_1\in\bar{\c N}^{\dag}[0,1]$ we rewrite 
this relation as follows:  

\begin{equation} \label{E4.3} 
\sum _{\iota\in\mathfrak{A}_1^+(p)}t_1^{-\iota} V_{\iota }
\equiv  E_0^{-1}\c E_0+\bar x_1-E_0^{-1}
\sigma (\bar x_1)E_{0}\,\op{mod}\,[\bar{\c L}^{[v_0]},
\bar{\c L}]_{O'_M(\c K')}\,.
\end{equation}
Note also that by above Lemma \ref{L4.2} 
this relation is still in $\bar{\c L}_{O'_M(\c K')}$.  

%%%%%%%%%%%%%%%%%%%%%%%%%%%%%%%%%%%%%%%%%%%%%%%%%%%

%%%%%%%%%%%%%%%%%%%%%%%%%%%%%%%%%%%%%%%%%%%%%%%%%%%555

\subsection{Solving relation \eqref{E4.3}}\label{S4.5}

 Let $\iota\in\mathfrak{A}^0_1$. Set (as earlier in the context of 
 $\iota\in\frak{A}_1^+(p)$) 

--- \ $r(\iota ):=\max\{r\ |\ \iota p^r \in\mathfrak{A}_1^0\}
=\max\{r\ |\ |\iota p^r|\leqslant p^{M-1}(p-1)b^*\}$. 
\medskip 

--- $m_{\iota }$ to be some fixed integer $\geqslant r(\iota )$. 
\medskip 

--- \ $\frak{A}_1^{\op{prim}}=\frak{A}^0_1\setminus p\frak{A}^0_1$. 
Note that $\frak{A}^+_1(p)=\{\iota\in
\frak{A}^{\op{prim}}_1\ |\ \iota >0\}$. 
\medskip 

For any $j\geqslant 0$, let $E_j=
\sigma ^{j}E_0$ and $\c E_j=\sigma ^j\c E_0$.

Consider the following formal sum  in $\bar{\c A}_{O'_M(\c K')}$
$$\sum _{j\geqslant 0}E_0^{-1}E_1^{-1}
\ldots E_j^{-1}\c E_jE_{j-1}\ldots E_0
=\sum _{\iota ,r}t_1^{-\iota p^r}\bar{\c F}[\iota p^r]\,$$
where the last sum is taken over 
$\iota\in\mathfrak{A}^{\op{prim}}_1$, 
$r\geqslant 0$ and all $\bar{\c F}[\iota p^r]\in\bar{\c A}_{W(k)}$.  

\begin{remark} The elements $\bar{\c F}[\iota p^r]$ play a crucial role in the 
proof of Theorem \ref{T4.1}. Namely, Proposition \ref{P4.4}b) below establishes their 
relation with elements $V_{\iota }$ (which generate the ideal $\bar{\c L}^{[v_0]}$ )
and proposition \ref{P4.7} below expresses $\bar{\c F}[\iota p^r]$ in the form of Lie elements 
$\bar{\c F}^0_{\gamma (\iota ), -M(\iota )+m_{\iota }}$. We are going to use now that 
$\bar{\c F}[\iota p^r]\in\bar{\c L}_{W(k)}$ because the proof of Proposition 
\ref{P4.7} from Sections \ref{S4.6}-\ref{S4.7} does not rely on any constructions 
from this Section.
\end{remark} 

Let 
$\bar{\c F}[\iota p^r]_0:=\bar{\c F}
[\iota p^r]-\sigma\bar{\c F}[\iota p^{r-1}]\in\bar{\c L}_{W_M(k)}$. Then 
$$\bar{\c F}[\iota p^r]=\sum _{0\leqslant j
\leqslant r}\sigma ^j\bar{\c F}[\iota p^{r-j}]_0$$

\begin{Prop} \label{P4.4}  
{\rm a)}\ $\bar x_1\equiv -\sum _{\iota ,r}
 \bar{\c F}[\iota p^r]
 t_1^{-\iota p^r}\,\op{mod}\,\bar{\c L}^{[v_0]}_{O'_M(\c K')}+
 \bar{\c N}^{\dag }(p-1)$, 
 where the sum is taken over all $\iota\in\frak{A}^{\op{prim}}_1$ 
 and $0\leqslant r\leqslant m_{\iota }$; 
 \medskip 
 
 {\rm b)}\ if $\iota\in\frak{A}_1^+(p)$, 
 then 
 $V_{\iota }\equiv 
 \sigma ^{-m_{\iota }}\bar{\c F}[\iota p^{m_{\iota }}]\,
 \op{mod}\,\left [\bar{\c L}^{[v_0]},\bar{\c L}\right ]_{W(k)}$.
\end{Prop} 

\begin{proof} Apply  induction on $1\leqslant w< p$ 
to  prove the proposition modulo commutators 
$C_{w}(\bar{\c L})_{O'_M(\c K')}$ 
for all $1\leqslant w\leqslant p$. 
Clearly, a) holds with $w=1$. 

Suppose $1\leqslant w<p$ and a) holds modulo 
$C_{w}(\bar{\c L})_{O'_M(\c K')}$. 

Rewrite \eqref{E4.3} in the following form 

\begin{equation} \label{E4.4} 
\sum _{\op{ch}\iota = 1}t_1^{-\iota} V_{\iota }+
\sigma{\bar x}_1^1-\bar x_1^1\equiv 
E_0^{-1}\c E_0-
\left (E_0^{-1}\sigma\bar x_1^1E_0-\sigma\bar x^1_1\right )\equiv 
\end{equation}
$$\sum _{r\leqslant m_{\iota }+1, \iota }
\bar{\c F}[\iota p^r]_0t_1^{-\iota p^r}
\,\op{mod}\,\left ([\bar{\c L}^{[v_0]},\bar{\c L}]
+C_{w+1}(\bar{\c L})\right )_{O'_M(\c K')}+\bar{\c N}^{\dag }(p-1)\,.
$$
Since the LHS belongs to $\bar{\c N}^{\dag }[0,1]$ 
the above congruence implies that 
for $r(i)<r\leqslant m_{\iota }+1$, $\bar{\c F}[\iota p^r]_0\in 
\left ([\bar{\c L}^{[v_0]},
\bar{\c L}]+C_{w+1}(\bar{\c L})\right )_{W(k)}$. (Indeed, if 
$t_1^{-\iota p^r}[D_{\bar a\bar n}]\in\bar{\c N}[0,1]$ then 
$qp^{M-1}|\bar a|\geqslant \iota p^r$ implies $|\bar a|>(p-1)v_0$ 
and $[D_{\bar a\bar n}]\in\bar{\c L}(p)_{W(k)}=0$.)

Continue the proof of proposition.

Apply the operators $\c S$ and $\c R$ from Lemma \ref{L2.9} 
to the RHS of \eqref{E4.4}. 

If $\iota <0$ then the expression for the appropriate part of 
$\bar x_1$ will be provided modulo $\left ([\bar{\c L}^{[v_0]},
\bar{\c L}]+C_{w+1}(\bar{\c L})\right )_{O'_M(\c K')}+\bar{\c N}^{\dag }(p-1)$ with  
$$\c S\left (\sum _{r}\bar{\c F}
[\iota p^r]_0t_1^{-\iota p^r}\right )=
-\sum _{j\geqslant 0}\sigma ^j\left (\sum _{\iota ,r}
\bar{\c F}[\iota p^r]_0t_1^{-\iota p^r}\right )=
-\sum _{\iota , r\leqslant m_{\iota }}\bar{\c F}[\iota p^r]t^{-\iota p^r}\,.$$
Note that the elements from 
$\bar{\c N}^{\dag }[0](p-1)\subset \bar{\c L}_{\m '}$ will not 
contribute to 
$\sum _{\iota\in\mathfrak{A}_1^+(p)}t^{-\iota }V_{\iota }$. 
This proves part a) in the case $\iota <0$. 
\medskip 

For $\iota >0$,   
$V_{\iota }t_1^{-\iota }$ is congruent modulo 
$\left ([\bar{\c L}^{[v_0]},\bar{\c L}]
+C_{w+_1}(\bar{\c L})\right )_{O'_M(\c K')}$ to  
$$\c R\left (\sum _{0\leqslant r\leqslant m_{\iota }}
\bar{\c F}[\iota p^r]_0t_1^{-\iota p^r}\right )
\equiv 
\sum _{0\leqslant r\leqslant m_{\iota }}\sigma ^{-r}\bar{\c F}
[\iota p^r]_0t_1^{-\iota }\equiv 
-\sigma ^{-m_{\iota }}\bar{\c F}[\iota p^{m_{\iota }}]t_1^{-\iota }\,.$$ 
This proves part b) modulo $C_{w+1}(\bar{\c L})_{W(k)}$. 

In this case $\bar x_1$ will be contributed modulo 
$\bar{\c L}^{[v_0]}_{O'_M(\c K')}$ with  
$$\c S\left (\sum _{0\leqslant r\leqslant m_{\iota }}
\bar{\c F}[\iota p^r]_0t_1^{-\iota p^r}\right )
=\sum _{0\leqslant r\leqslant m_{\iota }}
\sum _{0\leqslant r_1<r}\sigma ^{r_1}\left (t_1^{-\iota }
\sigma ^{-r}\bar{\c F}[\iota p^r]_0\right )
$$
$$
\equiv -\sum _{0\leqslant r\leqslant m_{\iota }}
\sum _{r\leqslant r_1\leqslant m_{\iota }}\sigma ^{r_1}\left (t_1^{-\iota }
\sigma ^{-r}\bar{\c F}[\iota p^r]_0\right )
$$

$$\equiv 
-\sum _{0\leqslant r_1\leqslant m_{\iota }}t_1^{-\iota p^{r_1}}\sigma ^{r_1}\left (
\sum _{r_1\leqslant r\leqslant m_{\iota }}
\sigma ^{-r}\bar{\c F}[\iota p^r]_0\right )
\equiv 
$$

$$\equiv 
-\sum _{0\leqslant r_1\leqslant m_{\iota }}t_1^{-\iota p^{r_1}}
(\bar{\c F}[\iota p^{r_1}]_0+\sigma ^{-1}\bar{\c F}[\iota p^{r_1+1}]_0+
\ldots +\sigma ^{r_1-m_{\iota }}\bar{\c F}[\iota p^{m_{\iota }}]_0)
\equiv 
$$
$$
-\sum _{0\leqslant r\leqslant m_{\iota }}\bar{\c F}[\iota p^r]t_1^{-\iota p^r}$$
because 
$$\sum _{0\leqslant r_1,r\leqslant m_{\iota }}
\sigma ^{r_1}\left (t_1^{-\iota }\sigma ^{-r}\bar{\c F}[\iota p^r]_0\right )
=\sum _{r_1}t_1^{-\iota p^{r_1}}\sigma ^{r_1}\left (\sum _{0\leqslant r\leqslant m_{\iota }}
\sigma ^{-r}\bar{\c F}[\iota p^r]_0\right )\in\bar{\c L}^{[v_0]}_{O'_M(\c K')}$$
\end{proof}

%%%%%%%%%%%%%%%%%%%%%%%%%%%%%%%%%%%%%%%%%%%%%%

%%%%%%%%%%%%%%%%%%%%%%%%%%%%%%%%%%%%%%%%%%%%%%%%5

\subsection{A combinatorial statement}\label{S4.6} 

Recall that  the values $\eta (\bar a, \bar n)$ were defined in Section\,\ref{S4.3}. 
If  
$e_{(N^*,0]}:=\sigma ^{N^*-1}e\circ 
\sigma ^{N^*-2}e\circ \dots \circ \sigma e\circ e$ 
then 
$$\wt{\exp}\,e_{(N^*,0]}=E_{N^*-1}\ldots E_1 E_0
=\sum _{\bar a, \bar n}\eta (\bar a,\bar n)
t_1^{-qp^{M-1}\gamma (\bar a,\bar n)}
D_{\bar a\bar n}\in\bar{\c L}_{O'_M(\c K')}\, .$$
As earlier, for all $i\geqslant 0$, 
$E_i=\sigma ^i(E_0)$ and   
$\gamma (\bar a, \bar n)=a_1p^{n_1}+\ldots +a_sp^{n_s}$. 
\medskip

Let $-e_{[0,N^*)}:=(-e)\circ (-\sigma e)\dots (-\sigma ^{N^*-1}e)$. Introduce 
similarly the constants $\eta ^o(\bar a,\bar n)\in W_M(k)$ 
by the following identity  
$$\wt{\exp}(-e_{[0,N^*)})=\sum \eta ^o(\bar a,\bar n)D_{\bar a\bar n}\, .$$

Set $\eta ^o(\bar a):=\eta ^o(\bar a,\bar 0)$. 
It can be easily seen that if 
there is a partition from the 
definition of $\eta $-constants in Sect,\,\ref{S4.3} 
but such that  $m_1<m_2<\dots <m_r$, then 
$$\eta ^o(\bar a, \bar n)=\sigma ^{m_1}
\left (\eta ^o(\bar a^{(1)})\right )\sigma ^{m_2}
\left (\eta ^o(\bar a^{(2)})\right )
\dots \sigma ^{m_r}\left (\eta ^o(\bar a^{(r)})\right )\, .$$
Otherwise,  $\eta ^o(\bar a,\bar n)=0$.  
\medskip 

If we don't care what vectors  $\bar a$ and $\bar n$ are 
we can agree just to write  $\eta (1,\dots ,s)$ instead of 
$\eta (\bar a,\bar n)$ and use similar agreement for $\eta ^o$. 
E.g. the equality  
$e_{(N^*,0]}\circ (-e_{[0,N^*)})=0$  
means that 
\begin{equation} \label{E4.5} 
\sum _{0\leqslant s_1\leqslant s}
\eta (1,\dots ,s_1)\eta ^o(s_1+1,\dots ,s)=
\delta _{0s}\,.
\end{equation}

 For $1\leqslant  s_1\leqslant s<p$, 
consider the subset $\Phi _{ss_1}$ of
permutations $\pi $ of order $s$ such that $\pi (1)=s_1$
and for any $1\leqslant l\leqslant s$,
the subset $\{\pi (1),\dots ,\pi (l)\}\subset [1,s]$ is \lq\lq connected\rq\rq ,
i.e.\, for some  $n(l)\in\N $, 
$$\{\pi (1),\dots ,\pi (l)\}=\{n(l),n(l)+1,\dots ,n(l)+l-1\}.$$
By definition, $\Phi _{s0}=\Phi _{s,s+1}=\emptyset $.

Set $B_{s_1}(1,\dots ,s)=\sum_{\pi\in\Phi _{ss_1} }
\eta (\pi (1),\dots ,\pi (s))$. Then, cf.\,\cite{Ab15},\,Sect.4,  

\begin{Lem} \label{L4.6} Suppose $0\leqslant s_1\leqslant s<p$. Then  
\medskip 

{\rm a)}\ $B_{s_1}(1,\dots ,s)+B_{s_1+1}(1,\dots ,s)=
\eta (s_1,\dots ,1)
\eta (s_1+1,\dots ,s)$;
\medskip 

{\rm b)}\ $\eta ^o(1,\dots ,s)=(-1)^s\eta (s,s-1,\dots ,1)$;
\medskip 

{\rm c)}\ for indeterminates $X_1,\dots ,X_s$, 
$$\sum_{\substack{1\leqslant s_1\leqslant s \\ \pi\in\Phi _{ss_1}}} 
(-1)^{s_1-1}X_{\pi ^{-1}(1)}\dots X_{\pi ^{-1}(s)}=
[\dots [X_1,X_2],\dots ,X_s].$$
\end{Lem}

%%%%%%%%%%%%%%%%%%%%%%%%%%%%%

\subsection{Lie elements $\bar{\c F}[\iota ]$ 
and $\bar{\c F}[\iota ]_0$} \label{S4.7} 

Agree to write $\bar n=(n_1,\dots ,n_s)\geqslant M$ if all  
$n_i\geqslant M$, and treat similarly $
\bar n>M$, $\bar n\leqslant M$ and $\bar n<M$. 
\medskip 

 For $1\leqslant s_1\leqslant s$, let 
$\gamma ^*_{[s_1,s]}(\bar a,\bar n)=
\sum\limits _{s_1\leqslant u\leqslant s}a_up^{n^*_u}$  
where $n_u^*=0$ if $n_u=M(\bar n):=\max\{n_1,\dots ,n_s\}$ and 
$n_u^*=-\infty $  (i.e. $p^{n_u^*}=0$), otherwise. 
\medskip 

 Suppose $\iota\in\mathfrak{A}^0_1$ is given as in 
 \eqref{E4.1}. Then 
$$\bar{\c F}[\iota ]=\sum _{\bar a,\bar n, A}
\sum _{1\leqslant s_1\leqslant s}
 p^A\eta ^o(1,\dots ,s_1-1)\eta (s_1,\dots ,s)
 \gamma _{[s_1,s]}^{*}(\bar a,\bar n)D_{\bar a\bar n}$$ 
Here the first sum is taken over 
 all $(\bar a,\bar n)$ of lengths $1\leqslant s<p$ such that 
 $\bar n\geqslant 0$,  
 $\gamma (\bar a,\bar n)=p^{M(\iota )}\gamma (\iota )$ and 
 $M(\bar n)=M(\iota )+A$. 
 
 The following proposition relates the elements $\bar{\c F}[\iota ]$ and the elements 
 $\bar{\c F}_{\gamma ,-N}$ from Section\,\ref{S4.3}. 
 \medskip

\begin{Prop} \label{P4.7}
If $\iota $ is given in notation \eqref{E4.1}  and $n\geqslant 0$ 
then 
$$\bar{\c F}[\iota p^n]=\sigma ^{M(\iota )+n}
\bar{\c F}^0_{\gamma (\iota ), -(M(\iota )+n)}\,.$$
\end{Prop}

\begin{proof} We have 
$$\sigma ^{-(M(\iota )+A+n)}\bar{\c F}[\iota p^n]=
\sum _{\substack{1\leqslant s_1\leqslant s<p\\ (\bar a,\bar n)}}p^A
\eta ^o(1,\dots ,s_1-1)\eta (s_1,\dots ,s)
\gamma _{[s_1,s]}^{*}(\bar a,\bar n)D_{\bar a,\bar n}\, ,$$ 
where the sum is taken for $(\bar a,\bar n)$ with $M(\bar n)=A$, 
$\bar n\geqslant -(M(\iota )+n)$  
and $\gamma (\bar a,\bar n)=\gamma (\iota )=p^A\alpha $. 
By Lemma \ref{L4.2},   
$$\eta ^o(1,\dots ,s_1-1)=(-1)^{s_1-1}\eta (s_1-1,\dots ,1)\,$$  
and we obtain 
$$\sum _{\substack{1\leqslant s_1\leqslant s<p \\ (\bar a,\bar n) }}
(-1)^{s_1-1}(B_{s_1-1}(1,\dots ,s)+
B_{s_1}(1,\dots ,s))p^A\gamma ^*_{[s_1,s]}(\bar a,\bar n)D_{(\bar a\bar n)}=$$

$$\sum _{\substack{1\leqslant s_1\leqslant s<p \\ (\bar a,\bar n)}}
(-1)^{s_1-1}B_{s_1}(1,\dots ,s)p^A(\gamma ^*_{[s_1,s]}(\bar a,\bar n)-
\gamma ^*_{[s_1+1,s]}(\bar a,\bar n)){D}_{(\bar a,\bar n)}=$$

$$\sum_{\substack{1\leqslant s<p \\ (\bar a,\bar n)}}
\sum_{1\leqslant s_1\leqslant s} (-1)^{s_1-1}
B_{s_1}(1,\dots ,s)p^Aa_{s_1}p^{n_{s_1}^*}{D}_{(\bar a\bar n)}=$$

$$\sum_{\substack{1\leqslant s<p \\ (\bar a,\bar n)}}
\sum_{\substack{1\leqslant s_1\leqslant s\\ \pi\in\Phi _{ss_1}}}
(-1)^{s_1-1}\eta ({\pi (1)},\dots ,{\pi (s)})p^A
a_{s_1}p^{n_{s_1}^*}D_{(\bar a,\bar n)}=$$

$$ \sum_{\substack{1\leqslant s<p \\ (\bar a,\bar n)}}
\eta (1,\dots ,s)a_1p^A
\sum_{\substack{1\leqslant s_1\leqslant s\\ \pi\in\Phi _{ss_1} }}
(-1)^{s_1-1}D_{a_{\pi ^{-1}(1)}n_{\pi ^{-1}(1)}}\dots
D_{a_{\pi ^{-1}(s)}n_{\pi ^{-1}(s)}} =$$
$$\sum _{\substack{1\leqslant s<p\\(\bar a,\bar n)}}
\eta (1,\dots ,s)a_1p^{n_1}[\dots [D_{a_1n_1},D_{a_2n_2}],\dots ,D_{a_sn_s}]
=\bar{\c F}^0_{\gamma (\iota ) ,-(M(\iota )+n)}\,.$$
because $M(\bar n)=n_1$. The proposition is proved. 
\end{proof} 

According to remark from Section \ref{S4.5} this finishes the proof of Theorem \ref{T4.1} 
(we should use $n=m_{\iota }$ in the last Proposition). 

\subsection{Description of $\bar{\c L}^{[v_0]}$  from \cite{Ab3}}  \label{S4.8}
We are going to use Theorem\,\ref{T4.1} to recover the description of 
$\bar{\c L}^{[v_0]}$ obtained 
in \cite{Ab3} in the context of Lie algebra $\bar{\c L}$.

\begin{Thm} \label{T4.9}
Let $\wt{N}=N^*-1$.  Then the  ideal $\bar{\c L}^{[v_0]}$ is the 
minimal ideal in $\bar{\c L}$ 
 such that $\bar{\c L}_{W_M(k)}^{[v_0]}$ contains all 
 $\bar{\c F}^0_{\gamma ,-\wt{N}}$ with $\gamma\geqslant v_0$.
\end{Thm}

\begin{proof} By Theorem\,\ref{T4.1},  $\bar{\c L}^{[v_0]}$ 
is the minimal ideal in $\bar {\c L}$ such that 
$\bar{\c L}^{[v_0]}_{W(k)}$ contains all 
$\bar{\c F}[\iota p^{m_{\iota }}]$, 
where $\iota \in\frak{A}_1^+(p)$ and $m_{\iota }
\geqslant 0$ is chosen such that 
$\iota p^{m_{\iota }}\notin \mathfrak{A}^0$. 

By \eqref{E4.1},  
$$\iota =p^m(qp^{M-1}\alpha -p^{M-1-A}b^*)=
p^{M(\iota )+M-1}(q\gamma (\iota )-b^*)
\in\mathfrak{A}_1^+(p)\,,$$
where $\gamma (\iota )=p^A\alpha $, $\alpha\in A[p-1,m]$ 
and $M(\iota )=m-A$. 
Note that $\iota\in\mathfrak{A}^+_1(p)$ means mainly 
that $\iota \leqslant p^{M-1}(p-1)b^*$. 
In this case by Prop.\,\ref{P4.7}, $\bar{\c F}[\iota p^{m_{\iota }}]$ 
can be replaced by 
$\bar{\c F}^0_{\gamma (\iota ),-\wt{N}}$ because 
for some $m_{\iota }\geqslant r(\iota )$, $\wt{N}=
N^*-1={M}(\iota )+m_{\iota }$. 
Note also that $\gamma (\iota )\in A[p-1, M(\iota )]$. 
\medskip 

As a result, it remains to verify the following proposition. 
\end{proof}

\begin{Prop}\label{P4.10} Suppose $\gamma\geqslant v_0$ 
is such that 
for any $\iota\in\mathfrak{A}^+_1(p)$, 
$\gamma \ne\gamma (\iota )$ and 
$\gamma \in A[p-1,\wt{N}]$. Then 
$\bar{\c F}^0_{\gamma ,-\wt{N}}\in\bar{\c L}^{[v_0]}_{W(k)}$. 
\end{Prop}

\begin{proof} 
We can assume that 
$p^{M-1+\wt{N}}(q\gamma -b^*)>p^{M-1}(p-1)b^*\,.$

The expression for $\bar {\c F}^0_{\gamma ,-\wt{N}}$ 
is a linear 
combination of the terms  
$$p^{n_1}[\dots [D_{a_1n_1},D_{a_2n_2}],\dots ,D_{a_rn_r}]\,,$$ 
where $M>n_1\geqslant n_2\geqslant 
\ldots \geqslant n_r\geqslant -\wt{N}$  
and $p^{n_1}a_1+\dots +p^{n_r}a_r=\gamma $. 
We can assume that all $a_i<(p-1)v_0$, because otherwise 
$\bar{\c F}^0_{\gamma ,-\wt{N}}\in\bar{\c L}(p)_{W(k)}$ and 
$\bar{\c F}^0_{\gamma ,-\wt{N}}=0$. 

Note that 
$\bar{\c F}^0_{\gamma ,-\wt{N}}$ can be presented as a linear 
combination of the terms 
$$[\dots [\bar{\c F}^0_{\gamma ', -(\wt{N}-1)},
D_{a_1',-\wt{N}}],\dots ,D_{a_r',-\wt{N}}]$$
where $\gamma =\gamma '+(a_1'+\dots +a_r')/p^{\wt{N}}$, 
$r\geqslant 0$, all $a'_i\in\Z ^0(p)$ 
and $\gamma '\in A[p-1,\wt{N}-1]$. 

If $a_1'+\dots +a_r'\geqslant (p-1)b^*/q$, i.e. 
it is $\geqslant (p-1)v_0$, then the whole term belongs to 
$\bar{\c L}(p)_{W(k)}=0$. 
Otherwise, the inequality 
$$b^*(p-1)/p^{\wt{N}}<q\gamma -b^*=
q\gamma '-b^*+q(a_1'+\dots +a'_r)/p^{\wt{N}}$$
implies that 
$q\gamma '-b^*>0$ and, therefore, $\gamma '\geqslant v_0$. 
If $\gamma '=\gamma (\iota ')$ for a suitable 
$\iota '\in\mathfrak{A}^+_1(p)$ then  
$\bar{\c F}^0_{\gamma ', -(\wt{N}-1)}\in\bar{\c L}^{[v_0]}$ 
and we are done. Otherwise, we should repeat 
the procedure with $\gamma '$ instead of $\gamma $. 

As a result we either prove our proposition or 
will come to the case where we have $\bar{\c F}_{\gamma ,0}$ with 
$q\gamma -b^*>p^{M-1}(p-1)b^*$. Proceed as earlier:

Let $\gamma =\gamma '+(a_1'+\ldots +a'_r)$, $r\geqslant 0$ 
and $v_p(\gamma ')\geqslant 1$. Then 
$$q\gamma -b^*=q\gamma '-b^*+q(a_1'+\ldots +a'_r)
>(p-1)b^*\,.$$
Again, if $a_1'+\ldots +a'_r\geqslant 
(p-1)v_0$ we are in $\bar{\c L}(p)_{W(k)}$. Otherwise, 
$q\gamma '-b^*>0$ implies $\gamma '\geqslant v_0$ and 
we have the following alternative: $\gamma '=\gamma (\iota ')$ 
with $\iota '\in\mathfrak{A}_1^+(p)$ or should repeat the procedure 
with $\bar{\c F}_{\gamma ', 1}\in p\bar {\c L}_{W(k)}$. It remains to note 
that $\bar{\c F}_{\gamma ,M}=0$. 
\end{proof}

%\end{document}

%%%%%%%%%%%%%%%%%%%%%%%%%%%%%%%%%%%%%%%%%%%%%%%%%%%%%%%%%%%

%%%%%%%%%%%%%%%%%%%%%%%%%%%%%%%%%%%%%%%%%%%%%%%%%%%%%%%%%%%5


\begin{thebibliography}{xxx}

\bibitem{Ab1} {\sc V.\,A.\,Abrashkin}, 
\textit{Ramification filtration of the Galois
group of a local field},  {Proceedings of the St.
Petersburg Mathematical Society, vol. III (1995), 35--100}, Amer.
Math. Soc. Transl. Ser. 2,  {\bf 166} (1995), Amer. Math. Soc., Providence, RI

\bibitem{Ab2} {\sc V.\,A. \,Abrashkin}, 
\textit{Ramification filtration of the Galois group of a local field. II}, 
{Proceedings of Steklov Math. Inst.} {\bf 208} (1995), 18--69 

\bibitem{Ab3} {\sc V.\,Abrashkin}, 
\textit{Ramification filtration of the Galois group of a local field. III}, 
{Izvestiya RAN: Ser. Mat.}, {\bf 62} (1998), no.\,5, 3--48;  
English transl. 
{Izvestiya: Mathematics} {\bf 62} (1998), no.\,5, 857--900


%\bibitem{Ab4} {\sc V.\,A.\, Abrashkin},   
%\textit{A group-theoretical property of the ramification filtration},    
%{Izvestiya RAN: Ser. Mat.}, {\bf 62} (1998), no.\,6, 3--26;  
%English transl.  
%{Izvestiya: Mathematics} {\bf 62} (1998), no.\,6, 1073--1094  


%\bibitem{Ab5}
%{\sc V.\,Abrashkin},  
%\textit{Report on the ramification filtration of the Galois group of a local field},   
%{Proceedings of the Research Conference on ``Number Theory and Arithmetical 
%Geometry: Arithmetical applications of Modular Forms'' (San Feliu de Guixols, 
%Spain, 24-29 October, 1997)}
%Inst. fur Exp. Math. Universitat Essen, 1998, 47--53 

\bibitem{Ab6} {\sc V.\,Abrashkin}, 
\textit{On a local analogue of the Grothendieck Conjecture}, 
{Int. J. Math.}  {\bf 11} (2000), no.\,1,  3--43


%\bibitem{Ab7} {\sc V. \,Abrashkin},  
%\textit{Ramification theory for higher dimensional fields},  
%{Contemp. Math.}  {\bf 300} (2002), 1--16 


\bibitem{Ab8} {\sc V. \,Abrashkin},  
\textit{Characteristic $p$ case of the Grothendieck conjecture for 
2-dimensional local fields},  
{Proceedings of Steklov Math. Institute}  {\bf 241} (2003), 1--35

\bibitem{Ab10} {\sc V.\,Abrashkin}, \textit{An analogue 
of the fields-of-norms functor 
and the Grothendieck Conjecture}, J. of 
Algebraic Geometry {\bf 16} (2007), \,671--730

%\bibitem{Ab9} {\sc V.\,Abrashkin}, 
%\textit{Characteristic $p$ analogue of modules with finite crystalline height},
%{Pure Appl. Math. Q.},  {\bf5} (2009), 469--494

\bibitem{Ab11} {\sc V.\,Abrashkin}, 
\textit{Modified proof of a local analogue of the Grothendieck
Conjecture}, Journal Th\'e orie des Nombres de Bordeaux {\bf 22} (2010), 1--50

\bibitem{Ab9} {\sc V.Abrashkin}, 
\textit{Galois groups of local fields, Lie algebras and ramification}. 
In: Arithmetic and Geometry, eds. Dieulefait, L., Faltings, G., 
Heath-Brown, D.R., Manin, Yu., Moroz, B.Z., Wintenberger, J.-P., 
Cambridge University Press  {\bf 420}, 1--23


%\bibitem{AJ} {\sc V.Abrashkin, R.Jenni} 
%\textit{The field-of-norms functor and the Hilbert symbol for higher local fields},  
%{Journal Th\' eorie des Nombres Bordeaux}, {\bf 24} (2012), no.\,1, 1--39 



\bibitem{Ab12} {\sc V.\, Abrashkin}, 
\textit{Groups of automorphisms of local fields of period $p$ 
and nilpotent class $<p$, I}, {Int. J. Math.} {\bf 28} (2017), 1750043 

\bibitem{Ab13} {\sc V.\, Abrashkin}, 
\textit{Groups of automorphisms of local fields of period $p$ 
and nilpotent class $<p$, II}. {Int. J. Math.}  {\bf 28} (2017), 1750066

\bibitem{Ab14} {\sc V.\, Abrashkin}, 
\textit {Groups of automorphisms of local fields of period $p^M$ 
and nilpotent class $<p$}, {Ann. Inst. Fourier}  {\bf 67} (2017), no.\,2, 605--635

\bibitem{Ab15} {\sc V.\, Abrashkin}, 
\textit {Ramification filtration and deformations}, 
{Sbornik: Mathematics}  {\bf 212} (2021), no.\,2, 135--169

\bibitem{Ab16} {\sc V.\, Abrashkin}, 
\textit {Ramification filtration and differential forms}, 
{Izv.\,Math.\,ser. matem.}  {\bf 87} (2023), no.\,3, 421-438


%\end{thebibliography}
%\end{document}

\bibitem{AH} {\sc E.\,Artin, H.\,Hasse}, 
\textit{Die beiden Erg{\" a}nzungss\" atze zum Reziprozit{\" a}tsgesetz 
der $l^n$-ten Potenzreste im K{\" o}rper der $l^n$-ten Einheitswurzeln}, 
{Abhandlungen aus dem Mathematischen Seminar der Universit\" at Hamburg}, 
(1928) {\bf 6}: 146–162 


%\end{thebibliography}
%\end{document}


%\bibitem{BV} {\sc D.G.\,Benois, S.V.\,Vostokov}, 
%\textit{On $p$-Extensions of Multidimensional Local Fields}, 
%Math.\,Nachr.\,{\bf 160} (1993), 59--68

%\bibitem{Be} {\sc P.\,Berthelot} \textit{   }  


%\bibitem{Bth} {\sc P.\,Berthelot,\,W.\,Messing}, 
%\textit{Th\'eorie de Deudonn\'e Cristalline III: 
%Theor\'emes d'\'Equivalence et de Pleine Fidelit\'e}, 
%The Grothendieck Festschrift (P.Cartier etc., eds.), 
%A Collection of Articles Written in Honour of 60th 
%Birthday of Alexander Grothendieck, 
%vol. 1, Birkhauser, 1990, 173--247.

%\bibitem{Bg} {\sc F. Bogomolov, Y. Tschinkel}, 
%\textit{Commuting elements in Galois groups of 
%function fields}, Motives, Polylogarithms and Hodge theory, 2002, 75--120  

\bibitem{BF} {\sc A.Bonfiglioli, R.Fulci}, 
\textit{ Topics in Noncommutative Algebra}, 
Lecture Notes in Mathematics 2034, Springer-Verlag 
Berlin heidelberg 2012

%\bibitem{Bou} {\sc N.\,Bourbaki}, \textit{Elements of Mathematics. 
%Lie Groups and Lie Algebras. Chapters 1-3}, 
%Springer-Verlag Berlin Heidelberg 1989, 461 pp.


\bibitem{Br} {\sc H.\,Br\" uckner} (1979) 
\textit{Explizites Reziprozitatsgesetz und Anwendungen.} Vorlesungen aus dem Fachbereich Mathematik der Universit\" at Essen, {\bf 2}, Universit\" at Essen, Fachbereich Mathematik, Essen, 1979

\bibitem{CW} {\sc J.\,Coates, A.\,Wiles}, 
\textit{Explicit reciprocity laws}, {Ast\' erisque}, v.\,41-42 (1977), 7-17


%\end{thebibliography}
%\end{document}

%\bibitem{CB} {\sc Colas Bordavid}, 
%\textit{Profinite completion and double-dual: 
%isomorphisms and counter-examples}, hal-00208000, version 1 - 18 Jan 2008

%\bibitem{De} 
%{\sc P.Deligne },  
%\textit{Les corps locaux de caract\'eristique $p$, limites
% de corps locaux de caract\'eristique $0$}, 
%{Representations of reductive groups over a local field, Travaux en cours},  
%Hermann, Paris, 1973, 119-157 

%\bibitem{Fe} {\sc I.\,Fesenko} {Invitation}

\bibitem{Fo} {\sc J.-M.\,Fontaine}, \textit{Representations $p$-adiques 
des corps locaux}, {The Grothendieck Festshrift. A collection of articles 
written in honor of 60th birthday of Alexander Grothendieck.} 
Eds. P.\,Cartier et al,: Basel: Birkhauser, 1990. Vol.2, 249-310



%\bibitem{Go1} {\sc N.L.Gordeev}, \textit{Ramification groups of 
%infinite $p$-extensions
%of a local field},  {Soviet Math.}
%{\bf 20} (1982), no.6, 
%Zapiski seminarov LOMI, 1977, 96-132

%\bibitem{Go2} {\sc N.L.Gordeev},  {Infinity of the
%number of relations in the
%Galois group of the maximal $p$-extension of a local
%field with bounded ramification}, 
%{Izvestia AN SSSR, Ser. Matem.}  {\bf 45} (1981), 592-607 




%\bibitem{Ha} {\sc M.\,Hall} {The theory of groups}, 
%The Macmillan Company New York, 1959

\bibitem{JW} {\sc U.Jannsen, K.Wingberg}, 
\textit{Die Struktur der absoluten Galoisgruppe
$p$-adischer Zahlk\" orper},  {Invent.\, math.} 
{\bf 70} (1982), 71--98

%\bibitem{Fe1} {\sc I.\,Fesenko}, {Multidimensional local 
%class field theory II}, 
%Algebra i Analiz (1991); English translation 
%in St. Petersburg Math.\,J.\,{\bf 3} (1992), 1103--1126

%\bibitem{Fe0} {\sc I.\,Fesenko}, 
%{Abelian local $p$-class field theory}, 
%Math. Annalen \,{\bf 301} (1995), no.\,3, 561--586

%\bibitem{Fe2} {\sc I.\,B.\,Fesenko}, {Abelian extensions 
%of complete discrete 
%valuation fields}, Number Theory Paris 1993/94, 
%Cambridge Univ.\,Press, 1996, 47-74

%\bibitem{Fe3} {\sc I.\,B.\,Fesenko}, 
%{Parshin’s higher local class field theory 
%in characteristic p}, 
%(M\"unster, 1999), Geom. Topol. Monogr., (2000), no.3, 75--79


\bibitem{He} {\sc G.\,Henniart}, 
\textit {Sur les lois de r\' eciprocité explicites. I.} {J. Reine Angew. Math.} 
{\bf 329} (1981), 177–203 


\bibitem{Iw} {\sc K.\,Iwasawa}, \textit{On explicit formulas 
for the norm residue symbol},  
{J. Math.\,Soc.\,Japan.}, {\bf 20} (1968), 151–165 




%\bibitem{Ka1} {\sc K.\,Kato}, \textit{A generalization of 
%local class field theory 
%by using $K$-groups I},  
%{J.\,Fac.\,Sci.\,Univ.\,Tokyo Sec.\,IA}\, 
%26\,No.\,2\,(1979),\,303--376

%\bibitem{Ka2} {\sc K.\,Kato}, 
%\textit{A generalization of local class field theory},  
%{J.\,Fac.\,Sci.\,Univ.\,Tokyo Sec.\,IA}\, 
%27\,No.\,3\,(1980),\,603--683

%\bibitem{Ka3} {\sc K.\,Kato}, \textit{Existence theorem for 
%higher local class field theory},  
%(M\"unster, 1999), Geom. Topol. Monogr., (2000), no.3, 165--195



%\bibitem{Ko} {\sc H.Koch}, \textit{Galois 
%theory of $p$-extensions}, 
%Springer Monographs in Mathematics, 2002, XIII, 191 pp.

\bibitem{KKL} \sc{H.\,Koch, S.\,Kukkuk, J.\,Labute} 
\textit{Nilpotent local class field theory}, {Acta Arithmetica} 
LXXXIII.1 (1998), 45-64

\bibitem{KdS} {\sc H.Koch, E.\,de\,Shalit}, 
\textit{Metabelian local class field theory},  
{J. Reine Angew. Math.}  {\bf 478}, (1996), 85-106

%\bibitem{LF} {\sc F.\,Laubie} 
%\textit{Extensions de Lie et groupes 
%d'automorphismes de corps locaux}, 
%{Comp. Math.}, {\bf 67} (1988), 165-189

\bibitem{LF1} {\sc F.Laubie} 
\textit{Une th\' eorie du corps de classes 
local non ab\' elien},  
{Compos. Math.}, {\bf 143} (2007), no. 2, 339--362.

\bibitem{La} {\sc  M. Lazard},  \textit{Sur
les groupes nilpotentes et les anneaux de Lie}, 
{Ann. Ecole Norm. Sup.} {\bf 71} (1954), 101--190

\bibitem{Le} {\sc F.\,Lemmermeyer} 
\textit{Reciprocity laws. From Euler to Eisenstein.} 
Springer Monographs in Mathematics. Berlin: Springer-Verlag, 2000

%\bibitem{MZh} {\sc A.\,I.\,Madunts, I.\,B.\,Zhukov}, 
%\textit{Multidimensional complete fields: 
%topology and other basic 
%constructions},  {Amer. Math. Soc. Transl.}, 
%Ser.2, {\bf165} (1995), 1--34


%\end{thebibliography}
%\end{document}


\bibitem{Mo} {\sc Sh.\,Mochizuki},  
\textit{A version of the Grothendieck conjecture for $p$-adic 
local fields}, {Int. J. Math.} {\bf 8} (1997),  no.4,  499--506


%\end{thebibliography}
%\end{document}

\bibitem{Ne} {\sc J.\,Neukirch}, 
\textit{Algebraische Zahlentheorie. Grundlehren der 
mathematischen Wissenschaften.} {\bf 322} Berlin: Springer-Verlag, 1999


%\bibitem{Pa1} {\sc A.\,N.\,Parshin}, 
%\textit{Class fields and algebraic $K$-theory 
%(Russian)}, 
%Uspekhi Mat. Nauk {\bf 30}\,(1975),  
%253-254 

%\bibitem{Pa2} {\sc A.\,N.\,Parshin}, 
%\textit{Local class field theory. (Russian)}, 
%Algebraic Geometry and its applications, 
%Trudy Mat. Inst. Steklov {\bf165}\, (1985),  
%143-170; English translation in: 
%Proc. Steklov Math. Inst. {\bf 165} (1985), issue 3, 157-185.

%\bibitem{Pa3} {\sc A.\,N.\,Parshin}, 
%\textit{Galois cohomology and Brauer group 
%of local fields. (Russian)}, 
%Trudy Mat. Inst. Steklov {\bf 4}\, (1990),  
%143-170; English translation in: 
%Proc. Steklov Inst. Math. {\bf 4} (1991), issue 4, 191-201.

%\bibitem{Sch} {\sc A.\,Scholl}, \textit{Higher fields of norms and 
%(phi, Gamma)-modules}, Documenta Mathematica, Extra Volume: John H. Coates' 
%Sixties Birthday (2006), 685-709

\bibitem{Se1} {\sc J.-P.Serre}, \textit {Local Fields},  
Berlin, New York: Springer-Verlag, 1959 

%\bibitem{Se2} {\sc J.-P.Serre}, \textit{Cohomologie Galoisienne}, 
%Springer Verlag, Berlin-Gottingen-Heidelberg-New York, 1964

\bibitem{Se3} {\sc J.-P.Serre}  
\textit{Structure de certains pro-p-groupes (d'apr\' es Demushkin)},  
{S\' eminaire Bourbaki}, {\bf 8}, Exp. No. 252, 145–155, 
Soc. Math. France, Paris, 1995.

\bibitem{Sh} {\sc I.R. Shafarevich}, 
\textit{On $p$-extensions}, {Mat.\,Sbornik} (1947) 
{\bf 20},  351-363   



\bibitem{Sh1} {\sc I.R.Shafarevich}, \textit{A general reciprocity law.} 
{Mat.\,Sb.} New Series, {\bf 26} (1950), 113–146 




%\bibitem{dSh} {\sc E.\,de Shalit}  (1986), 
%\textit{The explicit reciprocity law in local class field theory},  
%{Duke Math.\,J.} {\bf 53} (1986), 163–176


\bibitem{Vo} {\sc S.V.\,Vostokov} (1978), \textit{An explicit form of the reciprocity law.} 
{Izv. Akad. Nauk SSSR, Ser. Mat.}, {\bf 42} (1978), 1288–1321 



%\bibitem{Vo} {\sc S.\,V.\,Vostokov}, 
%\textit{Explicit construction  of class field theory 
%for a multidimensional local field}. 
%Izv.\,Akad.\,Nauk\,SSSR Ser.\,Mat.\,(1985) 
%no.\,2; English translation 
%in Math.\,USSR\,Izv.\,{\bf 26}\,(1986), \,263-288

%\bibitem{Wb} {\sc K.\,Wingberg}, 
%\textit{Galois groups of Poincare type}, 
%\textit{Galois groups of Poincar\'e type 
%over algebraic number fields}, 
%in: Galois groups over $\mathbb Q$, ed. by 
%Y.Ihara, K.Ribet, J.-P.Serre, 
%MSRI Publ. Springer (1989), 439-449


%\bibitem{Wi} {\sc 
%A.\,Wiles} (1978). \textit{Higher explicit reciprocity laws},  
%{Annals of Mathematics} {\bf 107} (1978), 235–254 


\bibitem{Wtb1} {\sc J.-P. Wintenberger}, 
\textit{Le corps des normes de certaines extensions
infinies des corps locaux; application}.  
{Ann. Sci. Ec. Norm. Super.,
IV. Ser}, {\bf16} (1983), 59--89

%\bibitem{Wtb2} {\sc J.-P.Wintenberger}, \textit {Extensions de 
%Lie et groupes d'automorphismes des 
%corps locaux de caract\' eristique p.} (French).  
%{C. R. Acad. Sci. Paris} S\'er. {\bf A-B 288} (1979), 
%no. 9, A477--A479


%\bibitem{Zh1} {\sc I.\,Zhukov}, 

%\textit{Higher dimensional local fields} 
%(M\"unster, 1999), Geom. Topol. Monogr., (2000), no.3, 5-18

%\bibitem{Zh2} {\sc I.\,B.\,Zhukov}, 
%\textit{Milnor and topological K-groups of multidimensional complete fields},   
%{Algebra i analiz}, (1997); English translation in 
%St.\,Petersburg Math.\,J.\,{\bf 9}\,(1998), 69-105.

%\bibitem{Zh3} {\sc I.\, Zhukov}, \textit{On ramification theory 
%in the case of an imperfect residue field.} 
%{Sbornik: Mathematics},  {\bf 194} (2003), %no.\,12, \bibitem
%1747-1774


%\bibitem{Zi} {\sc W.\,Zink} 
%\textit{Ramification in local Galois groups; the second 
%central step}, { Pure Appl. Math. Q.} {\bf 5} (2009), no. 1, 295--338.


%\bibitem{Ha} {\sc M.\,Hall} {The theory of groups}, 
%The Macmillan Company New York, 1959



%\bibitem{Kh} {\sc E.I. Khukhro}, 
%\textit{p-automorphisms of finite p-groups.} London Mathematical 
%Society Lecture Note Series, {\bf 246}. 
%Cambridge University Press, Cambridge, 1998. xviii+204 pp.%\bibitem{Fe} {\sc I.\,Fesenko} {Invitation}



\end{thebibliography}
\end{document}